\documentclass[11pt]{article}
\usepackage{amsmath,amssymb,tikz-cd}
\usepackage{mathrsfs}
\usepackage{mathrsfs}
\usepackage{mathrsfs}
\usepackage{mathtools}
\usepackage{amsfonts}
\usepackage{amssymb}
\usepackage{amsmath}
\usepackage{amsthm}
\usepackage{color}
\usepackage{enumitem}
\usepackage{bm}
\theoremstyle{plain}

\newtheorem{thm}{Theorem}[section]
\newtheorem{cor}[thm]{Corollary}
\newtheorem{lem}[thm]{Lemma}
\newtheorem{prop}[thm]{Proposition}
\newtheorem{exam}[thm]{Example}
\newtheorem{rem}[thm]{Remark}
\newtheorem{defn}[thm]{Definition}

\begin{document}
\title{{\Large\bf {Taylor-type joint spectrum and essential spectrum of row-bounded sequences of commuting operators\uppercase\expandafter{}
}}
\thanks{
This work is supported by National Natural Science Foundation of China: (12271298 and 11871308).}}
\author{{\normalsize Penghui Wang, Fan Zhang,  Ruoyu Zhang and Zeyou Zhu}} \date{}
 \maketitle
 
 \begin{abstract}
 As the cornerstone of multivariable operator theory, (essential) Taylor's joint spectrum plays crucial roles in many areas. The present paper is devoted to developing the theory of (essential) Taylor-type joint spectrum of row-bounded sequences of commuting operators on Hilbert spaces and to establishing the associated  Fredholm index theory. Furthermore, the holomorphic functional calculus for row-bounded sequences is established.  At last, we obtain a Gleason-Richter-Sundberg type index theorem in the infinite-variable setting, which connects the Fredholm index, the asymptotic Arveson's curvature invariant and fibre dimensions. 
 \end{abstract}
\


\maketitle
\numberwithin{equation}{section}
\newtheorem{theorem}{Theorem}[section]
\newtheorem{lemma}[theorem]{Lemma}
\newtheorem{proposition}[theorem]{Proposition}
\newtheorem{corollary}[theorem]{Corollary}
\newtheorem{claim}{Claim}
\def\si{\sigma_{\text{I}}}
\def\sii{\sigma_{\text{II}}}
\def\dT{\partial_{\mathbb T}}

\section{Introduction}

Let $\mathbb T_d=(T_1,\ldots,T_d)$ be a $d$-tuple of commuting bounded operators on a nonzero Hilbert space $H$. As the cornerstone of multivariable operator theory, the joint spectrum $\sigma(\mathbb T_d)$ was introduced by Taylor \cite{Taylorspectrum} via the exactness of the Koszul complex. It plays a very important role in many areas of mathematics;  see
\cite{ CR, CurtoYan1991, EschmeierPutinar1996,  Helemskii1981, Mller2003SpectralTO, FV, Vasilescu1982,  taylor} and references therein.

In the present paper, we introduce the Taylor-type joint spectra of row-bounded sequences of commuting bounded operators $\mathbb T=(T_1,T_2,\ldots)$ on $H$, where a sequence of bounded operators is called row-bounded if there is a constant $C>0$ such that 
$$
\sum\limits_{i=1}^k T_iT_i^*\leq C I, \quad \forall k\in\mathbb N.
$$
Moreover, if $C=1$, then the sequence  $\mathbb T$ is called a row-contraction. To define the Taylor-type joint spectrum, we need the Koszul complex of infinite length.  Let $\Lambda^n$ be the antisymmetric  $n$-fold tensor product of $\ell^2$ with itself for $n\in\mathbb N$, and $\Lambda^0=\mathbb C$ be the vacuum space. The antisymmetric Fock space is defined by $$\Lambda=\bigoplus\limits_{n=0}^\infty \Lambda^n.$$ For simplicity, set $\Lambda(H)=H\otimes_{\mathbb C}\Lambda$ and $\Lambda^n(H)=H\otimes_{\mathbb C}\Lambda^n$. Obviously, 
$$
\Lambda(H)=H \otimes_{\mathbb{C}} \Lambda \cong \mathop{\oplus}\limits_{n=0}^{\infty} H \otimes_{\mathbb{C}} \Lambda^{n}=\mathop{\oplus}\limits_{n=0}^{\infty} \Lambda^{n}(H).
$$
Let $\{e_i\}_{i=1}^\infty$ be an orthonormal basis of $\ell^2$.
For $i\in\mathbb N$, the creation operator  $E_i:\Lambda\to\Lambda$ is defined  by
$$
E_i\xi=e_i\wedge \xi, \quad \forall \xi\in\Lambda.
$$
Then it can be verified that 
\begin{eqnarray}\label{anticommutationrelations}
E_iE_j+E_jE_i=0,\quad E_i^*E_j+E_jE_i^*=\delta_{ij}I_{\Lambda}.
\end{eqnarray}
Furthermore obviously for $k\in\mathbb N$, $E_{i}(\Lambda^{k-1})\subset \Lambda^{k}.$ For convenience, write 
$$E_{i,k}=E_{i}|_{\Lambda^{k-1}}: \Lambda^{k-1}\to \Lambda^k,$$  and hence 
$$
\left\{\sum\limits_{i=1}^n T_i\otimes E_{i,k}^*\right\}_{n=1}^\infty\subset B\left(\Lambda^{k}(H),\Lambda^{k-1}(H)\right).
$$
As a prerequisite, in Section 2,  we prove that for any $k\in\mathbb N$, the sequence
$
\left\{\sum\limits_{i=1}^n T_i\otimes E_{i,k}^*\right\}_{n=1}^\infty
$
  is convergent in strong operator topology(SOT for short) and the limit is denoted by 
 \begin{eqnarray}
  \partial_{\mathbb T,k} =\operatorname{SOT-}\lim\limits_{n\to\infty}\sum\limits_{i=1}^{n} T_{i} \otimes E_{i,k}^{*}. 
 \end{eqnarray}
 Obviously, $\partial_{\mathbb T,0}=0$. \emph{In what follows, we always assume that $\mathbb T$ is a row-bounded sequence of commuting operators on $H$.} 
 Similar to \cite{taylor}, we have the following Koszul complex of infinite length associated with $\mathbb T$
\begin{eqnarray}\label{Koszul-complex}
K(\mathbb{T}): 0 \stackrel{\partial_{\mathbb{T}, 0}}{\longleftarrow} \Lambda^{0}(H) \stackrel{\partial_{\mathbb{T}, 1}}{\longleftarrow} \Lambda^{1}(H) \stackrel{\partial_{\mathbb{T}, 2}}{\longleftarrow} \Lambda^{2}(H)\longleftarrow \cdots .
\end{eqnarray}
As usual, the $k$-th homology group is defined by
$$
H^k_{\mathbb T}=\ker\partial_{\mathbb T,k}/\operatorname{Im}\partial_{\mathbb T,k+1}.
$$

With these preparations, we define the Taylor-type joint invertibility for commuting row-bounded sequence as follows. 

\begin{defn}\label{1and2invert}
Let $\mathbb T$ be a row-bounded sequence of commuting operators on $H$. $\mathbb T$ is called \emph{invertible} if 
\begin{itemize}
\item[(1)] $H^n_\mathbb T=0$ for any $n\geq0$, and
\item[(2)] there is a positive constant $c$, which is independent of $n$, such that for any $x\in\ker\partial_{\mathbb T,n}$ there exists $y\in\Lambda^{n+1}(H)$, which satisfies 
    $$
    x=\partial_{\mathbb T,n+1}y\quad{\text{and}}\quad\|y\|\leq c\|x\|.
    $$
\end{itemize}
\end{defn}

It should be pointed out that the uniformly controllable condition 2) in the above definition plays a central role. Next, we introduce the spectra for commuting row-bounded sequences.
\begin{defn}
The \emph{spectrum} for $\mathbb T$ is defined by
\begin{equation}\label{I-spectrum}
  \sigma(\mathbb T)=\{(\lambda_1,\lambda_2,\ldots)\in l^2; (\lambda_1 I-T_1,\lambda_2 I-T_2,\ldots) \text{ is not invertible}\}
\end{equation}
\end{defn}
For a single operator \(T\) on a Hilbert space \(H\) and a connected open set \(\Omega\subset\mathbb C\), the Cowen--Douglas class \(B_n(\Omega)\) \cite{CowenDouglas1978} consists of those \(T\) for which
\begin{itemize}
\item[(1)]$\operatorname{ran}(T-\omega I)=H ,\quad \forall\,\omega\in\Omega$,
\item[(2)]$\dim\ker(T-\omega I)=n,\quad \forall\,\omega\in\Omega$,
\item[(3)]$\overline{\operatorname{span}}\{\ker(T-\omega I):\,\omega\in\Omega\}=H.$
\end{itemize}
\noindent Cowen--Douglas theory shows that the eigenspaces \(\ker(T-\omega I)\) form a Hermitian holomorphic vector bundle over \(\Omega\), and that unitary equivalence corresponds to bundle equivalence; for \(n=1\), the curvature is a complete unitary invariant. Curto--Salinas\cite{CurtoSalinas1984} extended this viewpoint to commuting tuples via generalized Bergman kernels. Thus spectral theory encodes not only spectral points but also the local fibres and their geometry. Our definition of the spectrum for infinite commuting tuples is designed to capture this geometric information through the full infinite Koszul complex.

Słodkowski and Żelazko \cite[Theorem 2.3]{SZ1974} defined the following spectrum for a row-bounded sequence \(\mathbb T=(T_1,T_2,\ldots)\) of commuting operators on a Hilbert space \(H\):
\begin{equation}
\sigma_{SZ}(\mathbb T)
=
\left\{\lambda\in\mathbb C^\infty:
(\lambda_1,\cdots,\lambda_n)\in\sigma(T_1,\cdots,T_n)
\ \text{for all }n\ge 1\right\}.
\end{equation}
For a row-bounded sequence $\mathbb{T}$, by \cite[Theorem 2.3]{SZ1974}, it is easy to prove that $\sigma_{SZ}(\mathbb{T})\subset\ell^2$ is a nonempty compact subset of \((\ell^2,\tau_p)\), where  $(\ell^2,\tau_p)$ is considered as a topological subspace of $\mathbb{C}^\infty$ with product topology.
However, the case for $\sigma(\mathbb{T})$  is more complicated. In fact, we have the following result.

\begin{thm}\label{main-thm1} 
Let $\mathbb T$ be a row-bounded sequence of commuting operators on $H$. Then $\sigma(\mathbb T)$ is a bounded closed subset of $\ell^2$. Furthermore, for $1\le p\le 2$, if
\begin{eqnarray}\label{p-summability}
\sum\limits_{i=1}^\infty \|T_i\|^p<\infty,
\end{eqnarray}
then $\sigma(\mathbb T)$ is a nonempty compact subset of $\ell^p$.
\end{thm}
\begin{rem}\label{rem:nonempty}
\begin{itemize}
\item[(1)]Without the \(p\)-summability condition (\ref{p-summability}) in Theorem \ref{main-thm1}, the spectrum \(\sigma(\mathbb T)\) may be empty. An example of a row-bounded sequence \(\mathbb T\) with \(\sigma(\mathbb T)=\varnothing\) is constructed in Example \ref{empty spec}.
\item[(2)] In Remark 3.12. 1), we construct an example such that $\mathbb T$ is invertible, however for any $n\geq 0$, $(T_1,\cdots, T_n)$ is not invertible. That is, $0\not\in\sigma(\mathbb T)$ and $0\in\sigma_{SZ}(\mathbb T)$. 
\item[(3)] In Remark 3.12. 2), we show that for any nonempty bounded closed subset $F\subset\ell^2$, there is a row-bounded sequence $\mathbb T$, such that $\sigma(\mathbb T)=F$. This is not the case for $\sigma_{SZ}(\mathbb T)$, because $\pi_n\sigma_{SZ}(\mathbb T)=\sigma(\mathbb T_n)$, where $\pi_n$ is the projection onto the first $n$-coordinates. For example, let $\{e_n\}$ be the standard orthonormal basis of $\ell^2$, $q_{n}={1\over n} e_1+e_n$, then $S=\{q_n\}_{n=1}^\infty$  is bounded and closed. However $\pi_1(S)$ is not closed, therefore it is not the spectra of any operator.
\item[(4)]The  $p$-summability condition (\ref{p-summability}) can be replaced by the assumption that \(\left\{\sum\limits_{i=1}^{k}T_i\otimes E_i^*\right\}_{k\geq1}\) converges in norm and in this case, $\sigma(\mathbb T)=\sigma_{SZ}(\mathbb T)$, which is a nonempty compact subset of $\ell^2$, in the norm topology.
\end{itemize}
\end{rem}

 In Section 7, we calculate the spectrum of the coordinate sequence $(M_{z_1}, M_{z_2},\cdots)$ on the Drury-Arveson space $H^2$ of infinitely many variables, which  is the reproducing kernel Hilbert space of holomorphic functions on the open unit ball \(\mathbb B\) of \(\ell^2\) with the reproducing kernel 
$$
K_\lambda(z)={1\over 1-\langle z,\lambda\rangle}, \text{ for }\lambda,z\in\mathbb B.
$$ 
We have that $$\sigma(M_{z_1},M_{z_2},\cdots)=\bar{\mathbb B},$$ which is a (noncompact) bounded closed subset of $\ell^2$.

We now turn to the  functional calculus of $\mathbb T$ by holomorphic functions in ${\rm Hol}(\sigma(\mathbb T))$, where ${\rm Hol}(\sigma(\mathbb T))$ is the algebra of germs of
holomorphic functions on neighborhoods of $\sigma(\mathbb T)$. One studies the holomorphic functional calculus of an 
$n$-tuple because single-operator theory is insufficient for multivariate systems. Moreover, the Spectral Mapping Theorem \cite{taylor} coming from holomorphic functional calculus makes the spectrum more computable. For infinite-variable operator theory, holomorphic functional calculus of a row-bounded sequence of commuting operators should be one of the most important tasks. 

In what follows, for $d>1$, $\mathbb T_d=(T_1,\cdots,T_d)$ is always denoted to be a $d$-tuple of commuting operators. As usual, for a compact subset $K\subset\mathbb C^d$,  $[f]_d\in\mathrm{Hol}(K)$ means that there is a neighbourhood $U$ of $K$ such that $f$ is holomorphic on $U$. Recall that the classical holomorphic functional calculus of $\mathbb T_n$ is a unital  homomorphism  $\tau_n: {\rm Hol}(\sigma(
\mathbb T_n))\to B(H)$ satisfying 
				\begin{enumerate}
					\item[\textup{(1)}]
					For every $1\leq j\leq n$,
					$\tau_n([z_j]_n)=T_j$;
					
					\item[\textup{(2)}]
					For every open neighborhood $U$
					of $\sigma(\mathbb T_n)$ in $\mathbb C^n$, the map
					$\operatorname{Hol}(U)\to B(H)$,
					$f\mapsto\tau_n([f]_n)$, is continuous from the
					compact-open topology to the operator norm topology.
					
				\end{enumerate} 

The holomorphic functional calculus can be used to study von Neumann inequality. Precisely, let $\mathbb T_d$ be a $d$-contraction, then there exists $r>1$ such that
	$\sigma(\mathbb T_d)\subseteq r\overline{\mathbb B_d}$
	and for any $f\in\operatorname{Hol}(r\overline{\mathbb B_d})$,
	\[
	\|f(\mathbb T_d)\|
	\le \sup_{z\in r\overline{\mathbb B_d}}|f(z)|.
	\]
A homomorphism $\tau:\operatorname{Hol}(r\overline{\mathbb B_d})\to B(H)$ can be seen as the functional calculus of $\mathbb T_d$ on $r\overline{\mathbb B_d}$.
In general, if one replaces $\sigma(\mathbb T_n)$ by some compact set $K\subset\mathbb C^n$, then a unital homomorphism $\tau_n: {\rm Hol}(K)\to B(H)$ satisfying the above properties can be seen as the holomorphic functional calculus of $\mathbb T_n$ on $K$. For convenience, an element in ${\rm Hol}(K)$ will be written as $[g]_n$. 
For each $n\geq1$, let $\pi_n:\ell^2\to\mathbb C^n$ be the coordinate projection.

\begin{defn}\label{def:tau} Let $\mathbb T$ be a row-bounded sequence of commuting operators, $K\subset \ell^2$ be a compact set. 
		The family of holomorphic functional calculi $\{\tau_n\}_{n\geq1}$ is called \emph{compatible}
		if, for every $m\geq n\geq1$ and every
		$[g]_n\in\operatorname{Hol}(\pi_nK)$,
		\[
		\tau_m([g\circ\pi_{m,n}]_m)=\tau_n([g]_n),
		\]
		where $\pi_{m,n}:\mathbb C^m\to\mathbb C^n$
		is the coordinate projection.
		We say that \emph{$\mathbb T$ is compatible with respect to $\{\pi_nK\}_{n\geq1}$} if for each $n\geq1$, $(T_1,\ldots,T_n)$	admits a  holomorphic functional calculus	$\tau_n$ on $\pi_nK$ and $\{\tau_n\}_{n\geq1}$ is compatible.
\end{defn}
 
Next, we need the concept of compact-set von Neumann inequality. 
In what follows, \emph{let $\mathbb C[z_1,\ldots,z_n]$ denote the ring of polynomials in $z_1,\ldots,z_n$ with complex coefficients,
and let \begin{eqnarray}\mathcal P=\bigcup\limits_{n\ge1}\mathbb C[z_1,\ldots,z_n]\end{eqnarray} be the polynomial ring in infinitely many variables.}

\begin{defn}A row-bounded sequence of commuting operators $\mathbb T$ is called to satisfy the \emph{compact-set von Neumann inequality} if there exists a nonempty compact subset $\mathcal K$ of \(\ell^2\) such that for any $p\in \cal P$, 
\[
\|p(\mathbb T)\|
\leq \sup_{z\in \mathcal K}|p(z)|.
\]
\end{defn}

The compact-set von Neumann inequality holds for every finite tuple of commuting bounded operators $\mathbb T_d$, with $\mathcal K$ chosen to be a sufficiently large compact polydisc containing
$\sigma(\mathbb T_d)$.
However, for the case of infinite-variable operator theory,  it is much more complicated. The next proposition will be proved in Section 4. 
\begin{prop}\label{Prop:vNinequtocompactness}
Let $\mathbb T$ be a row-bounded sequence of commuting operators. If $\mathbb T$ satisfies the compact-set von Neumann inequality, then both  $\sigma(\mathbb T)$ and $\sigma_{SZ}(\mathbb T)$ are compact subsets in $\ell^2$.
\end{prop}

\begin{rem}
A simple sufficient condition for the compact-set von Neumann inequality is the norm summability \(\sum\limits_{i=1}^\infty \|T_i\| < \infty\). We refer to Section 4 for the proof.
\end{rem}
Now, we introduce the holomorphic functional calculus for row-bounded sequence of commuting operators. Recall that the compact-open topology on \(\operatorname{Hol}(U)\) is the topology of uniform convergence on compact subsets, generated by the seminorms \(p_K(f) = \sup\limits_{z \in K} |f(z)|\) for compact \(K \subset U\).

\begin{defn}\label{def:Gamma}
	Let $K$ be a subset of $\ell^2$.
	We say that $\mathbb T$ admits a \emph{continuous holomorphic functional calculus on $K$}  if there exists a unital algebra homomorphism $\Gamma:\operatorname{Hol}(K)\to B(H)$ satisfying the following conditions
	\begin{enumerate}
		\item[(1)]
		For every $j\geq1$, $\Gamma([z_j])=T_j$;
		
		\item[(2)]
		For every open neighborhood $U$ of $K$ in $\ell^2$, the map	$\Gamma: \operatorname{Hol}(U)\to B(H)$, $f\mapsto\Gamma([f])$ is continuous from the compact-open topology to the operator norm topology.
	\end{enumerate}
\end{defn}

With the above preparations, we have the following theorem, which is the main result on holomorphic functional calculus in the infinite-variable operator theory settings. 
\begin{thm}\label{thm:weighted-polynomial-calculus}
	Let $\mathbb T$ be a row-bounded sequence of commuting operators,  $K$ be a compact subset in $\ell^2$. 
	Then the following statements are equivalent.
	\begin{enumerate}
	\item[(1)]
			$\mathbb T$ satisfies the compact-set von Neumann inequality,  and there is a compatible family of holomorphic functional calculi  $\tau_n: {\rm Hol}(\pi_nK)\to B(H)$ of $\mathbb T_n$ with respect to $\pi_nK$;
				\item[(2)]
		$\mathbb T$ admits a continuous holomorphic functional
		calculus on $K$.
	\end{enumerate}
\end{thm}

Set $K=\sigma(\mathbb T)$.
As explained in Remark~\ref{rem:spectrum-calculus},
it follows from Theorem~\ref{thm:weighted-polynomial-calculus} that $\mathbb T$ admits a continuous holomorphic functional calculus on $\sigma(\mathbb T)$ if and only if $\mathbb T$ satisfies the compact-set von Neumann inequality and there is a compatible family of holomorphic functional calculi  $\tau_n: {\rm Hol}(\pi_n\sigma(\mathbb T))\to B(H)$ of $\mathbb T_n$ with respect to $\pi_n\sigma(\mathbb T)$.

	By the definition of $\sigma_{SZ}(\mathbb T)$, it is not hard to see that the family of holomorphic functional calculi $\{\tau_n: {\rm Hol}(\pi_n(\sigma_{SZ}(T)))\to B(H)\}$ is automatically compatible, and we have
	\begin{thm}\label{1.12}
		Let $\mathbb T$ be a row-bounded sequence
		of commuting operators.
		Then the following conditions are equivalent:
		\begin{enumerate}
			\item[(1)]
			$\mathbb T$ satisfies the compact-set von Neumann
			inequality; 
			
			\item[(2)]
			$\mathbb T$ admits a unique continuous holomorphic functional calculus \(\Gamma\)
			on $\sigma_{SZ}(\mathbb T)$ such that, for every $n\geq1$ and
			$[g]_n\in\operatorname{Hol}(\sigma(T_1,\ldots,T_n))$, 
			\[
			\Gamma([g\circ\pi_n])=g(T_1,\ldots,T_n).
			\]
		\end{enumerate}
	\end{thm}

At last, we will introduce  essential Taylor-type joint spectra and the related Fredholm index theory. To do this, inspired by Curto\cite{CR}, we have the following definition of Fredholmness.  
\begin{defn}\label{def:Fredholm}
The tuple  $\mathbb T$ is Fredholm if there is an $N\in\mathbb N$ such that 
\begin{itemize}
\item[(1)] for $n\leq N$, $\dim H_{\mathbb T}^n<\infty$;
\item[(2)] for $n > N$, $\dim H_{\mathbb T}^n=0$;
\item[(3)] there is a positive constant $c$, such that for $n>N$,  for any $x\in\ker\partial_{\mathbb T,n}$ there exists $y\in\Lambda^{n+1}(H)$ satisfying 
    $$
    x=\partial_{\mathbb T,n+1}y\quad{\text{and}}\quad\|y\|\leq c\|x\|.
    $$
\end{itemize}
 The Taylor-type essential joint spectrum of $\mathbb T$ is defined  by
$$
\sigma_{e}(\mathbb T)=\left\{\lambda \in \ell^{2} \mid \mathbb T-\lambda \text{ is not Fredholm}\right\}. 
$$
Moreover, if $\mathbb T$ is Fredholm, then its Fredholm index is defined by
$$
\operatorname{ind}(\mathbb T)=\sum_{k=0}^{\infty}(-1)^k\dim H^k_{\mathbb T}.
$$
\end{defn}
\begin{rem}
For any row-bounded sequence of commuting operators $\mathbb T$, if there is an $m$ such that $(T_1,\cdots,T_m)$ is Fredholm, then $(T_1,\cdots,T_m,T_{m+1})$ can be deformed to $(T_1,\cdots, T_m, I)$. It follows that the so-called ``finite Taylor Fredholmness" does not carry any information for the index of $\mathbb T$. This is the reason why we define Fredholmness through the infinite Koszul complex. 
\end{rem}

In Section 5 we develop the Fredholm theory for row-bounded sequences. We prove that the Fredholm index is continuous, and the essential Taylor-type spectrum \(\sigma_e(\mathbb T)\) is a bounded closed subset of \(\sigma(\mathbb T)\). The following result shows that in many cases, the joint spectrum and essential  joint spectrum coincide.

\begin{thm}\label{thm3.1.2}
Let \(\mathbb{T}\) be a row-bounded sequence of commuting operators on \(H\) such that 
\(\left\{\sum\limits_{i=1}^{k}T_i\otimes E_i^*\right\}_{k\geq1}\) converges in norm. Then
\[
\sigma_e(\mathbb{T})=\sigma(\mathbb{T}),
\]
which is a nonempty compact subset of \(\ell^2\).
\end{thm}

We next turn to applications to the Drury–Arveson space $H^2$ and its quotient modules. For the natural sequence \((M_{z_1},M_{z_2},\cdots)\) of multiplication operators on $H^2$, it is known \cite{curv} that
\[
\sum_{i=1}^\infty M_{z_i}M_{z_i}^*=I-E_0,
\]
where \(E_0\) is the orthogonal projection from \(H^2\) onto the constants \(\mathbb C\). Hence \((M_{z_1},M_{z_2},\cdots)\) is a row-contraction.

Let \(\mathcal M\) be a submodule of \(H^2\otimes\mathbb C^N\), and let \(\mathbb S=(S_{z_1},S_{z_2},\cdots)\) be the compression of the sequence $(M_{z_1},M_{z_2},\cdots)$ on the quotient module \(H^2\otimes\mathbb C^N/\mathcal M\). Based on this framework, we obtain a Gleason-Richter-Sundberg type index formula, which connects the Fredholm index and  Arveson's curvature invariant.
It can be used to  compute \(\operatorname{ind}(\mathbb S-\lambda)\). To state the result, we first recall the curvature invariant for row contractions.

Let \(\mathbb T=(T_1,T_2,\cdots)\) be a row-contraction on a Hilbert space \(H\). The defect operator of \(\mathbb T\) is
\[
\Delta_{\mathbb T}
=
\left(I-\operatorname{SOT-}\lim_{k\to\infty}\sum_{i=1}^k T_iT_i^*\right)^{1/2}.
\]
For \(z=(z_1,z_2,\cdots)\in\ell^2\), set
\[
S(z)=\operatorname{SOT-}\lim_{m\to\infty}\sum_{i=1}^m \bar z_iT_i.
\]
Then for \(z\in\mathbb B\), \(I-S(z)\) is invertible. Define
\[
F(z)=\Delta_{\mathbb T}(I-S(z)^*)^{-1}(I-S(z))^{-1}\Delta_{\mathbb T}.
\]
 Suppose that $F(z)$ has finite rank for every $z\in\mathbb B$. 
For $m\ge1$, let $\sigma_m$ denote the normalized surface
measure on $\partial\mathbb B_m$.
For \(\sigma_m\)-a.e. \(z^{(m)}=(z_1,\cdots,z_m)\in\partial\mathbb B_m\), put
\[
\kappa_0^m(z^{(m)})
=
\lim_{r\to1^-}(1-r^2)\operatorname{trace}
F(rz_1,\cdots,rz_m,0,0,\cdots),
\]
and let
\[
\kappa_m(\mathbb T)
=
\int_{\partial\mathbb B_m}
\kappa_0^m(z^{(m)})\,d\sigma_m(z^{(m)}).
\]
The asymptotic curvature invariant of \(\mathbb T\) is defined by
\[
\kappa(\mathbb T)=\lim_{m\to\infty}\kappa_m(\mathbb T).
\]

For a submodule \(\mathcal M\) of \(H^2\otimes\mathbb C^N\), set
\[
\mathcal M_\lambda=\{f(\lambda):f\in\mathcal M\},\qquad \lambda\in\mathbb B.
\]
The fiber dimension of \(\mathcal M\) is
\[
\operatorname{fd}(\mathcal M)
=
\sup_{\lambda\in\mathbb B}\dim\mathcal M_\lambda,
\]
and the zero set \(Z(\mathcal M)\) is
\[
Z(\mathcal M)
=
\{\lambda\in\mathbb B:\dim\mathcal M_\lambda<\operatorname{fd}(\mathcal M)\}.
\]
We then have the following theorem.
\begin{thm}\label{7.1}
Let \(\mathcal{M}\) be a submodule of \(H^2\otimes \mathbb{C}^N\) and \(\mathbb{S} = (S_{z_1}, S_{z_2}, \dots)\). Then for every \(\lambda \in \mathbb{B} \setminus Z(\mathcal{M})\), \(\mathbb{S} - \lambda\) is Fredholm with
\[
H^0_{\mathbb{S} - \lambda} \cong \mathbb{C}^N / \mathcal{M}_\lambda\quad\text{and}\quad H^p_{\mathbb{S} - \lambda} = 0,~p \ge 1.
\]
In particular, \(\mathbb{B}\setminus Z(\mathcal{M}) \subset \mathbb{B} \setminus \sigma_e(\mathbb{S})\), and 
\[
\operatorname{ind}(\mathbb{S} - \lambda) = \kappa(\mathbb{S}) = N - \operatorname{fd}(\mathcal{M}).
\]
\end{thm}

In \cite{index of Invariant Subspace}, Gleason, Richter and Sundberg obtained such an index theorem, which connects the Fredholm index, Arveson's curvature invariant and fibre bundle dimension in the finite-variable setting, and Theorem \ref{7.1} is an analogue to the Gleason-Richter-Sundberg index theorem in the infinite-variable setting.

The proof of Theorem \ref{7.1} relies on Hartz's column-row property\cite{Hartz2023}, which yields a characterization of column boundedness in terms of the boundedness of \(\partial_{\mathbb T}\). 

This paper is organized as follows. Section 2 constructs the infinite Koszul complex for row-bounded commuting sequences and studies the associated boundary operators. Section 3 develops the basic spectral theory, including invertibility criteria and boundedness, closedness, nonemptiness and compactness of the Taylor-type joint spectrum. Section 4 establishes the holomorphic functional calculus under the compact-set von Neumann inequality and proves a spectral mapping theorem for \(\sigma_{SZ}\). Section 5 introduces Fredholmness via the infinite Koszul complex and develops the Fredholm index and essential spectrum theory. Section 6 proves the equivalence between column boundedness and boundedness of the Koszul boundary operator on normalized complete Pick spaces. Section 7 applies the theory to the Drury–Arveson space and its quotient modules, and proves the infinite-variable Gleason–Richter–Sundberg type index theorem.

\section{The Koszul complex induced by $\mathbb T$}

\begin{thm}\label{prop 1.1}
Let $\mathbb T$ be a row-bounded sequence of commuting operators on $H$. Let $C>0$ such that $\sum\limits_{i=1}^\infty T_iT_i^*\le CI$. Then for $t\geq 1$, $\left\{\sum\limits_{i=1}^{n} T_{i} \otimes E_{i,t}^{*}\right\}_{n \geq 1}$ converges in SOT. Write
$$
\partial_{\mathbb{T}, t}=\operatorname{SOT-}\lim\limits _{n \rightarrow \infty} \sum\limits_{i=1}^{n} T_{i} \otimes E_{i,t}^{*},
$$ 
then $\|\partial_{\mathbb T, t}\|\leq \sqrt{tC}$. Moreover, the sequence of adjoints $\left\{\sum\limits_{i=1}^{n} T_{i}^* \otimes E_{i,t}\right\}_{n \geq 1}$  converges in SOT to $\partial_{\mathbb T, t}^*$.
\end{thm}

\begin{proof}
Let $\{e_i\}_{i\geq 1}$ be an orthonormal basis of $\ell^2$.  For $t\ge1$, define $C_t:\Lambda^t(H)\to \ell^2\otimes H\otimes\Lambda^{t-1}$ by
$$
C_tx=\sum_{i=1}^{\infty} e_i\otimes(I_H\otimes E_{i,t}^*)x.
$$
Notice that $\sum\limits_{i=1}^{\infty}E_{i,t}E_{i,t}^*=tI_{\Lambda^t}$, then we have
\[
\|C_tx\|^2
=\sum_{i=1}^{\infty}\|(I_H\otimes E_{i,t}^*)x\|^2
=\left\langle\left(I_H\otimes\sum_{i=1}^{\infty}E_{i,t}E_{i,t}^*\right)x,x\right\rangle
=t\|x\|^2,
\]
hence $\|C_t\|=\sqrt t$. Define $R_T:\ell^2\otimes H\to H$ by 
$$
R_T(e_i\otimes h_i)= T_ih_i.
$$ 
Since
$$
\left\|R_T\sum_{i=1}^{n}(e_i\otimes h_i)\right\|^2=\left\|\sum_{i=1}^{n}T_ih_i\right\|^2\leq C\sum_{i=1}^{n}\|h_i\|^2= C\left\|\sum_{i=1}^{n}(e_i\otimes h_i)\right\|^2,
$$
we have $\|R_T\|\le\sqrt C$. Let $P_n$ be the projection onto $\operatorname{span}\{e_1,\cdots,e_n\}$. For $x\in \Lambda^t(H)$, we have
\begin{equation}\label{sum n Tie}
\begin{aligned}
(R_T\otimes I_{\Lambda^{t-1}})(P_n\otimes I_{H\otimes\Lambda^{t-1}})C_tx
=&(R_T\otimes I_{\Lambda^{t-1}})(P_n\otimes I_{H\otimes\Lambda^{t-1}})\sum_{i=1}^{\infty} e_i\otimes(I_H\otimes E_{i,t}^*)x\\
=&(R_T\otimes I_{\Lambda^{t-1}})\sum_{i=1}^{n} e_i\otimes(I_H\otimes E_{i,t}^*)x\\
=&\sum_{i=1}^{n}T_i\otimes E_{i,t}^*x.
\end{aligned}
\end{equation}
Since $P_n\to I_{\ell^2}$ in SOT, by \eqref{sum n Tie} the sequence
$\left\{\sum\limits_{i=1}^{n}T_i\otimes E_{i,t}^*\right\}_{n\ge1}$
converges in SOT, and
\[
\partial_{\mathbb T,t}=(R_T\otimes I_{\Lambda^{t-1}})C_t,
\]
with
\[
\|\partial_{\mathbb T,t}\|
\le \|R_T\|\,\|C_t\|
\le \sqrt{tC}.
\]
Taking adjoints in \eqref{sum n Tie} yields
\begin{equation}\label{adj sum n Tie}
\sum_{i=1}^{n}T_i^*\otimes E_{i,t}=\left(\sum_{i=1}^{n}T_i\otimes E_{i,t}^*\right)^* =
C_t^*(P_n\otimes I_{H\otimes\Lambda^{t-1}})(R_T^*\otimes I_{\Lambda^{t-1}}).
\end{equation}
Since $P_n\to I_{\ell^2}$ in SOT, $\left\{\sum\limits_{i=1}^{n} T_{i}^* \otimes E_{i,t}\right\}_{n \geq 1}$ converges in SOT to $C_t^*(R_T^*\otimes I_{\Lambda^{t-1}})=\partial_{\mathbb T,t}^*$.
\end{proof}

Recall that for positive integer $d$, the antisymmetric Fock space over $\mathbb C^d$ is defined by 
 $$\Lambda_d=\bigoplus\limits_{k=0}^d\Lambda_d^k,$$ where $\Lambda_d^k$ is the antisymmetric $k$-fold tensor product of $\mathbb C^d$ with itself and $\Lambda_d^0=\mathbb C$. Let $H$ be a Hilbert space, and 
 $$
 \Lambda_d(H)=H\otimes \Lambda_d=\bigoplus\limits_{k=0}^d H\otimes\Lambda_d^k.
 $$
For simplicity, write $\Lambda_d^k(H)=H\otimes\Lambda_d^k$. Obviously, $\Lambda_d(H)$ admits a $\mathbb Z_2-$graded decomposition
$$
\Lambda_d(H)=\Lambda^{-}_d(H)\oplus \Lambda_d^+(H),
$$
where 
$$
\Lambda_d^{-}(H)=\bigoplus\limits_{k~odd}\Lambda_d^{k}(H),\quad \Lambda_d^{+}(H)=\bigoplus\limits_{k~even}\Lambda_d^{k}(H).
$$
As introduced in \cite{CR}, define $\widehat{T^{(d)}}\in B(\Lambda_d^{-}(H), \Lambda_d^{+}(H))$ by 
$$
\widehat{T^{(d)}}=(\partial_{T^{(d)}}+\partial_{T^{(d)}}^*)P_{\Lambda_d^{-}(H)},
$$
where $\partial_{T^{(d)}}=\sum\limits_{i=1}^{d}T_i\otimes E_i^*\in B(H\otimes\Lambda_d)$.  By \cite{Taylorspectrum} and \cite[Theorem 1]{CR}, 
$$
T^{(d)}~\text{is invertible}~\Leftrightarrow~\widehat{T^{(d)}}~ \text{is invertible}~\Leftrightarrow~ \operatorname{ker}\partial_{T^{(d)}}=\operatorname{Im}\partial_{T^{(d)}},
$$  
Similarly, $\Lambda(H)$ admits a $\mathbb Z_2$-graded decomposition as follows
$$
\Lambda^{-}(H)=\bigoplus\limits_{n=0}^{\infty}\Lambda^{2n+1}(H)\quad\text{and}\quad \Lambda^{+}(H)=\bigoplus\limits_{n=0}^{\infty}\Lambda^{2n}(H).
$$
We define
\[
\Lambda_{\mathrm{fin}}(H)
:= \bigcup_{N=0}^\infty \left(\bigoplus_{p=0}^N \Lambda^p(H)\right),
\]
and let
\[
\Lambda_{\mathrm{fin}}^{\pm}(H) := \Lambda_{\mathrm{fin}}(H) \cap \Lambda^{\pm}(H).
\]
Inspired by the definition of $\widehat{T^{(d)}}$, we can define 
$\widehat{\mathbb T}: \mathrm{Dom}(\widehat{\mathbb T})(\subseteq \Lambda^{+}(H)) \rightarrow \Lambda^{-}(H)$ by
$$
\widehat{\mathbb T}
\left(x_{0}, x_{2}, x_{4}, \cdots\right)
=
\left(
\partial_{\mathbb T,1}^{*} x_{0}+\partial_{\mathbb T,2} x_{2},\ 
\partial_{\mathbb T,3}^{*} x_{2}+\partial_{\mathbb T,4} x_{4},\ 
\partial_{\mathbb T,5}^{*} x_{4}+\partial_{\mathbb T,6} x_{6},\ \cdots
\right),
$$
where the domain $\mathrm{Dom}(\widehat{\mathbb T})$ is defined by
$$
\operatorname{Dom}(\widehat{\mathbb T})
=
\left\{
\left(x_{0}, x_{2}, \cdots\right) \in \Lambda^{+}(H)
\ \middle|\ 
\sum_{k=0}^{\infty}
\left\|
\partial_{\mathbb T,2k+1}^{*} x_{2k}
+
\partial_{\mathbb T,2k+2} x_{2k+2}
\right\|^{2}
<\infty
\right\},
$$
By Theorem \ref{prop 1.1}, $\Lambda_{\mathrm{fin}}^{+}(H)\subseteq \operatorname{Dom}(\widehat{\mathbb T})$. Since $\Lambda_{\mathrm{fin}}^{+}(H)$ is dense in $\Lambda^{+}(H)$, $\widehat{\mathbb T}$ is a densely defined operator.
Similarly, define $\partial_{\mathbb T}: \mathrm{Dom}(\partial_{\mathbb T})(\subseteq \Lambda(H))\to \Lambda(H)$ by 
$$
\partial_{\mathbb T}(x_0,x_1,\cdots)=(\partial_{\mathbb T,1}x_1,\partial_{\mathbb T,2}x_2,\cdots).
$$ 
where
\begin{equation}\label{D partial T}
\mathrm{Dom}\left(\partial_{\mathbb{T}}\right)=\left\{\left(x_{0}, x_{1}, \cdots\right) \in \Lambda(H)\, \left| \,\sum_{i=1}^{\infty}\left\|\partial_{\mathbb{T}, i} x_{i}\right\|^{2}<\infty\right.\right\}.
\end{equation}
By Theorem \ref{prop 1.1}, 
 $\Lambda_{\mathrm{fin}}(H) \subseteq \mathrm{Dom}\left(\partial_{\mathbb{T}}\right)$, and hence $\mathrm{Dom}\left(\partial_{\mathbb{T}}\right)$ is dense in $\Lambda(H)$.  Moreover, since  for all $n \geq 0$ $\partial_{\mathbb{T}, n} \partial_{\mathbb{T}, n+1}=0$, 
 $$
 \operatorname{Im}\left(\partial_{\mathbb{T}}\right) \subseteq \operatorname{ker}\left(\partial_{\mathbb{T}}\right).
 $$
 \begin{lem}\label{dom star}
 Let $\mathbb T$ be a row-bounded sequence of commuting operators on $H$. Then
\[
\mathrm{Dom}(\partial_{\mathbb T}^*)=\left\{y\in\Lambda(H)\,\middle|\,\sum_{i=1}^{\infty}\|\partial_{\mathbb T,i}^*y_{i-1}\|^2<\infty\right\}
\]
and
\[
\mathrm{Dom}(\widehat{\mathbb T}^*)
=
\left\{
y=(y_1,y_3,\dots)\in\Lambda^-(H)
\ \middle|\ 
\sum_{k=0}^{\infty}
\left\|
\partial_{\mathbb T,2k+1} y_{2k+1}
+
\partial_{\mathbb T,2k}^* y_{2k-1}
\right\|^{2}
<\infty
\right\}.
\]
Moreover,
$$
\partial_{\mathbb T}^*y=(0,\partial_{\mathbb T,1}^*y_0,\partial_{\mathbb T,2}^*y_1,\cdots),~\forall y\in \mathrm{Dom}(\partial_{\mathbb T}^*),
$$
and
$$
\widehat{\mathbb T}^*y
=
\left(
\partial_{\mathbb T,1} y_1,\ 
\partial_{\mathbb T,2}^* y_1+\partial_{\mathbb T,3}y_3,\ 
\partial_{\mathbb T,4}^* y_3+\partial_{\mathbb T,5}y_5,\ 
\cdots
\right),~\forall y\in \mathrm{Dom}(\widehat{\mathbb T}^*).
$$
 \end{lem}
 \begin{proof}
For \(y\in\mathrm{Dom}(\partial_{\mathbb T}^*)\), there exists $z\in \Lambda(H)$ such that
$$
\left\langle \partial_{\mathbb T}x, y\right\rangle=\left\langle x, z\right\rangle,~\forall x\in \mathrm{Dom}(\partial_{\mathbb T}).
$$
Hence for $i\geq 1$ and $x_i\in \Lambda^i(H)$, we have
$$
\left\langle x_i, \partial_{\mathbb{T},i}^*y_{i-1}\right\rangle=\left\langle \partial_{\mathbb{T},i}x_i, y_{i-1}\right\rangle=\left\langle \partial_{\mathbb T}x_i, y\right\rangle=\left\langle x_i, z\right\rangle=\left\langle x_i, z_i\right\rangle.
$$
Thus $z_i=\partial_{\mathbb{T},i}^*y_{i-1}$. It follows from $z\in \Lambda(H)$ that $\sum\limits_{i=1}^{\infty}\|\partial_{\mathbb T,i}^*y_{i-1}\|^2<\infty$. Conversely, suppose that $y\in \Lambda(H)$ satisfies $\sum\limits_{i=1}^{\infty}\|\partial_{\mathbb T,i}^*y_{i-1}\|^2<\infty$. Write
$$
z=(0,\partial_{\mathbb T,1}^*y_0,\partial_{\mathbb T,2}^*y_1,\cdots),
$$
Then $z\in \Lambda(H)$. Moreover, for any \(x\in\mathrm{Dom}(\partial_{\mathbb T})\), we have
$$
\left\langle x, z\right\rangle=\sum_{i=1}^{\infty}\langle x_i,\partial_{\mathbb T,i}^*y_{i-1}\rangle=\sum_{i=1}^{\infty}\langle \partial_{\mathbb T,i}x_i,y_{i-1}\rangle=\left\langle \partial_{\mathbb T}x, y\right\rangle.
$$
Thus \(y\in\mathrm{Dom}(\partial_{\mathbb T}^*)\), and 
$$
\partial_{\mathbb T}^*y=z=(0,\partial_{\mathbb T,1}^*y_0,\partial_{\mathbb T,2}^*y_1,\cdots).
$$
By the same reasoning, we have
\[
\mathrm{Dom}(\widehat{\mathbb T}^*)
=
\left\{
y=(y_1,y_3,\dots)\in\Lambda^-(H)
\ \middle|\ 
\sum_{k=0}^{\infty}
\left\|
\partial_{\mathbb T,2k+1} y_{2k+1}
+
\partial_{\mathbb T,2k}^* y_{2k-1}
\right\|^{2}
<\infty
\right\},
\]
with the convention \(y_{-1}=0\), and for \(y\in \mathrm{Dom}(\widehat{\mathbb T}^*)\),
\[
\widehat{\mathbb T}^*y
=
\left(
\partial_{\mathbb T,1} y_1,\ 
\partial_{\mathbb T,2}^* y_1+\partial_{\mathbb T,3}y_3,\ 
\partial_{\mathbb T,4}^* y_3+\partial_{\mathbb T,5}y_5,\ 
\cdots
\right).
\]
 \end{proof}

 \begin{lem}\label{lem:domain-hat-T}
 Let $\mathbb T$ be a row-bounded sequence of commuting operators on $H$. Then
\[
\operatorname{Dom}(\widehat{\mathbb T})
=
\operatorname{Dom}(\partial_{\mathbb T})
\cap
\operatorname{Dom}(\partial_{\mathbb T}^{*})
\cap
\Lambda^{+}(H),
\]
and
\[
\operatorname{Dom}(\widehat{\mathbb T}^*)
=
\operatorname{Dom}(\partial_{\mathbb T})
\cap
\operatorname{Dom}(\partial_{\mathbb T}^{*})
\cap
\Lambda^{-}(H).
\]
Moreover, for every \(x\in\operatorname{Dom}(\widehat{\mathbb T})\) and \(y\in\operatorname{Dom}(\widehat{\mathbb T}^*)\),
\begin{equation}\label{hat T x eq dx d*x}
\|\widehat{\mathbb T}x\|^{2}
=
\|\partial_{\mathbb T}x\|^{2}
+
\|\partial_{\mathbb T}^{*}x\|^{2},
\quad
\|\widehat{\mathbb T}^*y\|^{2}
=
\|\partial_{\mathbb T}y\|^{2}
+
\|\partial_{\mathbb T}^{*}y\|^{2}.
\end{equation}
\end{lem}
\begin{proof}
Since \(\partial_{\mathbb T,n}\partial_{\mathbb T,n+1}=0\) for all \(n\), we have \(\operatorname{Im}\partial_{\mathbb T,n+1} \perp \operatorname{Im}\partial_{\mathbb T,n}^*\). Thus, for any \(x_{2k} \in \Lambda^{2k}(H)\) and \(x_{2k+2} \in \Lambda^{2k+2}(H)\),
\begin{equation}\label{im d2k perp im d2k+1}
\left\| \partial_{\mathbb T,2k+1}^* x_{2k} + \partial_{\mathbb T,2k+2} x_{2k+2} \right\|^2
=
\| \partial_{\mathbb T,2k+1}^* x_{2k} \|^2 + \| \partial_{\mathbb T,2k+2} x_{2k+2} \|^2.
\end{equation}
Consequently, for \(x = (x_0, x_2, \dots) \in \Lambda^+(H)\), by Lemma \ref{dom star},
\[
x \in \operatorname{Dom}(\widehat{\mathbb T}) 
\iff \sum_{k=0}^\infty \| \partial_{\mathbb T,2k+1}^* x_{2k} \|^2 + \sum_{k=0}^\infty \| \partial_{\mathbb T,2k+2} x_{2k+2} \|^2 < \infty
\iff x \in \operatorname{Dom}(\partial_{\mathbb T}^*) \cap \operatorname{Dom}(\partial_{\mathbb T}).
\]
Moreover, for such \(x\), by \eqref{im d2k perp im d2k+1},
\[
\| \widehat{\mathbb T} x \|^2 = \sum_{k=0}^\infty \left\| \partial_{\mathbb T,2k+1}^* x_{2k} + \partial_{\mathbb T,2k+2} x_{2k+2} \right\|^2 = \| \partial_{\mathbb T}^* x \|^2 + \| \partial_{\mathbb T} x \|^2.
\]
For \(y \in \Lambda^-(H)\), write \(y = (y_1, y_3, \dots)\). Set \(y_{-1}=0\). Since \(\partial_{\mathbb T,2k}\partial_{\mathbb T,2k+1}=0\), we have \(\operatorname{Im}\partial_{\mathbb T,2k+1} \perp \operatorname{Im}\partial_{\mathbb T,2k}^*\). Thus
\begin{equation}\label{im d2k+1 perp im d2k}
\left\| \partial_{\mathbb T,2k+1} y_{2k+1} + \partial_{\mathbb T,2k}^* y_{2k-1} \right\|^2
=
\| \partial_{\mathbb T,2k+1} y_{2k+1} \|^2 + \| \partial_{\mathbb T,2k}^* y_{2k-1} \|^2.
\end{equation}
Consequently, for \(y \in \Lambda^-(H)\), by Lemma \ref{dom star},
\[
y \in \operatorname{Dom}(\widehat{\mathbb T}^*) 
\iff \sum_{k=0}^\infty \| \partial_{\mathbb T,2k}^* y_{2k-1} \|^2 + \sum_{k=0}^\infty \| \partial_{\mathbb T,2k+1} y_{2k+1} \|^2 < \infty
\iff y \in \operatorname{Dom}(\partial_{\mathbb T}^*) \cap \operatorname{Dom}(\partial_{\mathbb T}),
\]
and for such \(y\), by \eqref{im d2k+1 perp im d2k},
\[
\| \widehat{\mathbb T}^* y \|^2 = \sum_{k=0}^\infty  \left\| \partial_{\mathbb T,2k+1} y_{2k+1} + \partial_{\mathbb T,2k}^* y_{2k-1} \right\|^2=\| \partial_{\mathbb T} y \|^2 + \| \partial_{\mathbb T}^* y \|^2.
\]
\end{proof}
\begin{cor}\label{cor:hat-T-sum}
   Let $\mathbb T$ be a row-bounded sequence of commuting operators on $H$. Then
  \[
  \widehat{\mathbb T} = (\partial_{\mathbb T} + \partial_{\mathbb T}^*)\big|_{\operatorname{Dom}(\widehat{\mathbb T})}
  \quad\text{and}\quad
  \widehat{\mathbb T}^* = (\partial_{\mathbb T} + \partial_{\mathbb T}^*)\big|_{\operatorname{Dom}(\widehat{\mathbb T}^*)}.
  \]
\end{cor}
To continue, we need some fundamental properties on $\partial_{\mathbb T}$ whose proofs are elementary, and we list them here for reader's convenience.

\begin{lem}\label{sup=}
	Let \(\mathbb T =(T_1,T_2,\ldots)\) be a sequence of commuting operators.
	For each \(N \geq 1\), 
	\begin{align}\label{norm*}
		\left\| \sum_{n=1}^N T_n \otimes E_n^* \right\|=\left\| \sum_{n=1}^N T_n \otimes E_n \right\|.
	\end{align}
\end{lem}
\begin{proof}
	For each \(N\geq 1\), let \(U_N=(E_1+E_1^*)\cdots(E_N+E_N^*)\). 
	By \eqref{anticommutationrelations}, it is easy to see that \(U_N\) is a unitary operator and
	\begin{align*}
		(I_H \otimes U_N)\left(\sum_{n=1}^N T_n \otimes E_n^*\right)(I_H \otimes U_N^*)
		=
		(-1)^{N-1}\left(\sum_{n=1}^N T_n \otimes E_n\right).
	\end{align*}
	Therefore, \eqref{norm*} holds.
\end{proof}

\begin{lem}\label{bound and SOT}
	Let \(\mathbb T\) be a sequence of commuting operators satisfying
	\[
	\sup_{N \geq 1} \left\| \sum_{n=1}^N T_n \otimes E_n^*\right\|<\infty.
	\]
	Then \(\mathbb T\) is both row bounded and column bounded.
	Moreover, the series \(\sum\limits_{n=1}^\infty T_n \otimes E_n^*\) and \(\sum\limits_{n=1}^\infty T_n \otimes E_n\) converge in the strong operator topology.
\end{lem}
\begin{proof}
	Let
	\[
	L=\sup_{N \geq 1} \left\| \sum_{n=1}^N T_n \otimes E_n^*\right\|<\infty.
	\]
	For any \(\zeta=\sum\limits_{j=1}^N x_j \otimes e_j \in H \otimes \Lambda^1\), we have
	\begin{align*}
		\left\|\sum_{n=1}^N T_nx_n\right\|
		=\left\|\left( \sum_{n=1}^{N}T_n\otimes E_n^* \right) \zeta\right\| 
		\leq L\left( \sum_{n=1}^N\|x_n\|^2\right)^{\frac{1}{2}}.
	\end{align*}
	Therefore, \(\mathbb T\) is row-bounded. 
	For any \(x \in H\) and \(\omega_N= e_1 \wedge \cdots \wedge e_N\), we have
	\begin{align}\label{column bounded}
		\sum_{n=1}^N \|T_nx\|^2
		= \left\| \sum_{n=1}^N T_nx \otimes E_n^*\omega_N \right\|^2
		=\left\|\left( \sum_{n=1}^{N}T_n\otimes E_n^* \right)(x \otimes \omega_N)\right\|^2
		\leq L^2 \|x\|^2.
	\end{align}
	Therefore, \(\mathbb T\) is column-bounded.
	
	For each $N\geq 1$, let $P_N$ be the orthogonal projection of $\Lambda(H)$ onto \(H\otimes\Lambda\bigl[\mathbb C^N \bigr]\).
	For any $M>N$ and  $\xi\in\Lambda(H)$, since 
	\[
	\left(\sum\limits_{n=N+1}^{M}T_n\otimes E_n^*\right)P_N=0,
	\] 
	we have
	\[
	\left\|\left(\sum_{n=N+1}^{M}T_n\otimes E_n^*\right)\xi\right\|
	=\left\|\left(\sum_{n=N+1}^{M}T_n\otimes E_n^*\right)(I_{\Lambda(H)}-P_N)\xi\right\|
	\leq2L\,\|(I_{\Lambda(H)}-P_N)\xi\|.
	\]
	Hence,  $\left\{\left(\sum\limits_{n=1}^{N}T_n\otimes E_n^*\right)\xi\right\}_{N \geq 1}$ is a Cauchy sequence.
	Thus \(\sum\limits_{n=1}^\infty T_n \otimes E_n^*\) converges in the strong operator topology.

	Let \(\mathcal I=\{I\subset\mathbb N:|I|<\infty\}\). 
	For each \(I=\{i_1<\cdots<i_p\}\in\mathcal I\), let \(e_I=e_{i_1}\wedge\cdots\wedge e_{i_p}\), with \(e_{\emptyset}=1\).
	For any \(\xi  \in\Lambda(H)\), write \(\xi = \sum\limits_{I \in \mathcal I}x_I\otimes e_I\) with \(\sum\limits_{I \in \mathcal I}\|x_I\|^2 = \|\xi\|^2 <\infty\). 
	Then for \(M>N\), we have 
	\begin{equation}\label{leqTn}
		\begin{aligned}
			\left\|\left(\sum_{n=N+1}^{M}T_n\otimes E_n\right)P_N\xi\right\|^2
			&\leq
			\left\|\left(\sum_{n=N+1}^{M}T_n\otimes E_n\right)\left(\sum\limits_{I \subset \{1,\ldots,N\}}x_I\otimes e_I\right)\right\|^2\\
			&=
			\sum_{I\subset\{1,\ldots,N\}}
			\sum_{n=N+1}^{M}\|T_nx_I\|^2\\
			&\leq
			\sum_{I\in\mathcal I}
			\sum_{n=N+1}^{\infty}\|T_nx_I\|^2.
		\end{aligned}
	\end{equation}
	For any \(I \in \mathcal I\), by \eqref{column bounded}, we have
	\begin{align}\label{to0}
		\sum\limits_{n=N+1}^{\infty}\|T_nx_I\|^2 \to 0, \quad N \to \infty
	\end{align}
	and 
	\[
	\sum_{I\in\mathcal I}\sum_{n=N+1}^{\infty}\|T_nx_I\|^2 \leq L^2\sum_{I \in \mathcal I}\|x_I\|^2 =L^2\|\xi\|^2<\infty.
	\] 
	Hence, by the dominated convergence theorem for series, we have 
	\begin{align}\label{conTn}
		\lim\limits_{N \to \infty} \sum\limits_{I\in\mathcal I}\sum\limits_{n=N+1}^{\infty}\|T_nx_I\|^2 
		=
		\sum\limits_{I\in\mathcal I} \left(\lim\limits_{N \to \infty} \sum\limits_{n=N+1}^{\infty}\|T_nx_I\|^2 \right) \overset{(*)}{=}0,
	\end{align}
	where \((*)\) is given by \eqref{to0}.
	By \eqref{leqTn}, \eqref{conTn} and Lemma~\ref{sup=}, we have
	\begin{align*}
		\left\|\left(\sum_{n=N+1}^{M}T_n\otimes E_n\right)\xi\right\|
		&\leq 
		\left\|\left(\sum_{n=N+1}^{M}T_n\otimes E_n\right)(I_{\Lambda(H)}-P_N)\xi\right\|
		+\left\|\left(\sum_{n=N+1}^{M}T_n\otimes E_n\right)P_N\xi\right\|\\
		&\leq
		2L\left\|(I_{\Lambda(H)}-P_N)\xi\right\|
		+ 
		\left(\sum_{I\in\mathcal I}\sum_{n=N+1}^{\infty}\|T_nx_I\|^2\right)^{\frac12} \longrightarrow 0.
	\end{align*}
	Therefore, \(\left\{\left(\sum\limits_{n=1}^N T_n\otimes  E_n\right)\xi \right\}_{N \geq 1}\) is a Cauchy sequence. 
	Hence \(\sum\limits_{n=1}^\infty T_n\otimes  E_n \) converges
	in the strong operator topology.
\end{proof}

\begin{lem}\label{partial norm hatT}
	The operator  $\partial_{\mathbb T}$ is bounded if and only if so is $\widehat{\mathbb T}$; moreover, in this case $$\left\|\partial_{\mathbb T}\right\|=\left\|\widehat{\mathbb T}\right\|.$$
\end{lem}
\begin{proof}
Suppose that $\partial_{\mathbb T}$ is bounded, then by \eqref{hat T x eq dx d*x}, $\widehat{\mathbb T}$ is bounded. Then by Corollary \ref{cor:hat-T-sum} and $\partial_{\mathbb T}^2=0$,
\begin{equation}\label{hat T 2}
\left\|\widehat{\mathbb T}\right\|^2=\left\|\widehat{\mathbb T}\widehat{\mathbb T}^*\right\|\leq \|\partial_{\mathbb T}\partial_{\mathbb T}^*+\partial_{\mathbb T}^*\partial_{\mathbb T}\|.
\end{equation}
Since $\partial_{\mathbb T}\partial_{\mathbb T}^*$ and $\partial_{\mathbb T}^*\partial_{\mathbb T}$ are positive, and $\partial_{\mathbb T}\partial_{\mathbb T}^*\partial_{\mathbb T}^*\partial_{\mathbb T}=0$, we have
$$
\|\partial_{\mathbb T}\partial_{\mathbb T}^*+\partial_{\mathbb T}^*\partial_{\mathbb T}\|\leq \max\{\|\partial_{\mathbb T}\partial_{\mathbb T}^*\|,\|\partial_{\mathbb T}^*\partial_{\mathbb T}\|\}=\|\partial_{\mathbb T}\|^2.
$$
Thus by \eqref{hat T 2}, $\left\|\widehat{\mathbb T}\right\|\leq \|\partial_{\mathbb T}\|$.
Conversely, suppose that $\widehat{\mathbb{T}}$ is bounded, then $\widehat{\mathbb{T}}^*$ is also bounded. For $x\in \Lambda_{\mathrm{fin}}^+(H)$, by \eqref{hat T x eq dx d*x}, we have
$$
\|\partial_{\mathbb{T}}x\|^2\leq \|\widehat{\mathbb{T}}x\|^2,
$$
and for $y\in \Lambda_{\mathrm{fin}}^-(H)$,
$$ 
\|\partial_{\mathbb{T}}y\|^2\leq \|\widehat{\mathbb{T}}^*y\|^2.
$$
It follows from $\Lambda_{\mathrm{fin}}^-(H)\oplus\Lambda_{\mathrm{fin}}^+(H)$ is dense in $\Lambda(H)$ that $\partial_{\mathbb{T}}$ is bounded, and
$$
\|\partial_{\mathbb{T}}\|\leq \max\{\|\widehat{\mathbb{T}}\|,\|\widehat{\mathbb{T}}^*\|\}=\|\widehat{\mathbb{T}}\|.
$$
\end{proof}

\begin{lem}\label{partial closed}
 Both $\partial_{\mathbb{T}}$ and $\widehat{\mathbb{T}}$ are closed operators.
\end{lem}

\begin{proof}
Let
\[
        x^{(n)}=(x^{(n)}_0,x^{(n)}_1,\ldots)\in \mathrm{Dom}(\partial_T)
\]
satisfy
\[
        x^{(n)}\longrightarrow x=(x_0,x_1,\ldots)
        \quad\text{and}\quad
        \partial_{\mathbb{T}}x^{(n)}\longrightarrow
        y=(y_0,y_1,\ldots)
\]
in \(\Lambda(H)\). Then for each \(k\ge1\),
$
        x^{(n)}_k\longrightarrow x_k
$
in \(\Lambda^k(H)\), and
$
        \partial_{\mathbb{T},k}x^{(n)}_k
        \longrightarrow y_{k-1}
$
in \(\Lambda^{k-1}(H)\). By Theorem \ref{prop 1.1},
\[
        \partial_{\mathbb{T},k}x^{(n)}_k
        \longrightarrow
        \partial_{\mathbb{T},k}x_k.
\]
Hence for $k\geq 1$,
$
        y_{k-1}=\partial_{\mathbb{T},k}x_k
$.
Since \(y\in\Lambda(H)\),
\[
        \sum_{k=1}^{\infty}
        \|\partial_{\mathbb{T},k}x_k\|^2
        =
        \sum_{k=1}^{\infty}\|y_{k-1}\|^2
        =
        \|y\|^2<\infty.
\]
Thus \(x\in \mathrm{Dom}(\partial_\mathbb{T})\) and 
$$
        \partial_\mathbb{T}x=(\partial_{\mathbb{T},1}x_1, \partial_{\mathbb{T},2}x_2,\cdots)=(y_0,y_1,\cdots)=y.
$$
Therefore \(\partial_\mathbb{T}\) is closed. To prove that \(\widehat{\mathbb{T}}\) is also closed, let \(x^{(n)} \in \mathrm{Dom}(\widehat{\mathbb{T}})\) satisfy
\(
x^{(n)} \to x
\)
in \(\Lambda^+(H)\) and
\(
\widehat{\mathbb{T}}x^{(n)} \to y
\)
in \(\Lambda^-(H)\). By Lemma \ref{lem:domain-hat-T}, we have 
\[
\|\widehat{\mathbb{T}}(x^{(n)} - x^{(m)})\|^2 = \|\partial_{\mathbb{T}}(x^{(n)} - x^{(m)})\|^2 + \|\partial_{\mathbb{T}}^*(x^{(n)} - x^{(m)})\|^2.
\]
Thus \(\{\partial_{\mathbb{T}}x^{(n)}\}\) and \(\{\partial_{\mathbb{T}}^*x^{(n)}\}\) are Cauchy sequences. Let their limits be \(u\) and \(v\), respectively. Since \(\partial_{\mathbb{T}}\) and \(\partial_{\mathbb{T}}^*\) are closed, \(x \in \mathrm{Dom}(\partial_{\mathbb{T}}) \cap \mathrm{Dom}(\partial_{\mathbb{T}}^*)\) with \(\partial_{\mathbb{T}}x = u\) and \(\partial_{\mathbb{T}}^*x = v\). Hence by Lemma \ref{lem:domain-hat-T}, \(x \in \mathrm{Dom}(\widehat{\mathbb{T}})\). Then by Corollary \ref{cor:hat-T-sum},
$$
\widehat{\mathbb{T}}x=\partial_{\mathbb{T}}x+\partial_{\mathbb{T}}^*x=u+v=y,
$$
which implies that \(\widehat{\mathbb{T}}\) is closed.
\end{proof}

By the closed range theorem, $\operatorname{Im}\partial_{\mathbb T}$ is closed iff the following Property (CR) holds.

\noindent{\bf Property (CR)}\emph{ There exists a constant $c>0$ such that every $x\in\operatorname{Im}\partial_{\mathbb T}$ has $y\in \mathrm{Dom}(\partial_{\mathbb T})$ with $\partial_{\mathbb T}y=x$ and $\|y\|\le c\|x\|$. }

\begin{lem}\label{lem3.6}
  The following statements are equivalent.
\begin{itemize}\label{Imclosed}
  \item[(1)] $\operatorname{Im}\partial_{\mathbb{T}}$ is closed;
  \item[(2)] there exists a constant $c>0$ such that for each $k\geq 0$ and $x_{k}\in \operatorname{Im}\partial_{\mathbb{T},k+1}$, there is a $y_{k+1}\in \Lambda^{k+1}(H)$ such that $x_{k}=\partial_{\mathbb{T},k+1}y_{k+1}$ and $\|y_{k+1}\|\leq c\|x_{k}\|$;
  \item[(3)] there exists a constant $c>0$ such that for each $k\geq 1$ and $x_{k}\in \operatorname{Im}\partial_{\mathbb{T},k}^{*}$, there is a $y_{k-1}\in \Lambda^{k-1}(H)$ such that $x_{k}=\partial_{\mathbb{T},k}^{*}y_{k-1}$ and $\|y_{k-1}\|\leq c\|x_{k}\|$.
  \item[(4)] \(\operatorname{Im}\widehat{\mathbb T}\) is closed.
\end{itemize}
\end{lem}

\begin{proof}

$(1)\Rightarrow(2)$. Fix $k\ge0$ and $x_k\in\operatorname{Im}\partial_{\mathbb T,k+1}$. Let 
\begin{equation}\label{x def}
x = (0,0,\cdots,0,
    \mathord{\underbrace{x_k}_{\mathclap{\text{$k$-th}}}},
    0,\cdots)
\end{equation}
Then $x\in\operatorname{Im}\partial_{\mathbb T}$. By (CR), choose $y\in \mathrm{Dom}(\partial_{\mathbb T})$ with $\partial_{\mathbb T}y=x$ and $\|y\|\le c\|x\|$. For $k\geq 0$, Set $y_{k}=P_{\Lambda^{k}(H)}y$, then 
$$
x=\partial_{\mathbb{T}}y=(\partial_{\mathbb T,1}y_1,\partial_{\mathbb T,2}y_2,\cdots)
$$
By \eqref{x def}, $\partial_{\mathbb T,k+1}y_{k+1}=x_k$ and 
$$
\|y_{k+1}\|\le \|y\|\le c\|x\|=c\|x_k\|,
$$
Thus (2) holds.

$(2)\Rightarrow(1)$. Let $x=(x_0,x_1,\cdots)\in\operatorname{Im}\partial_{\mathbb T}$. For each $k$, choose $y_{k+1}$ as in (2). Then $y=(0,y_1,y_2,\ldots)\in \mathrm{Dom}(\partial_{\mathbb T})$, $\partial_{\mathbb T}y=x$, and $\|y\|\le c\|x\|$. Hence (CR) holds, so $\operatorname{Im}\partial_{\mathbb T}$ is closed.

Applying the reasoning used to prove $(1)\Leftrightarrow(2)$ to $\partial_{\mathbb T}^*$ shows (3) $\Leftrightarrow$ $\operatorname{Im}\partial_{\mathbb T}^*$ closed. By the closed range theorem, $\operatorname{Im}\partial_{\mathbb T}$ closed iff $\operatorname{Im}\partial_{\mathbb T}^*$ closed, so $(1)\Leftrightarrow(3)$.

$(1)\Rightarrow(4)$. Assume (1) holds. By the closed range theorem, \(\operatorname{Im}\partial_{\mathbb T}\) and \(\operatorname{Im}\partial_{\mathbb T}^*\) are closed, hence
\begin{equation}\label{Lambda+ frac}
\begin{aligned}
\Lambda^+(H)
=&(\operatorname{Im}\partial_{\mathbb T}\cap\Lambda^+)\oplus (\ker \partial_{\mathbb T}^*\cap\Lambda^+)\\
=&(\operatorname{Im}\partial_{\mathbb T}\cap\Lambda^+)\oplus ((\operatorname{Im}\partial_{\mathbb T}^*\oplus (\ker \partial_{\mathbb T}^*\cap (\operatorname{Im}\partial_{\mathbb T}^*)^\perp))\cap\Lambda^+)\\
=&(\operatorname{Im}\partial_{\mathbb T}\cap\Lambda^+)\oplus ((\operatorname{Im}\partial_{\mathbb T}^*\oplus (\ker \partial_{\mathbb T}^*\cap \ker\partial_{\mathbb T}))\cap\Lambda^+)\\
=&(\operatorname{Im}\partial_{\mathbb T}\cap\Lambda^+)
\oplus
(\operatorname{Im}\partial_{\mathbb T}^*\cap\Lambda^+)
\oplus
(\ker\partial_{\mathbb T}\cap\ker\partial_{\mathbb T}^*\cap\Lambda^+)\\
\overset{(*)}{=}&(\operatorname{Im}\partial_{\mathbb T}\cap\Lambda^+)
\oplus
(\operatorname{Im}\partial_{\mathbb T}^*\cap\Lambda^+)
\oplus \ker\widehat{\mathbb T}.
\end{aligned}
\end{equation}
where $(*)$ follows from \eqref{hat T x eq dx d*x}. Therefore for \(x\in\operatorname{Dom}(\widehat{\mathbb T})\ominus\ker\widehat{\mathbb T}\), we write \(x=a+b\) with \(a\in\operatorname{Im}\partial_{\mathbb T}\cap\Lambda^+\) and \(b\in\operatorname{Im}\partial_{\mathbb T}^*\cap\Lambda^+\). Since \(\partial_{\mathbb T}^2=0\), we have \(\partial_{\mathbb T}a=0\) and \(\partial_{\mathbb T}^*b=0\), hence
\begin{equation}\label{dTx}
\partial_{\mathbb T}x=\partial_{\mathbb T}b,
\quad
\partial_{\mathbb T}^*x=\partial_{\mathbb T}^*a.
\end{equation}
By (1) and the closed range theorem, there is \(c>0\) such that
\(\|\partial_{\mathbb T}b\|\ge c\|b\|\) for \(b\in\operatorname{Im}\partial_{\mathbb T}^*\cap\Lambda^+\), and
\(\|\partial_{\mathbb T}^*a\|\ge c\|a\|\) for \(a\in\operatorname{Im}\partial_{\mathbb T}\cap\Lambda^+\).
Therefore, by Lemma \ref{lem:domain-hat-T} and \eqref{dTx},
\[
\|\widehat{\mathbb T}x\|^2
=
\|\partial_{\mathbb T}x\|^2+\|\partial_{\mathbb T}^*x\|^2
=
\|\partial_{\mathbb T}b\|^2+\|\partial_{\mathbb T}^*a\|^2
\ge
c^2(\|a\|^2+\|b\|^2)
=
c^2\|x\|^2.
\]
Thus \(\operatorname{Im}\widehat{\mathbb T}\) is closed.

$(4)\Rightarrow(1)$. Assume (4) holds. Then \(\operatorname{Im}\widehat{\mathbb T}^*\) is also closed. Fix \(n\ge0\) and take 
$$
x\in\overline{\operatorname{Im}\partial_{\mathbb T,2n+1}}\subset\operatorname{Dom}(\widehat{\mathbb T}). 
$$
Then \(x\in\ker\partial_{\mathbb T,2n}\) and \(x\perp\ker\partial_{\mathbb T}^*\). By \eqref{hat T x eq dx d*x}, 
$$
\ker\widehat{\mathbb T}= \ker\partial_{\mathbb T}^*\cap \ker\partial_{\mathbb T}\cap\Lambda^+(H).
$$
Hence $x\perp\ker\widehat{\mathbb T}$. By the closed range theorem, there is \(c_1>0\) such that \(\|\widehat{\mathbb T}x\|\ge c_1\|x\|\). Hence
\begin{equation}\label{odd 3}
\|\partial_{\mathbb T,2n+1}^*x\|^2=\|\widehat{\mathbb T}x\|^2\ge c_1^2\|x\|^2,
\quad
x\in\overline{\operatorname{Im}\partial_{\mathbb T,2n+1}}.
\end{equation}
Next, take 
$$
y\in\overline{\operatorname{Im}\partial_{\mathbb T,2n+2}}\subset\operatorname{Dom}(\widehat{\mathbb T}^*).
$$
Then \(y\in\ker\partial_{\mathbb T,2n+1}\) and \(y\perp\ker\partial_{\mathbb T}^*\). By \eqref{hat T x eq dx d*x}, 
$$
\ker\widehat{\mathbb T}^*= \ker\partial_{\mathbb T}^*\cap \ker\partial_{\mathbb T}\cap\Lambda^-(H).
$$
Hence $y\perp\ker\widehat{\mathbb T}^*$. By the closed range theorem, there is \(c_2>0\) such that \(\|\widehat{\mathbb T}^*y\|\ge c_2\|y\|\). Hence
\begin{equation}\label{even 3}
\|\partial_{\mathbb T,2n+2}^*y\|^2=\|\widehat{\mathbb T}^*y\|^2\ge c_2^2\|y\|^2,
\quad
y\in\overline{\operatorname{Im}\partial_{\mathbb T,2n+2}},
\end{equation}
Let $c=\min\{c_1,c_2\}$. By \eqref{odd 3} and \eqref{even 3}, $\partial_{\mathbb T,n}\partial_{\mathbb T,n}^*|_{\overline{{\operatorname{Im}\partial_{\mathbb T,n}}}}$ is invertible for all $n\geq 1$, with
\begin{equation}\label{dd* on im}
\left\|(\partial_{\mathbb T,n}\partial_{\mathbb T,n}^*|_{\overline{{\operatorname{Im}\partial_{\mathbb T,n}}}})^{-1}\right\|\leq c^{-2}.
\end{equation}
For each $n\geq 0$ and $x_{n}\in \operatorname{Im}\partial_{\mathbb{T},n+1}$, take
$$
y_{n+1}=\partial_{\mathbb T,n+1}^*(\partial_{\mathbb T,n+1}\partial_{\mathbb T,n+1}^*|_{\overline{{\operatorname{Im}\partial_{\mathbb T,n+1}}}})^{-1}x_n.
$$
Then $\partial_{\mathbb T,n+1}y_{n+1}=x_n$, and
$$
\begin{aligned}
\|y_{n+1}\|^2
=&\left\langle \partial_{\mathbb T,n+1}^*(\partial_{\mathbb T,n+1}\partial_{\mathbb T,n+1}^*|_{\overline{{\operatorname{Im}\partial_{\mathbb T,n+1}}}})^{-1}x_n,\partial_{\mathbb T,n+1}^*(\partial_{\mathbb T,n+1}\partial_{\mathbb T,n+1}^*|_{\overline{{\operatorname{Im}\partial_{\mathbb T,n+1}}}})^{-1}x_n\right\rangle\\
=&\left\langle (\partial_{\mathbb T,n+1}\partial_{\mathbb T,n+1}^*|_{\overline{{\operatorname{Im}\partial_{\mathbb T,n+1}}}})^{-1}x_n, x_n\right\rangle\\
\leq &\left\|(\partial_{\mathbb T,n+1}\partial_{\mathbb T,n+1}^*|_{\overline{{\operatorname{Im}\partial_{\mathbb T,n+1}}}})^{-1}\right\|\|x_n\|^2\\
\leq &c^{-2}\|x_n\|^2.
\end{aligned}
$$
Hence (2) holds, and therefore (1) holds.
\end{proof}

Let $E$ be a nonzero Hilbert space. We will prove that, for most finitely generated submodules of the Drury-Arveson module $H^2\otimes E$, the operator $\partial_{\mathbb T}$ is unbounded. To prove this,  we need some preparations. Let $\mathrm{Mult}(H^2)$ denote the multiplier algebra of $H^2$. For a submodule $\mathcal{M}\subset H^{2}\otimes E$ and $\phi \in\mathrm{Mult}(H^2)$, define \(S_\phi\in B(\mathcal{M}^\perp)\) by
	\[
	S_\phi^{\mathcal{M}^\perp} f = P_{\mathcal{M}^\perp}(M_\phi\otimes I)f,~ f\in \mathcal{M}^\perp.
	\]
	For $\mu\in \mathbb{B}\setminus\{0\}$, let $ \varphi_{\mu}$ be the standard M\"obius automorphism of $\mathbb{B}$, given by
	$$
	\varphi_{\mu}(z)=\frac{-\sqrt{1-|\mu|^{2}}|\mu|^{2} z+\left(\sqrt{1-|\mu|^{2}}-1\right)\langle z, \mu \rangle \mu+|\mu|^{2} \mu}{|\mu|^{2}(1-\langle z, \mu \rangle)}.
	$$
	Define
	$v_{\mu}:H^{2}\to H^2$ by
	$
	v_{\mu} f=f \circ \varphi_{\mu} k_{\mu},
	$
	where $k_{\mu}=\frac{K_{\mu}}{\left\|K_{\mu}\right\|}$ is the normalized reproducing kernel of $H^{2}$.
	Then similar to \cite[Page 199]{Gu}, $v_{\mu}$ is a unitary operator, and $v_{\mu}^2=I$. We have the following lemma.
\begin{lem}\label{vmuM}
	Let \(\mu\in\mathbb B\setminus\{0\}\), and let \(\mathcal M\subset H^2\otimes E\) be a submodule. Put
	\[
	\widetilde{\mathcal M}=(v_\mu\otimes I_E)\mathcal M.
	\]
	Then the following statements hold.
	\begin{enumerate}
		\item[(i)] For every \(i\geq 1\), \(z_i\circ\varphi_\mu\in\mathrm{Mult}(H^2)\) and $M_{z_i\circ\varphi_\mu}=v_\mu M_{z_i}v_\mu$.
		\item[(ii)] \(\widetilde{\mathcal M}\) is a submodule of \(H^2\otimes E\).
		\item[(iii)] \(P_{\widetilde{\mathcal M}}=(v_\mu\otimes I_E)P_{\mathcal M}(v_\mu\otimes I_E)\).
		\item[(iv)] For every \(i\geq 1\),
		\[
		S_{z_i}^{\widetilde{\mathcal M}^\perp}
		=
		(v_\mu\otimes I_E)\,S_{z_i\circ\varphi_\mu}^{\mathcal M^\perp}\,(v_\mu\otimes I_E).
		\]
        \item[(v)] If \(\mathcal M\) is finitely generated, then so is
\(\widetilde{\mathcal M}\).
	\end{enumerate}
\end{lem}

\begin{proof}
	(i). Notice that
	\begin{equation}\label{kmukumphi}
		k_\mu(z) k_\mu(\varphi_\mu(z)) = \frac{1-\|\mu\|^2}{(1- \langle z,\mu\rangle)(1- \langle \varphi_\mu(z),\mu\rangle) }=1.
	\end{equation}
	Then for \(f\in H^2\), by \eqref{kmukumphi},
	\[
	(v_\mu M_{z_i}v_\mu f)(z) =k_\mu(z)( M_{z_i}v_\mu f)(\varphi_\mu (z))=k_\mu(z)\varphi_{\mu,i} (z)k_\mu(\varphi_\mu (z))f\bigl(\varphi_\mu(\varphi_\mu (z))\bigr)
	=\varphi_{\mu,i}(z) f(z).
	\]
	Thus (i) holds.
	
	\noindent(ii). For any $i\geq 1$, by (i),
	\[
	\begin{aligned}
		(M_{z_i}\otimes I_E)\widetilde{\mathcal M}
		&=(M_{z_i}\otimes I_E)(v_\mu\otimes I_E)\mathcal M\\
		&=(v_\mu\otimes I_E)(M_{z_i\circ\varphi_\mu}\otimes I_E)\mathcal M\\
		&\subset(v_\mu\otimes I_E)\mathcal M
		=\widetilde{\mathcal M},
	\end{aligned}
	\]
	Thus (ii) holds.
	
	\noindent(iii). Since \(v_\mu\otimes I_E\) is unitary and \((v_\mu\otimes I_E)^2=I\),
	\[
	P_{\widetilde{\mathcal M}}
	=P_{(v_\mu\otimes I_E)\mathcal M}
	=(v_\mu\otimes I_E)P_{\mathcal M}(v_\mu\otimes I_E)^*=(v_\mu\otimes I_E)P_{\mathcal M}(v_\mu\otimes I_E).
	\]
	
	\noindent(iv). Let \(g\in\mathcal M^\perp\), then by (iii), $(v_\mu\otimes I_E)g\in\widetilde{\mathcal M}^\perp$. Hence for $i\geq 1$,
	$$
	\begin{aligned}
		S_{z_i}^{\widetilde{\mathcal M}^\perp}(v_\mu\otimes I_E)g
		&=P_{\widetilde{\mathcal M}^\perp}(M_{z_i}\otimes I_E)(v_\mu\otimes I_E)g\\
		&\overset{(*)}{=}(v_\mu\otimes I_E)P_{\mathcal M^\perp}(v_\mu\otimes I_E)
		(M_{z_i}\otimes I_E)(v_\mu\otimes I_E)g\\
		&\overset{(**)}{=}(v_\mu\otimes I_E)P_{\mathcal M^\perp}(M_{z_i\circ\varphi_\mu}\otimes I_E)g\\
		&=(v_\mu\otimes I_E)S_{z_i\circ\varphi_\mu}^{\mathcal M^\perp}g,
	\end{aligned}
	$$
	where $(*)$ is from (iii) and $(**)$ is from (i). Therefore for $i\geq 1$,
	$$
	S_{z_i}^{\widetilde{\mathcal M}^\perp}(v_\mu\otimes I_E)=(v_\mu\otimes I_E)S_{z_i\circ\varphi_\mu}^{\mathcal M^\perp}.
	$$
	Then by \((v_\mu\otimes I_E)^2=I\), (iv) holds.

\noindent(v). Let
\[
\mathcal M
=
\overline{\left\{\sum_{j=1}^N M_{h_j}f_j:\ h_j\in\mathrm{Mult}(H^2)\right\}}.
\]
For \(h\in\mathrm{Mult}(H^2)\), by (i) and the unitarity of \(v_\mu\),
\[
v_\mu M_h v_\mu=M_{h\circ\varphi_\mu},
\]
with \(h\circ\varphi_\mu\in\mathrm{Mult}(H^2)\). Thus
\[
(v_\mu\otimes I_E)\mathcal M
=
\overline{\left\{\sum_{j=1}^N M_{\psi_j}(v_\mu\otimes I_E)f_j:
\ \psi_j\in\mathrm{Mult}(H^2)\right\}}.
\]
Therefore \(\widetilde{\mathcal M}\) is finitely generated by
\((v_\mu\otimes I_E)f_1,\cdots,(v_\mu\otimes I_E)f_N\).
\end{proof}

\begin{prop}\label{prop1.2}
	Let \(\mathcal M\) be a finitely generated submodule of \(H^2\otimes E\), and \(\mathbb{S}=\left(S_{z_1}^{{\mathcal M}^\perp}, S_{z_2}^{{\mathcal M}^\perp}, \cdots\right)\). Suppose that there is a \(\mu\in \mathbb{B}\), such that
	\begin{equation}\label{codim}
		\operatorname{codim}\{f(\mu)|f\in \mathcal M\}\geq 1.
	\end{equation}
	Then both \(\partial_\mathbb{S}\) and \(\widehat{\mathbb{S}}\) are unbounded.
\end{prop}

\begin{proof}

	Since \(\{E_{i,t}^*(e_1\wedge\cdots\wedge e_t):1\le i\le t\}\) is orthonormal, for \(x\in \mathcal M^\perp\),
	\begin{equation}
		\|\partial_{\mathbb{S},t}(x\otimes e_1\wedge\cdots\wedge e_t)\|^2
		=
		\left\|\sum_{i=1}^t S_{z_i}^{{\mathcal M}^\perp} x\otimes E_{i,t}^*(e_1\wedge\cdots\wedge e_t)\right\|^2
		=
		\sum_{i=1}^t\|S_{z_i}^{{\mathcal M}^\perp} x\|^2. \label{eq:2.6}
	\end{equation}
	To prove \(\partial_{\mathbb{S}}\) and \(\widehat{\mathbb S}\) are unbounded, by Lemma \ref{partial norm hatT} and \eqref{eq:2.6}, it suffices to prove that there exists \(x\in \mathcal M^\perp\) with \(\sum\limits_{i=1}^{\infty}\|S_{z_i}^{{\mathcal M}^\perp} x\|^2=\infty\).\\
	\noindent Case 1. \(\mu=0\).
	
	By \eqref{codim}, there exists a unit \(\zeta\in E\) with \(\langle f(0),\zeta\rangle=0\) for all \(f\in \mathcal M\), then \(1\otimes\zeta\in \mathcal M^\perp\). Let \(Q_\zeta\) be the projection onto \(H_1\otimes\mathbb C\zeta\), where \(H_1\) is the space of homogeneous functions of degree \(1\). Suppose \(\mathcal M=[f_1,\cdots,f_n]\), then for any polynomial \(p\) and any \(1\leq j\leq n\),
	\begin{equation}
		Q_\zeta(pf_j)=p(0)Q_\zeta f_j. \label{eq:2.7}
	\end{equation}
	Hence 
	$$
	Q_\zeta \mathcal M\subset \operatorname{span}\{Q_\zeta f_1,\cdots,Q_\zeta f_n\},
	$$
	which implies that \(Q_\zeta P_{\mathcal M}|_{H_1\otimes\mathbb C\zeta}\) has finite rank. Then from \(z_i\otimes\zeta\to0\) weakly,
	\[
	\|P_{\mathcal M}(z_i\otimes\zeta)\|^2
	=\langle P_{\mathcal M}(z_i\otimes\zeta),z_i\otimes\zeta\rangle
	=\langle Q_\zeta P_{\mathcal M}(z_i\otimes\zeta),z_i\otimes\zeta\rangle
	\to0.
	\]
	Therefore
	\begin{equation}\label{Si1zeta}
		\|S_{z_i}^{{\mathcal M}^\perp} (1\otimes\zeta)\|^2=\|P_{\mathcal M^\perp}(z_i\otimes\zeta)\|^2
		=1-\|P_{\mathcal M}(z_i\otimes\zeta)\|^2\to1,
	\end{equation}
	which implies that \(\sum\limits_{i=1}^\infty\|S_{z_i}^{{\mathcal M}^\perp} (1\otimes\zeta)\|^2=\infty\).\\ 
	Case 2. $\mu\neq 0.$
	
	By \eqref{codim}, there exists a unit \(\zeta\in E\) such that \(\langle f(\mu),\zeta\rangle=0\) for all \(f\in \mathcal M\). Then \(k_\mu\otimes\zeta\in \mathcal M^\perp\). Let
	\[
	\widetilde{\mathcal M}=(v_\mu\otimes I_E)\mathcal M.
	\]
	By Lemma \ref{vmuM}, \(\widetilde{\mathcal M}\) is a finitely generated submodule of \(H^2\otimes E\). Moreover, for every \(f\in \mathcal M\), 
$$
\begin{aligned}
\langle ((v_\mu\otimes I_E)f)(0),\zeta\rangle
=&\langle (v_\mu\otimes I_E)f,1\otimes\zeta\rangle\\
=&\langle f,(v_\mu\otimes I_E)(1\otimes\zeta)\rangle\\
=&\langle f,k_\mu\otimes \zeta\rangle\\
=&\sqrt{1-\|\mu\|^2}\langle f(\mu),\zeta\rangle\\
=&0.
\end{aligned}
$$
	Hence \(\operatorname{codim}\{f(0)\mid f\in\widetilde{\mathcal M}\}\geq 1\). Applying Case 1 to $\widetilde{\mathcal M}$, we obtain constants \(c_1,N_1>0\) such that for all \(i>N_1\),
	$
	\|S_{z_i}^{\widetilde{\mathcal M}^\perp}(1\otimes\zeta)\|>c_1.
	$
	Thus by Lemma \ref{vmuM}(iv),
	\begin{equation}\label{sziphi>c1}
		\|S_{z_i\circ\varphi_\mu}^{\mathcal M^\perp}(k_\mu\otimes\zeta)\|
		=\|S_{z_i}^{\widetilde{\mathcal M}^\perp}(1\otimes\zeta)\|>c_1,\quad \forall i\geq N_1.
	\end{equation}
	Notice that
	\begin{equation}\label{zicircvarphi}
		z_i\circ\varphi_\mu
	=\frac{-\sqrt{1-|\mu|^2}}{1-\langle z,\mu\rangle}z_i
	+\frac{(\sqrt{1-|\mu|^2}-1)\langle z,\mu\rangle\mu_i+|\mu|^2\mu_i}{|\mu|^2(1-\langle z,\mu\rangle)}.
	\end{equation}
Since $\lim\limits_{i\to\infty}\mu_i=0$, by \eqref{sziphi>c1}, there exists \(N_2>N_1\) such that for all \(i\geq N_2\),
$$
   \left\|S^{{\mathcal M}^\perp}_{f_{\mu,i}}(k_\mu\otimes\zeta)\right\|\leq \frac{1}{2}c_1< \frac{1}{2}\|S_{z_i\circ\varphi_\mu}^{\mathcal M^\perp}(k_\mu\otimes\zeta)\|,
$$
where $f_{\mu,i}=\frac{\left(\sqrt{1-|\mu|^{2}}-1\right)\langle z, \mu \rangle\mu_{i}+|\mu|^{2} \mu_{i}}{|\mu|^{2}(1-\langle z, \mu \rangle)}.$ Hence by \eqref{zicircvarphi},
\begin{equation}\label{szisk>0.5}
\left\|S_{\frac{\sqrt{1-|\mu|^2}z_i}{1-\langle z,\mu\rangle}}^{{\mathcal M}^\perp}(k_\mu\otimes\zeta)\right\|> \frac{1}{2}\|S_{z_i\circ\varphi_\mu}^{\mathcal M^\perp}(k_\mu\otimes\zeta)\|,\quad \forall i\geq N_2.
\end{equation}
By \eqref{szisk>0.5}, for $i>N_2$,
	$$
	\|S_{z_i\circ\varphi_\mu}^{\mathcal M^\perp}(k_\mu\otimes\zeta)\|< 2\sqrt{1-|\mu|^{2}}\|S_{\frac{z_i}{1-\langle z,\mu\rangle}}^{\mathcal M^\perp}(k_\mu\otimes\zeta)\|=2\sqrt{1-|\mu|^{2}}\|S_{K_\mu}^{\mathcal M^\perp}S_{z_i}^{\mathcal M^\perp}(k_\mu\otimes\zeta)\|,
	$$
	which implies that
	\begin{equation}\label{szigesksphi}
		\|S_{z_i}^{\mathcal M^\perp}(k_\mu\otimes\zeta)\|
		>\frac{1}{2\sqrt{1-|\mu|^{2}}}\|S_{K_\mu}\|^{-1}\|S_{z_i\circ\varphi_\mu}^{\mathcal M^\perp}(k_\mu\otimes\zeta)\|, \quad \forall i\geq N_2.
	\end{equation}
	Therefore, for \(i\geq N_2\), by \eqref{sziphi>c1} and \eqref{szigesksphi},
	\[
	\|S_{z_i}^{\mathcal M^\perp}(k_\mu\otimes\zeta)\|
	> \frac{c_1}{2\sqrt{1-|\mu|^{2}}\|S_{K_\mu}\|}.
	\]
	Consequently,
	$
	\sum\limits_{i=1}^\infty\|S_{z_i}^{\mathcal M^\perp}(k_\mu\otimes\zeta)\|^2=\infty.
	$

	\end{proof}

The condition \eqref{codim} in Proposition~\ref{prop1.2} is not superfluous.
In fact, the following example exhibits a proper submodule $M$ without \eqref{codim} and $S_{z_j}\neq0$ for all $j$, for which both $\partial_{\mathbb S}$ and $\widehat{\mathbb S}$ are bounded. Firstly we need a lemma.
\begin{lem}\label{lem:bounded-partial-normal}
Let $\mathbb T=(T_1,T_2,\cdots)$ be a row‑bounded sequence of commuting normal operators on $H$.
Then $\partial_{\mathbb T}$ is a bounded operator on $\Lambda(H)$ and
\[
\|\partial_{\mathbb T}\|=\left\|\sum_{i=1}^\infty T_iT_i^*\right\|^{1/2}.
\]
\end{lem}

\begin{proof}
Since each \(T_i\) is normal, Fuglede's theorem gives that for all $i,j\in\mathbb N$, \(T_iT_j^*=T_j^*T_i\). 
For each \(N\ge 1\), let \(\partial_{\mathbb T}^{(N)} = \sum\limits_{i=1}^N T_i \otimes E_i^*\). For \(n\ge 0\), a direct computation on \(\Lambda^n(H)\) gives
\begin{equation}\label{laplacian-truncated}
\partial_{\mathbb T,n}^{(N)*}\partial_{\mathbb T,n}^{(N)} + \partial_{\mathbb T,n+1}^{(N)}\partial_{\mathbb T,n+1}^{(N)*}
= \sum_{i,j=1}^N T_i^*T_j \otimes (E_iE_j^* + E_j^*E_i)\Big|_{\Lambda^n(H)}
= \sum_{i=1}^N T_iT_i^* \otimes I_{\Lambda^n}.
\end{equation}
By Theorem 2.1, we have
\[
\operatorname{SOT-}\lim_{N\to\infty}\partial_{\mathbb T,n}^{(N)} = \partial_{\mathbb T,n},
\qquad
\operatorname{SOT-}\lim_{N\to\infty}\left(\partial_{\mathbb T,n}^{(N)}\right)^* = \partial_{\mathbb T,n}^*.
\]
Thus passing to the SOT limit in \eqref{laplacian-truncated} yields
\begin{equation}\label{laplacian-normal}
\partial_{\mathbb T,n}^*\partial_{\mathbb T,n} + \partial_{\mathbb T,n+1}\partial_{\mathbb T,n+1}^* = \sum_{i=1}^\infty T_iT_i^* \otimes I_{\Lambda^n}.
\end{equation}
Let \(C = \left\|\sum\limits_{i=1}^\infty T_iT_i^*\right\|\). For any \(x\in\Lambda^{n+1}(H)\), we have
\[
\|\partial_{\mathbb T,n+1} x\|^2
\le \langle (\partial_{\mathbb T,n+1}^*\partial_{\mathbb T,n+1} + \partial_{\mathbb T,n+2}\partial_{\mathbb T,n+2}^*) x, x\rangle
= \left\langle \left(\sum_{i=1}^\infty T_iT_i^*\right)\otimes I\, x, x\right\rangle
\le C\|x\|^2 .
\]
Hence \(\sup\limits_{t\ge 1}\|\partial_{\mathbb T,t}\|\le\sqrt{C}\), which implies \(\partial_{\mathbb T}\) is bounded and \(\|\partial_{\mathbb T}\|\leq \sqrt{C}\).
Conversely, by \eqref{laplacian-normal} and \(\partial_{\mathbb{T},0}=0\), we have
\[
\|\partial_{\mathbb T,1}\partial_{\mathbb T,1}^*\| = \left\|\sum_{i=1}^\infty T_iT_i^*\right\| = C.
\]
Hence
\[
\|\partial_{\mathbb T}\|\geq \|\partial_{\mathbb T,1}\|=\sqrt{\|\partial_{\mathbb T,1}\partial_{\mathbb T,1}^*\|}=\sqrt{C}.
\]
\end{proof}

By \cite[Corollary 1.2]{Malman} (with \(w(r)=\frac12(1-r)^{1/4}\)), there exist a nonconstant singular inner function \(\theta\) and a sequence of polynomials \(\{p_n\}\) such that, for every \(z\in\mathbb D\),
\begin{equation}\label{theta z pn z}
\lim_{n\to\infty}\theta(z)p_n(z)=1
\quad\text{and}\quad
|\theta(z)p_n(z)|\le 4(1-|z|)^{-1/4}.
\end{equation}
For \(\theta\) chosen as above, the following two lemmas hold. For \(m\ge1\), let \(A_m^2\) denote the weighted Bergman space on \(\mathbb D\) with norm
\[
\|f\|_{A_m^2}^2=\int_{\mathbb D}|f(z)|^2(1-|z|^2)^{m-1}\,dA(z).
\]
where \(dA\) denotes the normalized area measure on \(\mathbb{D}\).
\begin{lem}\label{inner theta}
  $\overline{\theta\mathbb C[z]}^{A_m^2}=A_m^2$.
\end{lem}
\begin{proof}
By \eqref{theta z pn z}, there exists a constant \(C>0\), independent of \(n\) and \(z\), such that
\[
|\theta p_n-1|^2(1-|z|^2)^{m-1}
\le C\bigl((1-|z|)^{-1/2}+1\bigr)(1-|z|^2)^{m-1}.
\]
Since
\[
((1-r)^{-1/2}+1)(1-r^2)^{m-1}\in L^1(0,1),
\]
the dominated convergence theorem gives
\[
\lim_{n\to\infty}\|\theta p_n-1\|_{A_m^2}^2
=
\int_{\mathbb D}\lim_{n\to\infty}|\theta p_n-1|^2(1-|z|^2)^{m-1}\,dA(z)
=0.
\]
Thus \(1\in\overline{\theta\mathbb C[z]}^{A_m^2}\), which implies that
$$
\overline{\theta\mathbb C[z]}^{A_m^2}=A_m^2.
$$
\end{proof}
 Let
\[
\mathbb N^{(\mathbb N)}=\left\{(\beta_1,\beta_2,\cdots):\beta_j\in\mathbb N,\ \#\{j:\beta_j\ne0\}<\infty\right\}.
\]
For $\beta\in \mathbb N^{(\mathbb N)}$ and $z\in\ell^2$, write $|\beta|=\sum\limits_{j\ge1}^\infty\beta_j$, $\beta!=\prod\limits_{j\ge1}^\infty\beta_j!$ and $z^\beta=\prod\limits_{j\ge1}^\infty z_j^{\beta_j}$. 
\begin{lem}\label{thetaz1}
Let \(\{e_k\}_{k=1}^N\) be the standard orthonormal basis of \(\mathbb C^N\), and
$$
M=[\theta(z_1)\otimes e_1,\cdots,\theta(z_1)\otimes e_N]
$$ 
be a submodule of $H^2\otimes \mathbb{C}^N$.
Then 
$$
S_{z_1}\neq 0,\quad S_{z_i}=0,~\forall i\geq 2.
$$
\end{lem}
\begin{proof}
For every fixed $\beta=(0,\beta_2,\beta_3,\cdots)\in\mathbb N^{(\mathbb N)}\setminus\{0\}$, set $m=|\beta|\geq 1$. For each $n\ge0$, by Lemma \ref{inner theta}, there are polynomials $p_k$ with
$$
\|\zeta^n-p_k(\zeta)\theta(\zeta)\|_{A_{m}^2}\to0.
$$
Notice that
\[
\|z_1^n z^\beta\|_{H^2}^2=\frac{\beta!}{(m-1)!}\|\zeta^n\|_{A_{m}^2}^2.
\]
Then
\[
\|z_1^n z^\beta-p_k(z_1)\theta(z_1)z^\beta\|_{H^2}^2
=\frac{\beta!}{(m-1)!}\|\zeta^n-p_k(\zeta)\theta(\zeta)\|_{A_{m}^2}^2\to0,
\]
hence $z_1^n z^\beta\otimes e_j\in M$ for each $1\leq j\leq N$. Then for $f\in M^\perp$,
$$
\langle f,z_1^n z^\beta\otimes e_j\rangle=0,~\forall \beta=(0,\beta_2,\cdots)\in\mathbb N^{(\mathbb N)}\setminus\{0\},~1\leq j\leq N,
$$
which implies that $f=\sum\limits_{j=1}^{N}\sum\limits_{n=0}^{\infty}\langle f,z_1^n \otimes e_j\rangle z_1^n\otimes e_j$, i.e. $f\in H_{z_1}^2(\mathbb{D})\otimes \mathbb{C}^N$, where \(H_{z_1}^2(\mathbb D)\) denotes the Hardy space in the variable \(z_1\). Hence 
\begin{equation}\label{M perp in H z1}
M^\perp\subseteq H_{z_1}^2(\mathbb D)\otimes\mathbb C^N. 
\end{equation}
Let \(P_{z_1}\) be the orthogonal projection of \(H^2\otimes\mathbb C^N\) onto 
\(H_{z_1}^2(\mathbb D)\otimes\mathbb C^N\). By \eqref{M perp in H z1}, for 
\(f\in H_{z_1}^2(\mathbb D)\otimes\mathbb C^N\) we have
\begin{equation}\label{M perp eq Pz1 M}
f\in M^\perp
\quad\Longleftrightarrow\quad
f\perp P_{z_1} M.
\end{equation}
Moreover, it is easy to see that, 
\begin{equation}\label{Pz1 M}
P_{z_1} M=\theta H_{z_1}^2(\mathbb D)\otimes\mathbb C^N.
\end{equation}
Therefore by \eqref{M perp in H z1}, \eqref{M perp eq Pz1 M}, and \eqref{Pz1 M},
\[
M^\perp
=
\bigl(H_{z_1}^2(\mathbb D)\otimes\mathbb C^N\bigr)
\ominus
\bigl(\theta H_{z_1}^2(\mathbb D)\otimes\mathbb C^N\bigr)
=
\bigl(H_{z_1}^2(\mathbb D)\ominus \theta H_{z_1}^2(\mathbb D)\bigr)\otimes\mathbb C^N.
\]
Thus \(S_{z_i}=0\) for all \(i\ge2\). 

Now we show that \(S_{z_1}\ne0\). Suppose \(S_{z_1}=0\). Let 
$$
K_\theta=H_{z_1}^2(\mathbb D)\ominus \theta H_{z_1}^2(\mathbb D).
$$ 
Then by \(S_{z_1}=0\),
\begin{equation}\label{zKtheta in H2}
M_{z_1}K_\theta\subseteq \theta H_{z_1}^2(\mathbb D).
\end{equation}
Take \(f=1-\overline{\theta(0)}\theta\in K_\theta\). Since \(\theta\) is nonconstant, \(f\neq0\). By \eqref{zKtheta in H2}, \(z_1 f\in \theta H_{z_1}^2(\mathbb D)\), which implies that
\[
\frac{z_1}{\theta} = z_1\overline{\theta(0)}+\frac{z_1 f}{\theta} \in H_{z_1}^2(\mathbb D).
\]
Hence \(\theta\) divides \(z_1\) in \(H_{z_1}^2(\mathbb D)\). This forces \(\theta=cz_1\) for some \(c\neq0\), which contradicts the fact that \(\theta\) is a singular inner function. Therefore, \(S_{z_1}\ne0\).
\end{proof}
At the end of this section, we will give the following example to illustrate that $\partial \mathbb T$ could be bounded. 
\begin{exam}\label{exam:empty-zero-bounded}
		Let $N\in \mathbb N_+$. There exists a proper finitely generated submodule
		$\mathcal M\subset H^2\otimes \mathbb C^N$ with
		\[
		\operatorname{codim}_{\mathbb C^N}
		\{f(\mu):f\in\mathcal M\}=0,
		\qquad \mu\in\mathbb B,
		\]
		such that $\partial_{\mathbb S}$ and $\widehat{\mathbb S} $ are bounded. In fact, choose $\theta$ as in \eqref{theta z pn z}. Let
		
		$$
		g(z)=\theta\left(\sum_{j=1}^\infty \sqrt{\frac{1}{2^j}} z_j\right),~z\in\mathbb{B}
		$$
and
$$
		\mathcal M=[g \otimes e_1,\ldots,g  \otimes e_N]\subset H^2\otimes\mathbb C^N.
		$$
For $z\in \mathbb{B}$, since $\left|\sum\limits_{i=1}^\infty \sqrt{\frac{1}{2^j}} z_j\right|<1$, $g(z)\neq 0$, then
		\[
		\{f(z):f\in\mathcal M\}
		=
		\operatorname{span}\{g(z)e_k:1\leq k\leq N\}
		=
		\mathbb C^N,
		\]
		and hence
		\[
		\operatorname{codim}_{\mathbb C^N}\{f(z):f\in\mathcal M\}=0,
		\]
Write $a=\left(\sqrt{\frac{1}{2}},\sqrt{\frac{1}{2^2}},\cdots\right)$.
		Choose a unitary operator $V\in B(\ell^2)$ such that $Ve_1=a$.
		Define $U_V:H^2\to H^2$ by \((U_Vf)(z)=f(Vz)\), then \(U_V\) is unitary on \(H^2\).
		Since
		\[
		\sum_{j=1}^\infty a_j(Vz)_j
		=
		\langle Vz,a\rangle
		=
		z_1,
		\]
we have
$$
\begin{aligned}
		(U_V\otimes I_{\mathbb{C}^N})\mathcal M
=&[(U_Vg\otimes e_1),\cdots,(U_Vg\otimes e_N)]\\
=&\left[\theta\left(\sum_{j=1}^\infty a_j(Vz)_j\right)\otimes e_1,\cdots,\theta\left(\sum_{j=1}^\infty a_j(Vz)_j\right)\otimes e_N\right]\\
=&[\theta(z_1) \otimes e_1,\ldots,\theta(z_1) \otimes e_N].
\end{aligned}
$$
Write $\mathcal{M}'=(U_V\otimes I_{\mathbb{C}^N})\mathcal M$. Let 
$$
S'_{z_i}=P_{\mathcal{M}'^\perp}(M_{z_i}\otimes I_{\mathbb{C}^N})P_{\mathcal{M}'^\perp}.
$$
By Lemma \ref{thetaz1},
		\(S'_{z_i}=0\) for all \( i\geq 2\) and \(S'_{z_1}\neq 0\).
Write
\[
(Vz)_j = \langle Vz, e_j\rangle=  \sum_{k=1}^\infty \langle e_k, V^* e_j\rangle z_k.
\]
Then we have
\begin{equation}\label{USU=ajS}
\begin{aligned}
(U_V\otimes I_{\mathbb{C}^N}) S_{z_j} (U_V\otimes I_{\mathbb{C}^N})^*
&= (U_V\otimes I_{\mathbb{C}^N}) P_{\mathcal M^\perp}(M_{z_j}\otimes I_{\mathbb{C}^N})P_{\mathcal M^\perp}(U_V\otimes I_{\mathbb{C}^N})^* \\
&\overset{(*)}{=} P_{\mathcal{M}'^\perp} (U_V\otimes I_{\mathbb{C}^N}) (M_{z_j}\otimes I_{\mathbb{C}^N}) (U_V\otimes I_{\mathbb{C}^N})^* P_{\mathcal{M}'^\perp}\\
&\overset{(\square)}{=} P_{\mathcal{M}'^\perp} (M_{(Vz)_j}\otimes I_{\mathbb{C}^N}) P_{\mathcal{M}'^\perp} \\
&= P_{\mathcal{M}'^\perp} (M_{ \sum\limits_{k=1}^\infty \langle e_k, V^* e_j\rangle z_k}\otimes I_{\mathbb{C}^N}) P_{\mathcal{M}'^\perp}\\
&= \sum_{k\ge 1} \langle e_k, V^* e_j\rangle S'_{z_k} \\
&= a_j S'_{z_1},
\end{aligned}
\end{equation}
where $(*)$ is from $(U_V\otimes I_{\mathbb{C}^N})P_\mathcal{M}=P_{\mathcal{M}'}(U_V\otimes I_{\mathbb{C}^N})$ and $(\square)$ is from $U_VM_{z_j}=M_{(Vz)_j}U_V$.
By Lemma \ref{lem:bounded-partial-normal}, $\left\|
			\sum\limits_{j=1}^\infty a_jI_H\otimes E_j^*
			\right\|\leq 1$. Then by \eqref{USU=ajS},
		\[
		\begin{aligned}
			\|\partial_{\mathbb S}\|
			&=
			\left\|
			\sum_{j=1}^\infty
			(U_V\otimes I_{\mathbb{C}^N})^*(a_jS'_{z_1})(U_V\otimes I_{\mathbb{C}^N})\otimes E_j^*
			\right\| \\
			&=
			\left\|((U_V\otimes I_{\mathbb{C}^N})^*\otimes I_\Lambda)(S'_{z_1}\otimes I_\Lambda)\left(
			\sum_{j=1}^\infty a_jI_H\otimes E_j^*
			\right)
			((U_V\otimes I_{\mathbb{C}^N})\otimes I_\Lambda)\right\|\\ 
			&\leq 1
		\end{aligned}
		\]
		Hence by Lemma~\ref{partial norm hatT}, both $\partial_{\mathbb S}$ and $\widehat {\mathbb S}$ are bounded.
\end{exam}
\section{The properties of the spectrum $\sigma(\mathbb T)$}
In this section, we study basic properties of the spectrum $\sigma(\mathbb T)$.
\subsection{Invertibility criteria for row-bounded sequences}
In this subsection, we will give some criteria for the invertibility of row-bounded sequences. 
\begin{prop}\label{Iinvert iff uniform homotopy}
\label{prop:uniform-homotopy}
Let
	\[ 
	(X,\partial_X): \quad 0=X_{-1} \xleftarrow{\partial_{X,0}} X_0 \xleftarrow{\partial_{X,1}} X_1 \xleftarrow{\partial_{X,2}} X_2 \xleftarrow{\partial_{X,3}} \cdots 
	\] 
	be a Hilbert complex, where \(\{X_p\}_{p \geq -1}\) are Hilbert spaces, \( \partial_{X,p}\in B(X_p,X_{p-1})\) for \(p\geq0\). Then the following statements are equivalent.
\begin{itemize}
  \item [(1)] There exist uniformly bounded operators \(h_p:X_p\to X_{p+1},\ p\ge0\), with $h_{-1}=0$, such that
\[
\partial_{X,p+1}h_p + h_{p-1}\partial_{X,p} = I_{X_p}, \ p\ge0;
\]
  \item [(2)] There exists $c> 0$ such that, for every $p\geq 0$
  and every $x_p\in\ker\partial_{X,p}$, there is a $y_{p+1}\in X_{p+1}$ with
  \[
  \partial_{X,p+1}y_{p+1}=x_p,
  \qquad
  \|y_{p+1}\|\leq c\|x_p\|.
  \]
\end{itemize}
\end{prop}

\begin{proof}
$(1)\Rightarrow(2)$: It suffices to take \(c= \sup\limits_{k\geq 0}\|h_k\|\) and \(y_{p+1}= h_px_p\).\\
$(2)\Rightarrow(1)$: 
For every $x\in\ker\partial_{X,p}$, by (2), there exists
$y\in X_{p+1}$ such that
$\partial_{X,p+1}y=x$ and $\|y\|\leq c\|x\|$.
Hence
\[
\|x\|^2
=\bigl|\langle y,\partial_{X,p+1}^*x\rangle\bigr|
\leq c\|x\|\,\|\partial_{X,p+1}^*x\|,
\]
and consequently
\[
\bigl\langle
\partial_{X,p+1}\partial_{X,p+1}^*x,x
\bigr\rangle
=\|\partial_{X,p+1}^*x\|^2
\geq c^{-2}\|x\|^2.
\]
Thus 
$\left.
\partial_{X,p+1}\partial_{X,p+1}^*
\right|_{\ker\partial_{X,p}}$
is invertible on $\ker\partial_{X,p}$, with
\begin{align}\label{dd*norm}
\left\|
\left(
\left.
\partial_{X,p+1}\partial_{X,p+1}^*
\right|_{\ker\partial_{X,p}}
\right)^{-1}
\right\|
\leq c^2.
\end{align}
Define $h_{-1}=0$ and
\[
h_p
=
\partial_{X,p+1}^*
\left(
\left.
\partial_{X,p+1}\partial_{X,p+1}^*
\right|_{\ker\partial_{X,p}}
\right)^{-1}
P_{\ker\partial_{X,p}},\ p\geq0.
\]
By \eqref{dd*norm}, for $x\in\ker\partial_{X,p}$,
$$
\begin{aligned}
	\|h_p x\|^2 
	&= \langle \partial_{X,p+1}^* \left(\partial_{X,p+1}\partial_{X,p+1}^*|_{\ker\partial_{X,p}}\right)^{-1}x, \partial_{X,p+1}^* \left(\partial_{X,p+1}\partial_{X,p+1}^*|_{\ker\partial_{X,p}}\right)^{-1}x \rangle \\
	&= \langle \partial_{X,p+1}\partial_{X,p+1}^* \left(\partial_{X,p+1}\partial_{X,p+1}^*|_{\ker\partial_{X,p}}\right)^{-1}x, \left(\partial_{X,p+1}\partial_{X,p+1}^*|_{\ker\partial_{X,p}}\right)^{-1}x\rangle \\
	&=\langle x, \left(\partial_{X,p+1}\partial_{X,p+1}^*|_{\ker\partial_{X,p}}\right)^{-1}x\rangle \\
	&\le \|x\|\,\|\left(\partial_{X,p+1}\partial_{X,p+1}^*|_{\ker\partial_{X,p}}\right)^{-1}x\| \\
	&\le c^2\|x\|^2.
\end{aligned}
$$
Therefore $\sup\limits_{p\geq0}\|h_p\|\leq c$.
For $p\geq1$, by (2), we have \(\operatorname{ran}\partial_{X,p}
=\ker\partial_{X,p-1}\).
Hence
\begin{align}\label{ker*}
\ker \partial_{X,p}^* = (\operatorname{ran} \partial_{X,p})^\perp =(\ker \partial_{X,p-1})^\perp.
\end{align}
Obviously,  $\operatorname{ran} \partial^*_{X,p+1}$ is closed.
Then by \eqref{ker*},
\[
(\ker \partial_{X,p})^\perp = \operatorname{ran} \partial^*_{X,p}=\partial^*_{X,p}\bigl((\ker \partial_{X,p-1})^\perp + \ker \partial_{X,p-1}\bigr)=\partial^*_{X,p}\ker \partial_{X,p-1}.
\]
Thus, for every $x\in X_p$, there exists
$y\in\ker\partial_{X,p-1}$ such that
\(P_{(\ker\partial_{X,p})^\perp}x=\partial_{X,p}^*y\).
Then
\begin{align*}
&(\partial_{X,p+1}h_p + h_{p-1}\partial_{X,p})x\\
=& (\partial_{X,p+1}h_p + h_{p-1}\partial_{X,p})\left(P_{(\ker\partial_{X,p})^\perp}x+P_{\ker\partial_{X,p}x}\right)\\
 =& h_{p-1}\partial_{X,p}P_{(\ker\partial_{X,p})^\perp}x +\partial_{X,p+1}h_pP_{\ker\partial_{X,p}} x\\ 
 =&\partial_{X,p}^*
 \left(
 \left.\partial_{X,p}\partial_{X,p}^*
 \right|_{\ker\partial_{X,p-1}}
 \right)^{-1}
 \partial_{X,p}\partial_{X,p}^*y +\partial_{X,p+1}\partial_{X,p+1}^* 
 \left(\partial_{X,p+1}\partial_{X,p+1}^*|_{\ker\partial_{X,p}}\right)^{-1}
 P_{\ker\partial_{X,p}}x\\
 =&\partial_{X,p}^*y+P_{\ker\partial_{X,p}}x\\
 =&P_{(\ker\partial_{X,p})^\perp}x +P_{\ker\partial_{X,p}}x =x.
\end{align*}
Moreover, for $p=0$, we have
\[
\partial_{X,1}h_0 + h_{-1}\partial_{X,0} = \partial_{X,1}h_0 
=\partial_{X,1}\partial_{X,1}^*(\partial_{X,1}\partial_{X,1}^* | _{\ker\partial_{X,0}})^{-1}P_{\ker \partial_{X,0}}=P_{\ker \partial_{X,0}} =P_{X_0}.
\]

\end{proof}

\begin{cor}\label{TS=1}
Let $\mathbb{T}$ be a row-bounded sequence of commuting operators. If there exists a row-bounded sequence of commuting operators $\mathbb{S}=\left(S_{1}, S_{2},\cdots \right)$ on $H$ such that $\|\partial_\mathbb{S}\|$ is bounded, $S_{i} T_{j}=T_{j} S_{i}$ for $i, j\geq 1$ and 
$$
\operatorname{SOT}-\lim\limits_{n\to\infty}\sum\limits_{i=1}^{n} T_{i} S_{i}=I,
$$
then $\mathbb{T}$ is invertible.
\end{cor}
\begin{proof}
By Lemma~\ref{sup=} and Lemma~\ref{bound and SOT}, the operator
\[
\widetilde\partial_{\mathbb S}=\sum_{i\ge1}S_i\otimes E_i
\]
is bounded, with $\|\widetilde\partial_{\mathbb S}\|=\|\partial_{\mathbb S}\|$. Notice that
\[
\partial_{\mathbb T,p+1}\widetilde\partial_{\mathbb S}
=\sum_{i,k}T_kS_i\otimes E_k^*E_i
=\sum_{i,k}S_iT_k\otimes(\delta_{ik}I-E_iE_k^*)
=\sum_{i}S_iT_i\otimes I-\widetilde\partial_{\mathbb S}\partial_{\mathbb T,p},
\]
we have
\begin{equation}\label{eq:chain-identity}
\partial_{\mathbb T,p+1}\widetilde\partial_{\mathbb S}+\widetilde\partial_{\mathbb S}\partial_{\mathbb T,p}
=I_{\Lambda^p(H)}
\end{equation}
Then by Proposition \ref{prop:uniform-homotopy}, $\mathbb{T}$ is invertible.
\end{proof}

Recall that an unbounded operator ${\cal{T}}: \mathrm{Dom}(\mathcal{T})(\subseteq X)\to Y$ is invertible, if and only if $\cal T$ is a closed operator and ${\cal T}^{-1}:Y\to \mathrm{Dom}(\mathcal{T})$ is bounded. As an analogue to Curto\cite{CR}, we have the following result. 

\begin{prop}\label{thm3.2}
The following statements are equivalent.
\begin{itemize}
  \item[(1)] $\mathbb{T}$ is invertible;
  \item[(2)] $\operatorname{ker} \partial_{\mathbb{T}}= \operatorname{Im} \partial_{\mathbb{T}}$;
  \item[(3)] $\widehat{\mathbb{T}}$ is invertible.
\end{itemize}
\end{prop}
\begin{proof}
(1)$\Rightarrow$(2): By Lemma $\ref{lem3.6}$, $\operatorname{Im}\partial_{\mathbb{T}}$ is closed. Notice that
$$
\bigcup_{N=1}^{\infty}\left(\bigoplus\limits_{k=1}^{N} \operatorname{Im} \partial_{\mathbb{T}, k}\right) \subseteq \operatorname{Im} \partial_{\mathbb{T}} \subseteq \bigoplus\limits_{k=1}^{\infty} \operatorname{Im} \partial_{\mathbb{T}, k},
$$
and 
$$
\overline{\bigcup\limits_{N=1}^{\infty}\left(\bigoplus\limits_{k=1}^{N} \operatorname{Im} \partial_{\mathbb{T}, k}\right)}=\bigoplus\limits_{k=1}^{\infty} \operatorname{Im} \partial_{\mathbb{T}, k},
$$
hence
$
\operatorname{Im}\partial_{\mathbb{T}}=\bigoplus\limits_{k=1}^{\infty}\operatorname{Im}\partial_{\mathbb{T},k}.
$
Since $\mathbb{T}$ is invertible, we have 
$
\sum\limits_{k=0}^{\infty}\operatorname{dim}H^{k}_{\mathbb T}=0,
$
which implies that $\operatorname{ker}\partial_{\mathbb{T},k}=\operatorname{Im}\partial_{\mathbb{T},k+1}$ for $k\geq 0.$
Hence
$$
\operatorname{ker}\partial_{\mathbb{T}}=\bigoplus\limits_{k=0}^{\infty}\operatorname{ker}\partial_{\mathbb{T},k}=\bigoplus\limits_{k=1}^{\infty}\operatorname{Im}\partial_{\mathbb{T},k}=\operatorname{Im}\partial_{\mathbb{T}}.
$$
(2)\(\Rightarrow\)(3): 
Assume \(\ker \partial_{\mathbb{T}} = \operatorname{Im} \partial_{\mathbb{T}}\). Then \(\operatorname{Im} \partial_{\mathbb{T}}\) is closed. Hence by Lemma \ref{lem3.6} \(\operatorname{Im}\widehat{\mathbb{T}}\) is closed. By \eqref{hat T x eq dx d*x},
\begin{equation}\label{ker T eq 0}
\ker\widehat{\mathbb{T}}
=
\ker\partial_{\mathbb{T}}\cap\ker\partial_{\mathbb{T}}^*\cap\Lambda^+(H)
=
\operatorname{Im}\partial_{\mathbb{T}}\cap(\operatorname{Im}\partial_{\mathbb{T}})^\perp\cap\Lambda^+(H)
=
\{0\}.
\end{equation}
Applying the same argument to \(\widehat{\mathbb{T}}^*\), we obtain \(\ker\widehat{\mathbb{T}}^*=\{0\}\). Hence
\begin{equation}\label{im T eq -}
\operatorname{Im}\widehat{\mathbb{T}}=(\ker\widehat{\mathbb{T}}^*)^\perp=\Lambda^-(H).
\end{equation}
Therefore by \eqref{ker T eq 0} and \eqref{im T eq -}, \(\widehat{\mathbb{T}}\) is invertible.\\
(3)\(\Rightarrow\)(1): 
Assume \(\widehat{\mathbb{T}}\) is invertible. Then \(\widehat{\mathbb T}^*\) is also invertible. Hence there exists $c>0$ such that
\begin{equation}\label{hat T x ge c}
\|\widehat{\mathbb{T}}x\|^2\geq c\|x\|^2,~\forall x\in \operatorname{Dom}(\widehat{\mathbb{T}}),
\end{equation}
and
\begin{equation}\label{hat T x ge c2}
\|\widehat{\mathbb{T}}^*x\|^2\geq c\|x\|^2,~\forall x\in \operatorname{Dom}(\widehat{\mathbb{T}}^*).
\end{equation}
Fix \(p\ge0\) and take \(x_{2p}\in\ker\partial_{\mathbb T,2p}\subset \operatorname{Dom}(\widehat{\mathbb{T}})\). Then by \eqref{hat T x ge c},
\begin{equation}\label{d2p}
\|\partial_{\mathbb T,2p+1}^*x_{2p}\|^2
=
\|\widehat{\mathbb{T}}x_{2p}\|^2
\ge
c\|x_{2p}\|^2.
\end{equation}
Similarly, for \(x_{2p+1}\in\ker\partial_{\mathbb T,2p+1}\subset \operatorname{Dom}(\widehat{\mathbb{T}}^*)\), by \eqref{hat T x ge c2} we have
\begin{equation}\label{d2p+1}
\|\partial_{\mathbb T,2p+2}^*x_{2p+1}\|^2
=
\|\widehat{\mathbb T}^*x_{2p+1}\|^2
\ge
c\|x_{2p+1}\|^2.
\end{equation}
Hence by \eqref{d2p} and \eqref{d2p+1}, for all $n\geq 0$, $\partial_{\mathbb T,n+1}\partial_{\mathbb T,n+1}^*|_{\ker\partial_{\mathbb T,n}}$ is invertible on \(\ker\partial_{\mathbb T,n}\), with 
$$
\left\|(\partial_{\mathbb T,n+1}\partial_{\mathbb T,n+1}^*|_{\ker\partial_{\mathbb T,n}})^{-1}\right\|\leq c^{-1}.
$$
For \(x_n\in\ker\partial_{\mathbb T,n}\), set
\[
y_{n+1}=\partial_{\mathbb T,n+1}^*(\partial_{\mathbb T,n+1}\partial_{\mathbb T,n+1}^*|_{\ker\partial_{\mathbb T,n}})^{-1}x_n.
\]
Then
\[
\partial_{\mathbb T,n+1}y_{n+1}=x_n,~\text{and}~\|y_{n+1}\|^2\le c^{-1}\|x_n\|^2.
\]
Thus $\mathbb{T}$ is invertible.
\end{proof}
Recall that for any row-bounded sequence $\mathbb{T}$, \(\Lambda_{\mathrm{fin}}^+(H)\) is dense in \(\operatorname{Dom}(\widehat{\mathbb{T}})\) with respect to the graph norm. Therefore by \cite[Theorem 7.11]{MS}, we have the following corollary.
\begin{cor}\label{invert continious}
  Let $\mathbb{T}=(T_1,\cdots,T_n,\cdots)$ be a row-bounded sequence on $H$. If $\mathbb{T}$ is invertible, then there exists an $\varepsilon>0$ such that for every row-bounded sequence $\mathfrak{T}=(\mathfrak{T}_{1}, \mathfrak{T}_{2}, \cdots)$ satisfying 
  $$
  \left\|\left(\widehat{\mathbb{T}}-\widehat{\mathfrak{T}}\right)x\right\|<\varepsilon\|x\|,~\forall x\in \Lambda_{\mathrm{fin}}^+(H),
  $$
  $\mathfrak{T}$ is invertible.
\end{cor}

\begin{lem}\label{Tn and T'}
For fixed $n>0$, the following statements are equivalent.
\begin{itemize}
  \item[(1)] $T^{(n)}=(T_{1}, \cdots ,T_{n})$ is invertible;
  \item[(2)] $\mathfrak{T}^{(n)}=(T_{1}, \cdots , T_{n}, 0 , 0 ,\cdots)$ is invertible;
  \item[(3)] $\mathfrak{T}^{(n)}=(T_{1}, \cdots , T_{n}, 0 , 0 ,\cdots)$ is Fredholm.
\end{itemize}
\end{lem}
\begin{proof}
  $(\mathit{2})\Rightarrow (\mathit{3})$ is obvious.\\
   $(\mathit{1})\Rightarrow(\mathit{2})$: By Proposition $\ref{thm3.2}$ and $\operatorname{Im} \partial_{\mathfrak{T}^{(n)}}\subseteq \operatorname{ker} \partial_{\mathfrak{T}^{(n)}}$,  it's sufficient to prove $\operatorname{ker} \partial_{\mathfrak{T}^{(n)}}\subseteq\operatorname{Im} \partial_{\mathfrak{T}^{(n)}}$. For $m\ge 0$ and a strictly increasing sequence $i_1<\cdots<i_m$ of integers $>n$, set
\[
W_{i_{1}\cdots i_{m}}=\overline{\operatorname{span}\Bigl\{
f\otimes e_{j_{1}}\wedge\cdots\wedge e_{j_{k}}\wedge e_{i_{1}}\wedge\cdots\wedge e_{i_{m}} :
1\le j_{1}<\cdots<j_{k}\le n,\; 0\le k\le n,\; f\in H \Bigr\}},
\]
where for $m=0$ the sequence is empty and $W_{i_{1}\cdots i_{m}}$ is understood as
$H\otimes\Lambda_n\subseteq \Lambda(H)$.
Then
\[
\Lambda(H)=\bigoplus\limits_{\substack{n<i_{1}<\cdots<i_{m} \\ 0\le m<\infty}} W_{i_{1}\cdots i_{m}} .
\]
Hence for $x\in \Lambda \left( H \right)$, write
\begin{equation}\label{x W decom}
x=\sum\limits_{n< i_{1} < \cdots < i_{m}, 0\leq m < \infty} P_{W_{i_{1} \cdots i_{m}}}x,
\end{equation}
where $ P_{W_{i_{1} \cdots i_{m}}}$ is the orthogonal projection from $\Lambda(H)$ onto ${W_{i_{1} \cdots i_{m}}}$. Notice that 
$$
\|\partial_{\mathfrak{T}^{(n)}}\|=\left\|\sum_{i=1}^{n}T_i\otimes E_i^*\right\|\leq \sum_{i=1}^{n}\|T_i\|<\infty.
$$
Then for all $x\in \operatorname{ker} \partial_{\mathfrak{T}^{(n)}}$, by \eqref{x W decom},
\begin{equation}\label{ker mathcal T}
  \sum\limits_{n< i_{1} < \cdots < i_{m}, 0\leq m < \infty} \partial_{\mathfrak{T}^{(n)}} P_{W_{i_{1} \cdots i_{m}}}x=\partial_{\mathfrak{T}^{(n)}}\sum\limits_{n< i_{1} < \cdots < i_{m}, 0\leq m < \infty} P_{W_{i_{1} \cdots i_{m}}}x=\partial_{\mathfrak{T}^{(n)}}x=0.
\end{equation}
It is easy to see that $W_{i_{1} \cdots i_{m}}$ is a reduced subspace of $\partial_{\mathfrak{T}^{(n)}} $ , hence by $(\ref{ker mathcal T})$,
\begin{equation}
P_{W_{i_{1} \cdots i_{m}}} x \in \operatorname{ker} \partial_{\mathfrak{T}^{(n)}}, ~\forall n<i_{1}<\cdots<i_{m},~m\geq 0\label{e2.1}
\end{equation}
For fixed $m\geq 0$ and $ n<i_{1}<\cdots<i_{m}$, define unitary maps $U_{i_{1} \cdots i_{m}}: W_{i_{1} \cdots i_{m}} \rightarrow H \otimes \Lambda_n$ by
$$
U:\,f \otimes e_{j_{1}} \wedge \cdots \wedge e_{j_{k}} \wedge e_{i_{1}} \wedge \cdots \wedge e_{i_{m}} \mapsto f \otimes e_{j_{1}} \wedge \cdots \wedge e_{j_{k}},\quad\text{for }m\geq 1
$$ 
and for $m=0$ the map is the identity on $H\otimes\Lambda_n$. It is easy to see that
\begin{equation}\label{e2.2}
  U_{i_{1} \cdots i_{m}}^{-1} \partial_{T^{(n)}} U_{i_{1} \cdots i_{m}} x=\partial_{\mathfrak{T}^{(n)}} x, ~\forall x \in W_{i_{1} \cdots i_{m}}.
\end{equation}
Then by $(\ref{e2.1})$ and $(\ref{e2.2})$ we have $U_{i_{1} \cdots i_{m}} P_{W_{i_1 \cdot i_{m}}} x \in \operatorname{ker} \partial_{T^{(n)}}$. So for $m\geq 0$, from the fact that $T^{(n)}$ is invertible, there is a constant $c_{n}>0$ and $y_{i_{1} \cdots i_{m}}^{\prime} \in H\otimes \Lambda_n$ such that
$$
  U_{i_{1} \cdots i_{m}} P_{W_{i_{1}\cdots i_{m}}} x=\partial_{T^{(n)}} y_{i_{1} \cdots i_{m}}^{\prime} ~\text{and} ~
 \left\|y_{i_{1} \cdots i_{m}}^{\prime}\right\| \leq c_{n}\left\|U_{i_{1} \cdots i_{m}} P_{W_{i_{1} \cdots i_{m}} }x\right\|.
$$
 Hence by $(\ref{e2.2})$,
\begin{equation}\label{e2.3}
  \partial_{\mathfrak{T}^{(n)}} U_{i_{1} \cdots i_{m}}^{-1} y_{i_{1} \cdots i_{m}}^{\prime}=U_{i_{1} \cdots i_{m}}^{-1} \partial_{T^{(n)}} y_{i_{1} \cdots i_{m}}^{\prime}=U_{i_{1} \cdots i_{m}}^{-1} U_{i_{1} \cdots i_{m}} P_{W_{i_{1} \cdots i_{m}}} x=P_{W_{i_{1} \cdots i_{m}}} x,
\end{equation}
and
\begin{equation}\label{e2.4}
  \left\|U_{i_{1} \cdots i_{m}}^{-1} y_{i_{1} \cdots i_{m}}^{\prime}\right\|=\left\|y_{i_{1} \cdots i_{m}}^{\prime}\right\| \leq c_n\left\|U_{i_{1} \cdots i_{m}} P_{W_{i_{1} \cdots i_{m}}} x\right\|=c_n\left\|P_{W_{i_{1} \cdots i_{m}}} x\right\|.
\end{equation}
Let 
$$
y=\sum\limits_{n< i_{1} < \cdots < i_{m}, 0\leq m < \infty} U_{i_{1} \cdots i_{m}}^{-1} y_{i_{1} \cdots i_{m}}^{\prime},
$$
by $(\ref{e2.3})$, $(\ref{e2.4})$ and \eqref{x W decom}, we have $\|y\|\leq c_n\|x\|$ and $x=\partial_{\mathfrak{T}^{(n)}}y$, which implies that $$\ker \partial_{\mathfrak{T}^{(n)}}\subseteq \operatorname{Im}\partial_{\mathfrak{T}^{(n)}}.$$\\
$(\mathit{3})\Rightarrow(\mathit{1})$: From $(\mathit{3})$, we have $\operatorname{Im}\partial_{\mathfrak{T}^{(n)}}$ is closed. Notice that $\partial_{T^{(n)}}=\partial_{\mathfrak{T}^{(n)}}|_{H\otimes\Lambda_n}$, hence $\operatorname{Im}\partial_{T^{(n)}}$ is closed.
Suppose that $T^{(n)}$ is not invertible, then there is a non-negative integer $m\leq n$ such that $\operatorname{ker} \partial_{T^{(n)},m} \neq \operatorname{Im} \partial_{T^{(n)},m+1}$. Hence there exists a nonzero $x\in \operatorname{ker}\partial_{T^{(n)},m}\ominus\operatorname{Im}\partial_{T^{(n)},m+1}$. We claim that
\begin{equation}\label{Hk decom}
  \left(I_{H} \otimes E_{n+k}\right) x \in \operatorname{ker} \partial_{\mathfrak{T}^{(n)},m+1} \ominus \operatorname{Im} \partial_{\mathfrak{T}^{(n)} , m+2}, ~\forall k\geq 1.
\end{equation}
In fact, by $E_iE_j^*=-E_j^*E_i$ for $i\neq j$, we have
$$
\partial_{\mathfrak{T}^{(n)},m+1}\left(I_{H} \otimes E_{n+k}\right) x=-\left(I_{H} \otimes E_{n+k}\right) \partial_{\mathfrak{T}^{(n)},m}  x=-\left(I_{H} \otimes E_{n+k}\right) \partial_{T^{(n)},m}x=0.
$$
which implies that 
$$
\left(I_{H} \otimes E_{n+k}\right) x \in \operatorname{ker} \partial_{\mathfrak{T}^{(n)},m+1}.
$$
Since $x\perp \operatorname{Im}\partial_{T^{(n)},m+1}$, we have $\partial_{T^{(n)},m+1}^*x=0$. Notice that $$\partial_{T^{(n)},m+1}^*x=\left(\sum\limits_{i=1}^{n}T_i^*\otimes E_i\right)x=\partial_{\mathfrak{T}^{(n)},m+1}^*x.$$
Hence $$\partial_{\mathfrak{T}^{(n)},m+1}^*\left(I_{H} \otimes E_{n+k}\right)x=-\left(I_{H} \otimes E_{n+k}\right)\partial_{\mathfrak{T}^{(n)},m+1}^*x=0,$$ which implies that
$$
\left(I_{H} \otimes E_{n+k}\right) x\perp \operatorname{Im} \partial_{\mathfrak{T}^{(n)} , m+2}.
$$
The claim is proved. Since $E_{n+k} E_{n+j}^{*}+E_{n+j}^{*} E_{n+k}=0, \forall j \neq k, ~j,k\geq 1$, we have for $1\leq j<k$,
$$
\begin{aligned}
\left\langle\left(I_{H} \otimes E_{n+k}\right) x,\left(I_{H} \otimes E_{n+j}\right) x\right\rangle 
& =\left\langle\left(I_{H} \otimes E_{n+j}\right)^{*}\left(I_{H} \otimes E_{n+k}\right) x, x\right\rangle\\
& =-\left\langle\left(I_{H} \otimes E_{n+k}\right)\left(I_{H} \otimes E_{n+j}^{*}\right) x, x\right\rangle \\
& =0.
\end{aligned}
$$
So by \eqref{Hk decom}, $\operatorname{dim} \operatorname{ker} \partial_{\mathfrak{T}^{(n)}, m+1}/\operatorname{Im} \partial_{\mathfrak{T}^{(n)},m+2}=\infty$, which contradicts to $(\mathit{3})$.
\end{proof}
Let $\mathbb{T}$ be a row-bounded sequence of commuting operators. If there exists an $n$ such that
  $(T_1,\cdots,T_n)$ is Taylor invertible,  set
  \[
  \mathbb{T}^{(n)}=(T_1,\cdots,T_n,0,0,\cdots).
  \]
Then by Lemma \ref{Tn and T'}, \(\mathbb{T}^{(n)}\) is invertible. Hence by Proposition \ref{thm3.2}, \(\widehat{\mathbb{T}^{(n)}}\) is invertible. Since
\[
\partial_{\mathbb{T}^{(n)}}\partial_{\mathbb{T}^{(n)}}^* + \partial_{\mathbb{T}^{(n)}}^*\partial_{\mathbb{T}^{(n)}} = {\widehat{\mathbb{T}^{(n)}}}{}^*{\widehat{\mathbb{T}^{(n)}}} \oplus {\widehat{\mathbb{T}^{(n)}}}{\widehat{\mathbb{T}^{(n)}}}{}^*
\]
with respect to the orthogonal decomposition \(\Lambda(H) = \Lambda^+(H) \oplus \Lambda^-(H)\), it follows that \(\partial_{\mathbb{T}^{(n)}}\partial_{\mathbb{T}^{(n)}}^* + \partial_{\mathbb{T}^{(n)}}^*\partial_{\mathbb{T}^{(n)}}\) is invertible on \(\Lambda(H)\). For simplicity, set
  $$
     \Delta_{\mathbb{T}^{(n)}}=\partial_{\mathbb{T}^{(n)}}\partial_{\mathbb{T}^{(n)}}^*+\partial_{\mathbb{T}^{(n)}}^*\partial_{\mathbb{T}^{(n)}},
  $$
  hence for $p\geq 0$,
  \begin{equation}\label{DeltaTn}
     \Delta_{\mathbb{T}^{(n)},p}:=\Delta_{\mathbb{T}^{(n)}}|_{\Lambda^p(H)}=
     \partial_{\mathbb{T}^{(n)},{p+1}}\partial_{\mathbb{T}^{(n)},{p+1}}^*+\partial_{\mathbb{T}^{(n)},p}^*\partial_{\mathbb{T}^{(n)},p}.
  \end{equation}
  For $p<0$, write $\Lambda^p(H)=0$ and $\Delta_{\mathbb{T}^{(n)},p}=0$.
  Then we have the following Lemma.

\begin{lem}\label{N nilpotent}
  Let $\mathbb{T}$ be a row-bounded sequence of commuting operators, and suppose that there exists an $n$ such that
  $(T_1,\cdots,T_n)$ is Taylor invertible. For each $p\geq 0$, set
  \begin{equation}\label{def Np}
  N_p=(\partial_{\mathbb{T},p+1}-\partial_{\mathbb{T}^{(n)},p+1})\partial_{\mathbb{T}^{(n)},p+1}^*\Delta_{\mathbb{T}^{(n)},p}^{-1} + \partial_{\mathbb{T}^{(n)},p}^*\Delta_{\mathbb{T}^{(n)},p-1}^{-1} (\partial_{\mathbb{T},p}-\partial_{\mathbb{T}^{(n)},p})\in B(\Lambda^p(H)).
  \end{equation}
  Then $$N_p^{n+1}=0.$$
\end{lem}
\begin{proof}
  Let 
  $$
  \mathcal{E}_n=\operatorname{span}\{e_i:1\leq i\leq n\}.
  $$ 
  Then $\Lambda(H)$ can be decomposed as
  $$
  \Lambda(H)
  = H\otimes \Lambda[\mathcal{E}_n]\otimes \Lambda[\mathcal{E}_n^\perp]
  = \bigoplus_{r=0}^n H \otimes \Lambda^r[\mathcal{E}_n] \otimes \Lambda[\mathcal{E}_n^\perp],
  $$
  where $\Lambda[\mathcal{E}_n]$ (resp. $\Lambda[\mathcal{E}_n^\perp]$) is the antisymmetric Fock space over $\mathcal{E}_n$ (resp. $\mathcal{E}_n^\perp$), and $\Lambda^r[\mathcal{E}_n]$ is the antisymmetric  $r$-fold tensor product of $\mathcal{E}_n$ with itself for $0\leq r\leq n$.
For \(p, r \in \mathbb{Z}\), define
\begin{equation}\label{Kr<0}
  \mathcal{K}_{r,p}=
  \begin{cases}
    (H \otimes \Lambda^r[\mathcal{E}_n] \otimes \Lambda[\mathcal{E}_n^\perp])\cap \Lambda^p(H), & p \ge 0 \text{ and } 0 \le r \le n, \\[1ex]
    \{0\}, & p<0 \text{ or } r<0 \text{ or } r>n.
  \end{cases}
\end{equation}
  Then it is easy to see that for each $p\geq 0$ and $r\in\mathbb{Z}$,
  \begin{equation}\label{basic K}
  \partial_{\mathbb{T}^{(n)},p}(\mathcal{K}_{r,p})\subseteq \mathcal{K}_{r-1,p-1},~\partial_{\mathbb{T}^{(n)},p}^*(\mathcal{K}_{r,p-1})\subseteq \mathcal{K}_{r+1,p}.
  \end{equation}
  Hence by \eqref{DeltaTn} and \eqref{basic K}, $\Delta_{\mathbb{T}^{(n)},p}\mathcal{K}_{r,p}\subset \mathcal{K}_{r,p}$.
  Then $P_{\mathcal{K}_{r,p}}\Delta_{\mathbb{T}^{(n)},p}=\Delta_{\mathbb{T}^{(n)},p}P_{\mathcal{K}_{r,p}}$, which implies that
  \begin{equation}\label{d invert K}
    \Delta_{\mathbb{T}^{(n)},p}^{-1}\mathcal{K}_{r,p}\subseteq\mathcal{K}_{r,p}.
  \end{equation}
  By the definition of $\mathcal{K}_{r,p}$,
  $$
   (\partial_{\mathbb{T},p}-\partial_{\mathbb{T}^{(n)},p})(\mathcal{K}_{r,p})=\left(\sum_{i=n+1}^{\infty}T_i\otimes E_i^*\right)|_{\Lambda^p(H)}(\mathcal{K}_{r,p})\subseteq \mathcal{K}_{r,p-1}.
  $$
  Thus by \eqref{def Np}, \eqref{basic K}and \eqref{d invert K}, $N_p\mathcal{K}_{r,p}\subseteq \mathcal{K}_{r+1,p}$. Then by \eqref{Kr<0}, for $0\leq r\leq n$,
  $$
  N_p^{n+1}\mathcal{K}_{r,p}=\{0\}.
  $$
  It follows from $\Lambda^p(H)=\bigoplus\limits_{r=0}^n\mathcal{K}_{r,p}$ that $N_p^{n+1}=0$.
\end{proof}

\begin{lem}\label{finite invertible to I invertible}
  Let $\mathbb{T}$ be a row-bounded sequence of commuting operators. If there exists an $n$ such that
  $(T_1,\cdots,T_n)$ is Taylor invertible, then $K(\mathbb{T})$ is exact.
  Moreover, under the same hypotheses, if $\partial_{\mathbb{T}}$ is bounded on
  $\Lambda(H)$, then $\mathbb{T}$ is invertible.
\end{lem}

\begin{proof}
  Notice that
  \begin{equation}\label{partial delta commu}
  \begin{aligned}
  \partial_{\mathbb{T}^{(n)},p}\Delta_{\mathbb{T}^{(n)},p}
  &=\partial_{\mathbb{T}^{(n)},p}(\partial_{\mathbb{T}^{(n)},{p+1}}\partial_{\mathbb{T}^{(n)},{p+1}}^*+\partial_{\mathbb{T}^{(n)},p}^*\partial_{\mathbb{T}^{(n)},p})\\
  &=\partial_{\mathbb{T}^{(n)},p}\partial_{\mathbb{T}^{(n)},p}^*\partial_{\mathbb{T}^{(n)},p}\\
  &=(\partial_{\mathbb{T}^{(n)},p}\partial_{\mathbb{T}^{(n)},p}^*+\partial_{\mathbb{T}^{(n)},p-1}^*\partial_{\mathbb{T}^{(n)},p-1})\partial_{\mathbb{T}^{(n)},p}\\
  &=\Delta_{\mathbb{T}^{(n)},p-1}\partial_{\mathbb{T}^{(n)},p}.
  \end{aligned}
  \end{equation}
  Hence $\partial_{\mathbb{T}^{(n)},p}\Delta_{\mathbb{T}^{(n)},p}^{-1}=\Delta_{\mathbb{T}^{(n)},p-1}^{-1}\partial_{\mathbb{T}^{(n)},p}$. Then
  \begin{equation}\label{dna+adn=I}
  \begin{aligned}
  &\partial_{\mathbb{T}^{(n)},p+1}\partial_{\mathbb{T}^{(n)},p+1}^*\Delta_{\mathbb{T}^{(n)},p}^{-1}+\partial_{\mathbb{T}^{(n)},p}^*\Delta_{\mathbb{T}^{(n)},p-1}^{-1} \partial_{\mathbb{T}^{(n)},p}\\
  =&\partial_{\mathbb{T}^{(n)},p+1}\partial_{\mathbb{T}^{(n)},p+1}^*\Delta_{\mathbb{T}^{(n)},p}^{-1}+\partial_{\mathbb{T}^{(n)},p}^*\partial_{\mathbb{T}^{(n)},p}\Delta_{\mathbb{T}^{(n)},p}^{-1} \\
  =&\Delta_{\mathbb{T}^{(n)},p}\Delta_{\mathbb{T}^{(n)},p}^{-1}=I_{\Lambda^p(H)}.
  \end{aligned}
  \end{equation}
  For simplicity, write $A_p=\partial_{\mathbb{T}^{(n)},p+1}^*\Delta_{\mathbb{T}^{(n)},p}^{-1}$. By \eqref{dna+adn=I},
  \begin{equation}\label{da+ad=I+N1}
  \begin{aligned}
  \partial_{\mathbb{T},p+1} A_p + A_{p-1} \partial_{\mathbb{T},p} 
  &= (\partial_{\mathbb{T}^{(n)},p+1}+\partial_{\mathbb{T},p+1}-\partial_{\mathbb{T}^{(n)},p+1})A_p + A_{p-1}(\partial_{\mathbb{T}^{(n)},p}+\partial_{\mathbb{T},p}-\partial_{\mathbb{T}^{(n)},p}) \\
  &= (\partial_{\mathbb{T}^{(n)},p+1}A_p+A_{p-1}\partial_{\mathbb{T}^{(n)},p})+(\partial_{\mathbb{T},p+1}-\partial_{\mathbb{T}^{(n)},p+1})A_p+ A_{p-1}(\partial_{\mathbb{T},p}-\partial_{\mathbb{T}^{(n)},p})\\
  &= I_{\Lambda^p(H)} + N_p .
  \end{aligned}
  \end{equation}
  Since  $\partial_{\mathbb{T},p}\partial_{\mathbb{T},p+1}=0$, by \eqref{da+ad=I+N1},
  \begin{equation}\label{partialT I+N commute} 
  \begin{aligned}
  \partial_{\mathbb{T},p}(I_{\Lambda^p(H)} + N_p)
  &=\partial_{\mathbb{T},p}A_{p-1}\partial_{\mathbb{T},p}\\
  &=(\partial_{\mathbb{T},p}A_{p-1}+A_{p-2}\partial_{\mathbb{T},p-1})\partial_{\mathbb{T},p}\\
  &=(I_{\Lambda^{p-1}(H)} + N_{p-1})\partial_{\mathbb{T},p} .
  \end{aligned}
   \end{equation}
  By Lemma \ref{N nilpotent}, $N_p^{n+1}=0$. hence 
  \begin{equation}\label{I+Nsum=I}
  (I_{\Lambda^p(H)} + N_p)\sum_{j=0}^{n}(-N_p)^j=\sum_{j=0}^{n}(-N_p)^j(I_{\Lambda^p(H)} + N_p)=I_{\Lambda^p(H)}.
  \end{equation}
  Then by \eqref{partialT I+N commute} and \eqref{I+Nsum=I},
  \begin{equation}\label{sum Np commute}
  \begin{aligned}
  \partial_{\mathbb{T},p}\sum_{j=0}^{n}(-N_{p})^j
  =&\sum_{j=0}^{n}(-N_{p-1})^j(I_{\Lambda^{p-1}(H)}+N_{p-1})\partial_{\mathbb{T},p}\sum_{j=0}^{n}(-N_{p})^j\\
  =&\sum_{j=0}^{n}(-N_{p-1})^j\partial_{\mathbb{T},p}(I_{\Lambda^{p}(H)}+N_{p})\sum_{j=0}^{n}(-N_{p})^j\\
  =&\sum_{j=0}^{n}(-N_{p-1})^j\partial_{\mathbb{T},p}.
  \end{aligned}
  \end{equation}
  For $p\geq 0$, define $B_p:\Lambda^{p}(H)\to \Lambda^{p+1}(H)$ by
  $
  B_{p} = A_p\sum\limits_{j=0}^{n}(-N_p)^j .
  $
  It follows from \eqref{I+Nsum=I} and \eqref{sum Np commute} that
  \[
  \partial_{\mathbb{T},p+1}B_p+B_{p-1}\partial_{\mathbb{T},p}
  = (\partial_{\mathbb{T},p+1}A_p+A_{p-1}\partial_{\mathbb{T},p})\sum_{j=0}^{n}(-N_p)^j\\
  \overset{(*)}{=} I_{\Lambda^p(H)},
  \]
  where $(*)$ follows from \eqref{da+ad=I+N1} and \eqref{I+Nsum=I}. Thus for any $p\geq 0$ and $x\in \ker\partial_{\mathbb{T},p}$, we have
  $$
  x=(\partial_{\mathbb{T},p+1}B_p+B_{p-1}\partial_{\mathbb{T},p})x=\partial_{\mathbb{T},p+1}(B_px).
  $$
  Hence $\ker \partial_{\mathbb{T},p} = \operatorname{Im} \partial_{\mathbb{T},p+1}$ for every $p\ge 0$, i.e.\
  $K(\mathbb{T})$ is exact.

  Moreover, if $\partial_{\mathbb{T}}$ is bounded, then by \eqref{def Np}
  $$
  \sup_{p\geq0} \|N_p\|
  \leq 2\|\partial_{\mathbb{T}}-\partial_{\mathbb{T}^{(n)}}\|\,\|\partial_{\mathbb{T}^{(n)}}^*\|\sup_{p\geq 0}\|\Delta_{\mathbb{T}^{(n)},p}^{-1}\|
  \leq 2\|\partial_{\mathbb{T}}-\partial_{\mathbb{T}^{(n)}}\|\,\|\partial_{\mathbb{T}^{(n)}}^*\|\|\Delta_{\mathbb{T}^{(n)}}^{-1}\|
  <\infty
  $$
  Hence 
  $$
  \begin{aligned}
  \sup_{p\geq0}\|B_p\|
  &\leq \sup_{p\geq0}\|\partial_{\mathbb{T}^{(n)},p+1}^*\|\|\Delta_{\mathbb{T}^{(n)},p}^{-1}\|\left\|\sum_{j=0}^{n}(-N_{p})^j\right\|\\
  &\leq \|\partial_{\mathbb{T}^{(n)}}^*\|\|\Delta_{\mathbb{T}^{(n)}}^{-1}\|\sum_{j=0}^{n}\sup_{p\geq0}\|N_{p}\|^j
  <\infty,
  \end{aligned}
  $$
  which implies that  $\|B_px\|\leq c\|x\|$ for some $c>0$ (independent of $x$). Therefore
  $\mathbb{T}$ is invertible.
\end{proof}

It is worth noting that for a row-bounded sequence $\mathbb{T}$ with $\|\partial_{\mathbb{T}}\|=\infty$, the exactness of the Koszul complex $K(\mathbb{T})$ does not necessarily imply that $\mathbb{T}$ is invertible. For a bounded open set $\Omega\subset \mathbb C^2$, define
\[
W^{1,2}(\Omega)
=
\{f\in L^2(\Omega):\ 
\partial_{z_i}f,\ \partial_{\bar z_i}f\in L^2(\Omega),\ i=1,2\},
\]
and
\[
W^{1,\infty}(\Omega)
=
\{f\in L^\infty(\Omega):\ 
\partial_{z_i}f,\ \partial_{\bar z_i}f\in L^\infty(\Omega),\ i=1,2\}.
\]
\begin{exam}
Let $\Omega = \{z = (z_1, z_2) \in \mathbb{C}^2 : 1 < |z| < 2\}$, $X = W^{1,2}(\Omega)$, and $Y = L^2(\Omega; \mathbb{C}^2)$. Define the bounded operator $B: X \to Y$ by $Bf = \bar{\partial}f = (\frac{\partial f}{\partial \bar{z}_1}, \frac{\partial f}{\partial \bar{z}_2})$. Then
\[
B(z_i f) = z_i Bf, \qquad B1 = 0.
\]
Set $H = X \oplus \Big(\bigoplus\limits_{j \ge 1} Y_j\Big)$ with $Y_j = Y$, and write elements as $(f, u_1, u_2, \cdots)$. For $1\leq i\leq 2$, define $A_i:H\to H$ by
\[
A_i(f, u_1, u_2, \cdots) = (z_i f, z_i u_1, z_i u_2, \cdots).
\]
For $j\geq 1$, define $R_j:H\to H$ by
\[
R_j(f, u_1, u_2, \cdots) = \left(0, \cdots, \underbrace{\frac{1}{\sqrt{j}} Bf}_{Y_j}, 0, \cdots \right).
\]
Consider $\mathbb{T} = (A_1, A_2, R_1, R_2, \cdots)$. Since $A_1 A_2 = A_2 A_1$, $A_i R_j = R_j A_i$, and $R_j R_k = 0$, $\mathbb{T}$ is commuting. Moreover, $R_j R_j^* = \operatorname{diag}\left(0, \cdots, \underbrace{\frac{1}{j} B B^*}_{Y_j}, \cdots \right)$, so
\[
\left\| \sum_{j > N} R_j R_j^* \right\| = \frac{\|B\|^2}{N + 1} \to 0.
\]
Thus $\mathbb{T}$ is row-bounded.

We first show that $(A_1, A_2)$ is Taylor invertible. Let $b_i(z) = \frac{\bar{z}_i}{|z_1|^2 + |z_2|^2}$ for $i = 1, 2$, then $b_i \in C^1(\Omega)\cap W^{1,\infty}(\Omega)$. Hence
\[
b_i,\ \partial_{z_k}b_i,\ \partial_{\bar z_k}b_i\in L^\infty(\Omega),
\qquad k=1,2.
\]
By the Leibniz rule, the multiplication operator \(M_{b_i}\) is bounded on
\(W^{1,2}(\Omega)\) and \(L^2(\Omega)\). Define $C_i:H\to H$ by
$$
C_i(f, u_1, u_2, \cdots) = (b_i f, b_i u_1, b_i u_2, \cdots).
$$
Then it is easy to see that $A_i C_j = C_j A_i$ and $ A_1 C_1 + A_2 C_2 = I_H$, which implies that $(A_1, A_2)$ is Taylor invertible. By Lemma \ref{finite invertible to I invertible}, the Koszul complex $K(\mathbb T)$ is exact.

Next, we prove that $\mathbb{T}$ is not Taylor invertible. For $N\geq 3$, Set $\omega_N = e_3 \wedge \cdots \wedge e_{N+2}$ and $x_N = (1, 0, 0, \cdots) \otimes \omega_N$. Since $B1 = 0$, we have $\partial_{\mathbb{T},N} x_N = 0$ and $\|x_N\| = \|1\|_X$. 

Assume \(\mathbb{T}\) is Taylor invertible. Then there exist a constant \(C > 0\) independent of \(N\) and \(y_N \in \Lambda^{N+1}(H)\) such that
\begin{equation}\label{dyneqxn}
\partial_{\mathbb{T},N+1} y_N = x_N, \qquad \|y_N\| \le C \|1\|_X.
\end{equation}
Let \(P_X:H\to X\) and \(P_{Y_j}:H\to Y_j\) be the orthogonal projections. 
For a unit vector \(\zeta\in\Lambda\), let \(p_\zeta:\Lambda\to\mathbb C\zeta\) be the orthogonal projection defined by \(p_\zeta(\xi)=\langle\xi,\zeta\rangle\zeta\), and set \(P_\zeta = I_H\otimes p_\zeta\). 
For \(i=1,2\), write
\[
h_{i,N}\otimes e_i\wedge \omega_N
=
P_{e_i\wedge \omega_N}y_N,
\]
and set
$
f_{i,N}=P_X h_{i,N}\in X.
$
Applying \(P_X\otimes p_{\omega_N}\) to 
\(\partial_{\mathbb T,N+1}y_N=x_N\) gives
\begin{equation}\label{zf1+zf2=1}
z_1 f_{1,N}+z_2 f_{2,N}=1.
\end{equation}
Moreover, by \eqref{dyneqxn},
\begin{equation}\label{f1+f2<1}
\|f_{1,N}\|_X^2 + \|f_{2,N}\|_X^2 \le \|y_N\|^2 \le C^2 \|1\|_X^2.
\end{equation}
For \(1 \le j \le N\), let \(\omega_{N,j} = E_{j+2}^* \omega_N\). Applying 
\(P_{Y_j}\otimes p_{e_1\wedge \omega_{N,j}}\) and 
\(P_{Y_j}\otimes p_{e_2\wedge \omega_{N,j}}\) to 
\(\partial_{\mathbb{T},N+1} y_N = x_N\), respectively, gives
\begin{equation}\label{Bf-zg=0}
-\frac{1}{\sqrt{j}} B f_{1,N} - z_2 g_{j,N} = 0, \qquad -\frac{1}{\sqrt{j}} B f_{2,N} + z_1 g_{j,N} = 0,
\end{equation}
where \(g_{j,N} \in Y_j\) is determined by
\[
g_{j,N}\otimes e_1\wedge e_2\wedge \omega_{N,j}
=
\bigl(P_{Y_j}\otimes p_{e_1\wedge e_2\wedge \omega_{N,j}}\bigr)y_N.
\]
Hence by \eqref{Bf-zg=0} and \eqref{f1+f2<1},
\[
\begin{aligned}
\left( \sum_{j=1}^{N} \frac{1}{j} \right) (\|B f_{1,N}\|_Y^2 + \|B f_{2,N}\|_Y^2)
&=\sum_{j=1}^{N} \int_\Omega |z|^2 |g_{j,N}(z)|^2 dV(z)\\
&\leq  \sum_{j=1}^{N}  4 \|g_{j,N}\|_Y^2 \\
&\le 4 \|y_N\|^2 \le 4 C^2 \|1\|_X^2,
\end{aligned}
\]
which implies that 
\begin{equation}\label{bf1=bf2=0}
\lim\limits_{N\to\infty} B f_{1,N}=\lim\limits_{N\to\infty} B f_{2,N}=0
\end{equation}
By \eqref{f1+f2<1}, we can take weakly convergent subsequences $f_{i, N_\nu} \rightharpoonup f_i$ in $X$ for \(i = 1, 2\). Since \(B\) is bounded, by \eqref{bf1=bf2=0} we have
\[
\bar\partial f_i= B f_i= 0, ~\forall ~i=1,2.
\]
Thus $f_1, f_2$ are holomorphic on $\Omega$. Moreover by \eqref{zf1+zf2=1}, $z_1f_1+z_2f_2=1$ on $\Omega$. Then by Hartogs' extension theorem it is a contradiction. Therefore, \(\mathbb{T}\) is not invertible.

\end{exam}

\subsection{Fundamentals on $\sigma(\mathbb{T})$}
\begin{thm}\label{bounded closed}
  Let $\mathbb{T}$ be a sequence of commuting operators on $H$ satisfying 
  \[
  \sum_{i=1}^{\infty}T_iT_i^* \le C^2 I
  \]
  for some $C>0$. Then $\sigma(\mathbb{T})$ is a bounded closed set. Moreover, 
  \[
  \sigma(\mathbb{T}) \subset C\overline{\mathbb{B}}.
  \]
\end{thm}
Before proving Theorem $\ref{bounded closed}$, we need some lemmas. For a row-bounded sequence of commuting operators $\mathbb{T}$, define
\[
\sigma_r(\mathbb T)
=\bigl\{\lambda\in\ell^2:\operatorname{Im}\partial_{\mathbb T-\lambda,1}
\ne\Lambda^0(H)\bigr\}.
\]
Then $\sigma_r(\mathbb{T})\subset\sigma(\mathbb{T})$. We have the following lemma.
\begin{lem}\label{normal contrac}
  Let $\mathbb{T}$ be a row-bounded sequence of commuting operators on $H$. Suppose that $T_{i} T_{j}^{*}=T_{j}^{*} T_{i}$, $\forall i, j\geq 1$. Then $\sigma(\mathbb{T})=\sigma_{r}(\mathbb{T})$.
\end{lem}
\begin{proof}
it suffices to prove that $\sigma(\mathbb{T}) \subseteq \sigma_{r}(\mathbb{T})$. Suppose that 
$\lambda \notin \sigma_{r}(\mathbb{T}),$ which implies that
$\partial_{\mathbb{T}-\lambda, 1} \partial_{\mathbb{T}-\lambda, 1}^{*}$ is invertible, then there is a constant $c>0$, such that for $ x \in H,$
\begin{equation}\label{ddstarx}
\left\langle\partial_{\mathbb{T}-\lambda, 1} \partial_{\mathbb{T}-\lambda, 1}^{*} x, x\right\rangle=\left\|\partial_{\mathbb{T}-\lambda, 1}^{*} x\right\|^{2} \geq c\|x\|^{2}.
\end{equation}
For $n \geq 0$, let 
$$
y=\sum\limits_{1 \leq i_{1}<\cdots<i_{n}} f_{i_{1} \cdots i_{n}} \otimes e_{i_{1}} \wedge \cdots \wedge e_{i_{n}} \in \operatorname{ker} \partial_{\mathbb{T}-\lambda, n}.
$$
Notice that for \(n \geq 1\),
\begin{equation}\label{all n eq 1}
\begin{aligned}
& \partial_{\mathbb{T}-\lambda, n}^{*} \partial_{\mathbb{T}-\lambda, n}+\partial_{\mathbb{T}-\lambda, n+1} \partial_{\mathbb{T}-\lambda, n+1}^{*} \\
= & \left(\sum_{i=1}^{\infty} \left(T_{i}^{*}-\bar{\lambda}_{i}\right) \otimes E_{i}\right)\left(\sum_{j=1}^{\infty} \left(T_{j}-\lambda_{j}\right) \otimes E_{j}^{*}\right)+\left(\sum_{j=1}^{\infty}\left( T_{j}-\lambda_{j}\right) \otimes E_{j}^{*}\right)\left(\sum_{i=1}^{\infty}\left(T_{i}^{*}-\bar{\lambda}_{i}\right) \otimes E_{i}\right)\bigg|_{\Lambda^n(H)} \\
= & \left(\sum_{i=1}^{\infty}\left(T_{i}^{*}-\bar{\lambda}_{i}\right)\left(T_{i}-\lambda_{i}\right) \otimes E_{i} E_{i}^{*}+\sum_{i \neq j}\left(T_{i}^{*}-\bar{\lambda}_{i}\right)\left(T_{j}-\lambda_{j}\right) \otimes E_{i} E_{j}^{*}\right.\\&
\left.+  \sum_{i=1}^{\infty}\left(T_{i}-\lambda_{i}\right)\left(T_{i}^{*}-\bar{\lambda}_{i}\right) \otimes E_{i}^{*} E_{i}+\sum_{i \neq j}\left(T_{j}-\lambda_{j}\right)\left(T_{i}^{*}-\bar{\lambda}_{i}\right) \otimes E_{j}^{*} E_{i}\right)\bigg|_{\Lambda^n(H)} \\
= & \sum_{i=1}^{\infty}\left(T_{i}-\lambda_{i}\right)\left(T_{i}^{*}-\bar{\lambda}_{i}\right) \otimes\left(E_{i} E_{i}^{*}+E_{i}^{*} E_{i}\right)+\sum_{i \neq j}\left(T_{i}^{*}-\bar{\lambda}_{i}\right)\left(T_{j}-\lambda_{j}\right) \otimes\left(E_{i} E_{j}^{*}+E_{j}^{*} E_{i}\right)\bigg|_{\Lambda^n(H)} \\
= & \sum_{i=1}^{\infty}\left(T_{i}-\lambda_{i}\right)\left(T_{i}^{*}-\bar{\lambda}_{i}\right) \otimes I\bigg|_{\Lambda^n(H)} \\
= & \partial_{\mathbb{T}-\lambda, 1} \partial_{\mathbb{T}-\lambda, 1}^{*} \otimes I\bigg|_{\Lambda^n(H)}.
\end{aligned}
\end{equation}
Hence By \eqref{all n eq 1} and \eqref{ddstarx},
$$
\begin{aligned}
&\left\|\partial_{\mathbb{T}-\lambda, n+1}^{*} y\right\|^{2} \\& =\left\langle\partial_{\mathbb{T}-\lambda, n+1}^{*} y, \partial_{\mathbb{T}-\lambda, n+1}^{*} y\right\rangle \\
& =\left\langle\left(\partial_{\mathbb{T}-\lambda, n}^{*} \partial_{\mathbb{T}-\lambda, n}+\partial_{\mathbb{T}-\lambda, n+1} \partial_{\mathbb{T}-\lambda, n+1}^{*}\right) y, y\right\rangle \\
& =\left\langle\left(\partial_{\mathbb{T}-\lambda, 1} \partial_{\mathbb{T}-\lambda, 1}^{*} \otimes I\right) \sum_{1\leq i_{1}<\cdots<i_{n}} f_{i_{1} \cdots i_{n}} \otimes e_{i_{1}} \wedge \cdots \wedge e_{i_{n}}, \sum_{1\leq i_{1}<\cdots<i_{n}} f_{i_{1} \cdots i_{n}} \otimes e_{i_{1}} \wedge \cdots \wedge e_{i_{n}}\right\rangle \\
& =\sum_{1\leq i_{1}<\cdots<i_{n}}\left\langle\partial_{\mathbb{T}-\lambda, 1} \partial_{\mathbb{T}-\lambda, 1}^{*} f_{i_{1} \cdots i_{n}}, f_{i_{1} \cdots i_{n}}\right\rangle \\
& \geq \sum_{1\leq i_{1}<\cdots<i_{n}} c\left\|f_{i_{1} \cdots i_{n}}\right\|^{2} \\
& =c\|y\|^{2}.
\end{aligned}
$$
which implies that $\mathbb{T}-\lambda$ is invertible, i.e. $\lambda\notin \sigma(\mathbb{T})$.
\end{proof}

For a reproducing kernel Hilbert space \(\mathcal{H}\), let \(\mathrm{Mult}(\mathcal{H})\) denote the multiplier algebra of \(\mathcal{H}\).
For \(t >0\), let \(A_{n,t}^{2}\) be the reproducing kernel Hilbert space on \(\mathbb{B}_{n}\) with kernel \(K(z,w)=\frac{1}{(1-\left\langle z,w\right\rangle)^t}\). It is not difficult to see that
\begin{align}\label{normAnt}
\|z^{\alpha}\|_{A_{n,t}^{2}}^{2}=\frac{\alpha!\,\Gamma(t)}{\Gamma(|\alpha|+t)}.
\end{align}
Let $H_{n}^{2}$ be the Arveson space in $n$ variables. 
Notice that $H_{n}^2=A_{n,1}^{2}$. 
Then we have the following Lemma.
\begin{lem}\label{embed}
  Let \(n \geq 1\). If $f \in \mathrm{Mult}\left(H_{n}^{2}\right)$, then $f\in \mathrm{Mult}\left(H^{2}\right)$ and \(\|f\|_{\mathrm{Mult}(H^2)}= \|f\|_{\mathrm{Mult}(H_n^2)}\).
\end{lem}
\begin{proof}
	For any \(m \geq n\), by \cite[Lemma 6.2]{AHMR2022},  we have \(f \in \mathrm{Mult}(H_{m}^2)\) and 
	\begin{equation}\label{+meq}
	\|f\|_{\mathrm{Mult}(H_{m}^2)} = \|f\|_{\mathrm{Mult}(H_{n}^2)}.
	\end{equation}
	For every $g\in\bigcup\limits_{m=n}^{\infty}H_{m}^2$, choose $m\geq n$ such that
	$g\in H_m^2$. Then, by \eqref{+meq},
	\begin{align*}
		\|fg\|_{H^2}=
		\|fg\|_{H_m^2}
		\leq \|f\|_{\mathrm{Mult}(H_m^2)}\|g\|_{H_m^2}
		=\|f\|_{\mathrm{Mult}(H_n^2)}\|g\|_{H^2}.
	\end{align*}
	Notice that \(\bigcup\limits_{m=n}^{\infty}H_{m}^2\) is dense in \(H^2\).
	Then \(f\in\mathrm{Mult}(H^2)\) and \(\|f\|_{\mathrm{Mult}(H^2)}
	\leq\|f\|_{\mathrm{Mult}(H_n^2)}\). 
	It is easy to see that \(\|f\|_{\mathrm{Mult}(H^2)}
	\geq\|f\|_{\mathrm{Mult}(H_n^2)}\). Therefore, \(\|f\|_{\mathrm{Mult}(H^2)}
	=\|f\|_{\mathrm{Mult}(H_n^2)}\). 
\end{proof}

With the preparations above, we now prove Theorem \ref{bounded closed}.
\vskip1.5mm
\noindent$Proof~of~Theorem~\ref{bounded closed}.$  By Corollary $\ref{invert continious}$, we have $\sigma(\mathbb{T})$ is closed. Fix $\varepsilon>0$, and write $\mathbb{T}_{\varepsilon}=(\frac{1}{C+\varepsilon}T_1,\frac{1}{C+\varepsilon}T_2,\cdots)$. Then by Definition 1.2, it's easy to verify that
\begin{equation}\label{si T eq si T e}
\sigma\left(\mathbb{T}\right)=(C+\varepsilon) \sigma(\mathbb{T}_{\varepsilon}).
\end{equation}
Let $\phi_{\varepsilon}(A)=\sum\limits_{i=1}^\infty \left(\frac{1}{C+\varepsilon}T_i\right)A\left(\frac{1}{C+\varepsilon}T_i\right)^*$ for $A\in B(H)$. Then
$$
\phi_{\varepsilon}^{n}(I)=\phi_{\varepsilon}^{n-1}\left(\sum_{i=1}^{\infty} \frac{1}{(C+\varepsilon)^2}T_{i} T_{i}^{*}\right) \leq \phi_{\varepsilon}^{n-1}(\frac{C^2}{(C+\varepsilon)^2}I) = \frac{C^2}{(C+\varepsilon)^2} \phi_{\varepsilon}^{n-1}(I).
$$
So $\lim\limits _{n \rightarrow \infty} \phi_{\varepsilon}^{n}(I)=0$. By \cite{curv}, $\mathbb{T}_{\varepsilon}$ is unitarily equivalent to the row-bounded sequence 
$$
\mathbb{S}=\left(S_{z_1}^{\mathcal{M}^{\perp}}, S_{z_2}^{\mathcal{M}^{\perp}}, \cdots\right),
$$
which acts on a quotient module $\mathcal{M}^{\perp}$ of $H^2\otimes E$.  We claim that $\sigma(\mathbb{S}) \subseteq \bar{\mathbb{B}}$. In fact, for $ \lambda=\left(\lambda_{1}, \lambda_{2}, \cdots\right) \in l^{2} \backslash \bar{\mathbb{B}}$, there is an $ N \in \mathbb{N} $ such that $\sum\limits_{i=1}^{N}\left|\lambda_{i}\right|^{2}>1$. By \cite[Theorem~1]{CSW}, there exist $f_1,\cdots,f_N\in \mathrm{Mult}\left(H_{N}^{2} \right)$, such that
$$
\sum_{i=1}^{N}(M_{f_i}\otimes I)(M_{z_i-\lambda_i}\otimes I)=I_{H_{N}^{2} \otimes E}.
$$
By Lemma $\ref{embed}$,
$$
\sum_{i=1}^{N}(M_{f_i}\otimes I)(M_{z_i-\lambda_i}\otimes I)=I_{H^{2} \otimes E}.
$$
Therefore
$$
\sum_{i=1}^{N}P_{\mathcal{M}^\perp}(M_{f_i}\otimes I)P_{\mathcal{M}^\perp}(S_{z_i}^{\mathcal{M}^{\perp}}-\lambda_iI_{\mathcal{M}^\perp})=I_{\mathcal{M}^\perp}.
$$
By Corollary  $\ref{TS=1}$, $\lambda \notin \sigma(\mathbb{S})$, which implies that $\sigma(\mathbb{S})\subseteq \overline{\mathbb{B}}$. The claim is proved. Hence
$$
\sigma(\mathbb{T}_{\varepsilon})\subseteq\sigma(\mathbb{S})\subseteq \overline{\mathbb{B}}.
$$
Then by \eqref{si T eq si T e}, we have
$$
\sigma(\mathbb{T}) = \bigcap_{n=1}^{\infty}\left(C+\frac{1}{n}\right)\sigma(\mathbb{T}_{\frac{1}{n}}) \subseteq \bigcap_{n=1}^{\infty}\left(C+\frac{1}{n}\right) \cdot \bar{\mathbb{B}}=C \cdot \bar{\mathbb{B}}.
$$
$\hfill\square$
\begin{rem}\label{rem31}
\begin{itemize}
\item [(1)]It is worth noting that $\sigma(\mathbb{T})$ and $\sigma_{SZ}(\mathbb{T})$ are not always equal. For example, let $H$ be an infinite-dimensional separable Hilbert space, $\left\{e_{i}\right\}_{i \geq 1}$ be an orthonormal basis of $H$. Let 
$$
T_{i} x=\left\langle x, e_{i}\right\rangle e_{i},~ \forall~ i \geq 1, ~x\in H.
$$
Then $\sigma_{SZ}(\mathbb{T}) \neq \sigma(\mathbb{T})$. In fact, obviously $\left(T_{1}, \cdots, T_{n}\right)$ is not invertible for all $n\geq 1$, which implies that $0\in\sigma_{SZ}(\mathbb{T})$. Since $T_{i} T_{j}^{*}=T_{j}^{*} T_{i}, \forall i, j$ and $\partial_{\mathbb{T}, 1} \partial_{\mathbb{T}, 1}^{*}=\sum\limits_{k=1}^{\infty} T_{k} T_{k}^{*}=I$, by Lemma $\ref{normal contrac}$, $T$ is invertible, which implies that $0\notin \sigma(\mathbb{T})$.
  \item [(2)]  For every nonempty bounded closed subset $E$ of $l^{2}$, there exists a row-bounded sequence of commuting operators $\mathbb{T}$ such that $\sigma(\mathbb{T})=E$. In fact, for a bounded subset $E$ of $l^{2}$, there exists a constant $C>0$ such that $E \subseteq C \cdot \bar{\mathbb{B}}$. Let 
\begin{equation}\label{example 2.7.2}
  H=L^{2}(E, d \mu) ~\text{and}~ \mathbb{T}=\left(M_{z_{1}}, M_{z_{2}}, \cdots\right),
\end{equation}
where $\mu$ is the counting measure. For $f \in H$, we have
$$
\left\langle\sum_{i=1}^{\infty} M_{z_{i}} M_{z_{i}}^{*} f, f\right\rangle=\sum_{i=1}^{\infty} \int_{E}\left|\bar{z}_{i} f\right|^{2} d \mu=\int_{E} \sum_{i=1}^{\infty}\left|\bar{z}_{i} f\right|^{2} d \mu=\int_{E}|f|^{2} \sum_{i=1}^{\infty}\left|\bar{z}_{i}\right|^{2} d \mu \leq C^2\left\langle f, f\right\rangle,
$$
so $\mathbb{T}=\left(M_{z_{1}}, M_{z_{2}}, \cdots\right)$ is a row-bounded sequence. Notice that $M_{z_{i}} M_{z_{j}}^{*}=M_{z_{j}}^{*} M_{z_{i}}, \forall i, j$, by Lemma \ref{normal contrac}, we have $\sigma(\mathbb{T})=\sigma_{r}(\mathbb{T})$. Since $E$ is closed and $\mu$ is the counting measure, we have
$$
\begin{aligned}
\lambda \in \sigma_{r}(\mathbb{T}) & \Leftrightarrow 0 \in \sigma\left(\sum_{i=1}^{\infty}\left(M_{z_{i}}-\lambda_{i} I\right)\left(M_{z_{i}}^{*}-\bar{\lambda}_{i} I\right)\right) \\
& \Leftrightarrow \lambda \in\left\{\lambda \in l^{2} \mid \forall \varepsilon>0, \mu\left(\left\{x \in E: \sum_{i=1}^{\infty}\left|\lambda_{i}-x_{i}\right|^{2}<\varepsilon\right\}\right)>0\right\} \\
& \Leftrightarrow \lambda \in E.
\end{aligned}
$$
  \item [(3)] For a commuting d-tuple $(T_1,\cdots,T_d)$, by \cite{Taylorspectrum}, Taylor proved the spectral projection theorem. But it is false in infinitely many operators case. For example, Let $x_{i}=\left(\frac{1}{i}, \delta_{2, i}, \delta_{3, i}, \cdots\right)$, $i\geq 1$, where $\delta_{i, j}$ is the Kronecker symbol. Set $E=\left\{x_{i}\right\}_{i \geq 1}$. Obviously $\pi_{1}E=\left\{\frac{1}{i}\right\}_{i\geq 1}$ is not a compact set. By (2) there exists $\mathbb{T}$ such that $\sigma(\mathbb{T})=E$. Hence 
$$
\pi_1\sigma(\mathbb{T})\neq \sigma(T_1),
$$
where $\pi_1$ is the projection from $l^2$ onto $\mathbb{C}$. Moreover, since $\pi_1(\sigma(\mathbb{T}))$ is not compact, it follows that $\sigma(\mathbb{T})$ is not compact in $(l^2,\tau_p)$.
\end{itemize}
\end{rem}

\begin{thm}\label{thm:normconv-nonemptyI}
Let \(\mathbb T\) be a row-bounded sequence of commuting operators. Suppose that there exists \(J\subseteq\mathbb N\) such that
$
\sum\limits_{j\notin J}T_jT_j^*
$ and
$
\sum\limits_{j\in J}T_j^*T_j
$
both converge in operator norm. Then \(\sigma_{SZ}(\mathbb T)\subseteq\sigma(\mathbb T)\), and consequently \(\sigma(\mathbb T)\neq\emptyset\).
\end{thm}

\begin{proof}
Assume that $\sigma_{SZ}(\mathbb{T})\nsubseteq\sigma(\mathbb{T})$. Choose $\lambda \in \sigma_{SZ}(\mathbb{T})$ such that $\lambda \notin \sigma(\mathbb{T})$. Since $\lambda \notin \sigma(\mathbb{T})$, by Proposition \ref{Iinvert iff uniform homotopy}, there exists $h\in B(\Lambda(H))$ satisfying $\partial_{\mathbb{T}-\lambda} h + h \partial_{\mathbb{T}-\lambda} = I$. Therefore, for $\zeta\in D(\partial_{\mathbb{T}-\lambda})\cap D(\partial_{\mathbb{T}-\lambda}^*) $,
\begin{equation}\label{zeta C control}
\|\zeta\|^2 
\le (\|h\| (\|\partial_{\mathbb{T}-\lambda} \zeta\| + \|\partial_{\mathbb{T}-\lambda}^* \zeta\|))^2
\le 2\|h\|^2 ( \|\partial_{\mathbb{T}-\lambda} \zeta\|^2 + \|\partial_{\mathbb{T}-\lambda}^* \zeta\|^2 ),
\end{equation}
Moreover from $\lambda \in \sigma_{SZ}(\mathbb{T})$, we have for every $N\geq 1$, $(T_1-\lambda_1I,\cdots,T_N-\lambda_NI)$ is not Taylor invertible. Thus, there exists a unit vector $\xi_N \in H \otimes \Lambda [\mathbb{C}^N]$ such that
\begin{equation}\label{T-lambdaN le N-2}
\left\|\left(\sum_{j=1}^{N}(T_j-\lambda_jI)\otimes E_j^*\right) \xi_N\right\|^2 + \left\|\left(\sum_{j=1}^{N}(T_j-\lambda_jI)^*\otimes E_j\right) \xi_N\right\|^2 \le N^{-2}.
\end{equation}
Let $F_{N,M} = J \cap \{N+1, \cdots, M\}$, $\eta_{N,M} = \bigwedge\limits_{j \in F_{N,M}} e_j$, and $\zeta_{N,M} = \xi_N \wedge \eta_{N,M}$. Then $\|\zeta_{N,M}\| = \|\xi_N\| = 1$. 
Notice that for $N<k<k'$,
\begin{equation}\label{orthT1}
\left\langle \left(\sum_{j=1}^{N}(T_j-\lambda_j I)\otimes E_j^*\right)\xi_N\wedge\eta_{N,M},\,
((T_k-\lambda_k I)\otimes E_k^*)\xi_N\wedge\eta_{N,M}\right\rangle=0.
\end{equation}
and
\begin{equation}\label{orthT2}
\left\langle ((T_k-\lambda_k I)\otimes E_k^*)\xi_N\wedge\eta_{N,M},\,
((T_{k'}-\lambda_{k'} I)\otimes E_{k'}^*)\xi_N\wedge\eta_{N,M}\right\rangle=0.
\end{equation}
Therefore, by \eqref{orthT1} and \eqref{orthT2},
\begin{equation}\label{partialT}
\begin{aligned}
&\|\partial_{\mathbb T-\lambda}\zeta_{N,M}\|^2\\
=&\left\|\sum_{j\ge1}\big((T_j-\lambda_j I)\otimes E_j^*\big)(\xi_N\wedge\eta_{N,M})\right\|^2\\
=&
\left\|\left(\sum_{j=1}^{N}(T_j-\lambda_j I)\otimes E_j^*\right)(\xi_N\wedge\eta_{N,M})\right\|^2
+\sum_{j>N}\left\|((T_j-\lambda_j I)\otimes I)(\xi_N\wedge(E_j^*\eta_{N,M}))\right\|^2\\
=&\left\|\left(\sum_{j=1}^{N}(T_j-\lambda_j I)\otimes E_j^*\right)\xi_N\right\|^2\|\eta_{N,M}\|^2
+\sum_{j>N}\|((T_j-\lambda_j I)\otimes I)\xi_N\|^2\|E_j^*\eta_{N,M}\|^2\\
=&\left\|\left(\sum_{j=1}^{N}(T_j-\lambda_j I)\otimes E_j^*\right)\xi_N\right\|^2
+\sum_{j\in F_{N,M}}\|((T_j-\lambda_j I)\otimes I)\xi_N\|^2\\
=&\left\|\left(\sum_{j=1}^{N}(T_j-\lambda_j I)\otimes E_j^*\right)\xi_N\right\|^2
+\sum_{\substack{j>N\\ j\in J,\, j\le M}}\|((T_j-\lambda_j I)\otimes I)\xi_N\|^2,
\end{aligned}
\end{equation}
Similarly, for $N<k<k'$,
\begin{equation}\label{orthTstar1}
\left\langle \left(\sum_{j=1}^{N}(T_j-\lambda_j I)\otimes E_j^*\right)^*\xi_N\wedge\eta_{N,M},\,
((T_k-\lambda_k I)\otimes E_k^*)^*\xi_N\wedge\eta_{N,M}\right\rangle=0.
\end{equation}
and
\begin{equation}\label{orthTstar2}
\left\langle ((T_k-\lambda_k I)\otimes E_k^*)^*\xi_N\wedge\eta_{N,M},\,
((T_{k'}-\lambda_{k'} I)\otimes E_{k'}^*)^*\xi_N\wedge\eta_{N,M}\right\rangle=0.
\end{equation}
Therefore, by \eqref{orthTstar1}, and \eqref{orthTstar2},
\begin{equation}\label{partialTstar}
\begin{aligned}
&\|\partial_{\mathbb T-\lambda}^*\zeta_{N,M}\|^2\\
=&\left\|\sum_{j\ge1}\big((T_j-\lambda_j I)^*\otimes E_j\big)(\xi_N\wedge\eta_{N,M})\right\|^2\\
=&
\left\|\left(\sum_{j=1}^{N}(T_j-\lambda_j I)^*\otimes E_j\right)(\xi_N\wedge\eta_{N,M})\right\|^2
+\sum_{j>N}\left\|((T_j-\lambda_j I)\otimes I)^*(\xi_N\wedge(E_j\eta_{N,M}))\right\|^2\\
=&\left\|\left(\sum_{j=1}^{N}(T_j-\lambda_j I)^*\otimes E_j\right)\xi_N\right\|^2\|\eta_{N,M}\|^2
+\sum_{j>N}\|((T_j-\lambda_j I)\otimes I)^*\xi_N\|^2\|E_j\eta_{N,M}\|^2\\
=&\left\|\left(\sum_{j=1}^{N}(T_j-\lambda_j I)^*\otimes E_j\right)\xi_N\right\|^2
+\sum_{\substack{j>N\\ j\notin F_{N,M}}}\|((T_j-\lambda_j I)\otimes I)^*\xi_N\|^2\\
=&\left\|\left(\sum_{j=1}^{N}(T_j-\lambda_j I)^*\otimes E_j\right)\xi_N\right\|^2
+\sum_{\substack{j>N\\ j\notin J}}\|((T_j-\lambda_j I)\otimes I)^*\xi_N\|^2
+\sum_{\substack{j>M\\ j\in J}}\|((T_j-\lambda_j I)\otimes I)^*\xi_N\|^2,
\end{aligned}
\end{equation}
Then by \eqref{zeta C control}, \eqref{partialT} and \eqref{partialTstar},
\begin{equation}\label{partialzeta frac}
\begin{aligned}
&\frac{1}{2\|h\|^2}\\
=&\frac{1}{2\|h\|^2}\|\zeta_{N,M}\|^2\\
\leq&\|\partial_{\mathbb{T}-\lambda} \zeta_{N,M}\|^2 + \|\partial_{\mathbb{T}-\lambda}^* \zeta_{N,M}\|^2 \\
=& \left\| \left(\sum_{j=1}^{N}(T_j-\lambda_jI)\otimes E_j^*\right) \xi_N \right\|^2
 + \left\| \left(\sum_{j=1}^{N}(T_j-\lambda_jI)^*\otimes E_j\right) \xi_N  \right\|^2 \\
& + \sum_{\substack{j>N\\ j\in J,\, j\le M}} \|\big( (T_j-\lambda_j I) \otimes I \big)\xi_N\|^2
 + \sum_{\substack{j>N\\ j\notin J}} \|\big( (T_j-\lambda_j I)^* \otimes I \big)\xi_N\|^2
 + \sum_{\substack{j>M\\ j\in J}} \|\big( (T_j-\lambda_j I)^* \otimes I \big)\xi_N\|^2.
\end{aligned}
\end{equation}
Hence by \eqref{T-lambdaN le N-2} and \eqref{partialzeta frac},
\begin{equation}\label{j>M ge j>N}
\begin{aligned}
  &\sum_{\substack{j>M\\ j\in J}} \|\big( (T_j-\lambda_j I)^* \otimes I \big)\xi_N\|^2\\
  &\geq \frac{1}{2\|h\|^2}-N^{-2}-
  \sum_{\substack{j>N\\ j\in J,\, j\le M}} \|\big( (T_j-\lambda_j I) \otimes I \big)\xi_N\|^2
 -\sum_{\substack{j>N\\ j\notin J}} \|\big( (T_j-\lambda_j I)^* \otimes I \big)\xi_N\|^2.
 \end{aligned}
\end{equation}
Since both
$
\sum\limits_{j\notin J}T_jT_j^*
~\text{and}~
\sum\limits_{j\in J}T_j^*T_j
$
converge in operator norm, we have
\begin{equation}\label{2 and 3 to 0}
\lim\limits_{N\to\infty}
\left\|\sum_{\substack{j > N \\ j \notin J}} (T_j - \lambda_j I)(T_j - \lambda_j I)^* + \sum_{\substack{j > N \\ j \in J}} (T_j - \lambda_j I)^* (T_j - \lambda_j I)\right\|=0.
\end{equation}
Hence by \eqref{j>M ge j>N} and \eqref{2 and 3 to 0}, there exists an $N_0>0$ such that for all $N>N_0$ and $M>N$,
$$
  \sum_{\substack{j>M\\ j\in J}} \|\big( (T_j-\lambda_j I)^* \otimes I \big)\xi_N\|^2>\frac{1}{4\|h\|^2},
$$
which contradicts to $\sum\limits_{j \in J} (T_j - \lambda_j I)(T_j - \lambda_j I)^*$ converges in the strong operator topology. Therefore $\sigma_{SZ}(\mathbb{T}) \subseteq \sigma(\mathbb{T})$. 
\end{proof}
\begin{lem}\cite[Theorem 2.3]{SZ1974}\label{project sz}
  Let \(\mathbb T=\{T_j\}_{j\ge1}\) be a sequence of commuting operators, which may not be row-bounded. For any non-empty set $I\subset\mathbb{N}^+$, define
  $
  \mathbb{T}^{(I)}=\{T_j\}_{j\ge1,j\in I}
  $
  and $\pi_I$ be the projection from $\mathbb{C}^{\infty}$ onto $\mathbb{C}^I$. Then 
  $$
  \sigma_{SZ}(\mathbb{T}^{(I)})=\pi_I\sigma_{SZ}(\mathbb{T}).
  $$
\end{lem}
In general, \(\sigma(\mathbb T)\) and \(\sigma_{SZ}(\mathbb T)\) need not coincide.
However, for a row-bounded sequence \(\mathbb{T}\) of commuting operators on \(H\) such that 
\(\left\{\sum\limits_{i=1}^{k}T_i\otimes E_i^*\right\}_{k\geq1}\) converges in norm, we prove the following theorem.
\begin{thm}\label{thm3.1}
Let \(\mathbb{T}\) be a row-bounded sequence of commuting operators such that the sequence
\(\left\{\sum\limits_{i=1}^{k}T_i\otimes E_i^*\right\}_{k\geq1}\) converges in norm. Then
\[
\sigma(\mathbb{T})=\sigma_{SZ}(\mathbb{T}),
\]
which is a nonempty compact subset of \(\ell^2\).
\end{thm}
\begin{proof}
Since $\left\{\sum\limits_{i=1}^{k}T_{i}\otimes E_{i}^*\right\}_{k\geq1}$ converges in norm,  we have
$
\partial_{\mathbb{T}}
$
is bounded.
Notice that for \(h\in H\), we have
\begin{equation}\label{ddstar=ttstar}
\left(\sum\limits_{i=1}^{k}T_{i}\otimes E_{i,1}^*\right)\left(\sum\limits_{i=1}^{k}T_{i}\otimes E_{i,1}^*\right)^*h
=\left(\sum\limits_{i=1}^{k}T_{i}\otimes E_{i,1}^*\right)\sum_{i=1}^k T_i^*h\otimes e_i
=\left(\sum_{i=1}^k T_iT_i^*\right)h.
\end{equation}
Hence from $\left\{\sum\limits\limits_{i=1}^{k}T_{i}\otimes E_{i}^*\right\}_{k\geq1}$ converges in norm, by \eqref{ddstar=ttstar},
$
\sum\limits_{i=1}^\infty T_iT_i^*
$
converges in operator norm. Then by Theorem \ref{thm:normconv-nonemptyI}, $\sigma_{SZ}(\mathbb{T})\subseteq\sigma(\mathbb{T})$.

Conversely, suppose that $\lambda \notin \sigma_{SZ}(\mathbb{T})$, then there is an $n>0$ such that $(\mathbb{T}-I_{\lambda})^{(n)}$ is invertible. By Lemma \ref{lem:bounded-partial-normal}, 
$$
\|\partial_{\mathbb{T}-I_{\lambda}}\|\leq \|\partial_{\mathbb{T}}\|+\|\partial_{I_{\lambda}}\|\leq \|\partial_{\mathbb{T}}\|+\|\lambda\|  <\infty.
$$
It follows from Lemma \ref{finite invertible to I invertible} that $\mathbb{T}-I_{\lambda}$ is invertible. Therefore
$
\sigma(\mathbb{T}) \subseteq \sigma_{SZ}(\mathbb{T}).
$
Hence 
\begin{equation}\label{sz=e=si}
\sigma_{SZ}(\mathbb{T})=\sigma(\mathbb{T}).
\end{equation}
By \cite[Theorem 2.3]{SZ1974}, $\sigma_{SZ}(\mathbb{T})\neq\emptyset$. Then by \eqref{sz=e=si}, $\sigma(\mathbb{T})\neq \emptyset$. At last we prove that $\sigma(\mathbb{T})$ is a compact subset of $\ell^2$.
By Lemma \ref{project sz} and \eqref{sz=e=si}, for every $n,p>0$ and $\lambda\in \sigma(\mathbb{T})$, 
$$
(\lambda_n,\cdots,\lambda_{n+p})\in \sigma(T_n,\cdots,T_{n+p}).
$$
Hence by \cite[Lemma 5.3]{ClouatreHartz2018}, we have
\begin{equation}\label{spec radi}
 \sum_{i=n}^{n+p}|\lambda_i|^2\leq \left\|\sum\limits_{i=n}^{n+p}T_iT_i^*\right\|.
\end{equation}
Similar to \eqref{ddstar=ttstar}, 
\begin{equation}\label{ddstar leq partial}
\left\|\sum\limits_{i=n}^{n+p}T_iT_i^*\right\|=\left\|\sum_{i=n}^{n+p}T_i\otimes E_{i,1}^*\right\|^2\leq \left\|\sum_{i=n}^{\infty}T_i\otimes E_i^*\right\|^2.
\end{equation}
Hence by \eqref{spec radi} and \eqref{ddstar leq partial}.
\begin{equation}\label{tail converge}
  \lim\limits_{n\to\infty}\sup\limits_{\lambda\in \sigma(\mathbb{T})}\sum_{i=n}^{\infty}|\lambda_i|^2=0.
\end{equation}
It follows from Theorem \ref{bounded closed} and \eqref{tail converge} that $\sigma(\mathbb{T})$ is a compact subset of $\ell^2$.
\end{proof} 
\begin{rem}
The converse of Theorem \ref{thm3.1} also holds: for any nonempty compact subset \(K\) of \(\ell^2\), there exists a row-bounded sequence \(\mathbb{T}\) of commuting operators such that \(\sigma(\mathbb{T})=K\) and \(\sum\limits_{i=1}^\infty T_i\otimes E_i^*\) converges in norm. Indeed, as in Remark \ref{rem31}(2), let \(H=L^2(K,d\mu)\) and \(\mathbb{T}=(M_{z_1},M_{z_2},\cdots)\), where \(\mu\) is the counting measure. Then \(\sigma(\mathbb{T})=K\). Since \(K\) is compact in \(\ell^2\),
\[
\lim_{N\to\infty}\sup_{z\in K}\sum_{i>N}|z_i|^2=0.
\]
For \(N<M\), by Lemma \ref{lem:bounded-partial-normal},
\[
\left\|\sum_{i=N+1}^M T_i\otimes E_i^*\right\|^2
\leq
\left\|\sum_{i=N+1}^M T_iT_i^*\right\|
=
\sup_{z\in K}\sum_{i=N+1}^M |z_i|^2
\le
\sup_{z\in K}\sum_{i>N}|z_i|^2
\to 0.
\]
which implies that \(\sum\limits_{i=1}^\infty T_i\otimes E_i^*\) converges in norm.
\end{rem}
\begin{thm}
  Let $\mathbb{T}$ be a row‑bounded sequence of commuting operators. If for $1\le p\le 2$,
$
\sum\limits_{i=1}^\infty \|T_i\|^p<\infty,
$
then $\sigma(\mathbb T)$ is a nonempty compact subset of $\ell^p$.
\end{thm}
\begin{proof}
Fix $1\leq p\leq 2$.
Since $\sum\limits_{j=1}^{\infty}\|T_i\|^p<\infty$ and $p\leq 2$, we have $\sum\limits_{j=1}^{\infty}\|T_i\|^2\leq M\sum\limits_{j=1}^{\infty}\|T_i\|^p<\infty$ for some $M>0$. Hence $\sum\limits_{j=1}^\infty T_j T_j^*$ converges in operator norm, which implies that $\mathbb{T}$ is row-bounded. By Theorem \ref{thm:normconv-nonemptyI}, we have 
\begin{equation}\label{siT neq empty}
\sigma(\mathbb{T}) \neq \emptyset.
\end{equation}
Let 
$$
K_{\mathbb{T}} = \{z \in \ell^p : |z_j| \le \|T_j\|, \forall j\geq 1\}.
$$ 
Since $\sup\limits_{z \in K_{\mathbb{T}} } \sum\limits_{j > N} |z_j|^p \le \sum\limits_{j > N} \|T_j\|^p \to 0$ and $\pi_N(K_{\mathbb{T}} )$ is compact for $N\geq 1$, it follows that $K_{\mathbb{T}}$ is compact in $\ell^p$. By Corollary \ref{TS=1}, we have 
\begin{equation}\label{siT sub KT}
\sigma(\mathbb{T})\subset K_{\mathbb{T}}\subset\ell^p.
\end{equation}
By Theorem \ref{bounded closed}, $\sigma(\mathbb{T})$ is closed in $\ell^2$, which implies that $\sigma(\mathbb{T})$ is closed in $\ell^p$. Thus by \eqref{siT neq empty} and \eqref{siT sub KT}, $\sigma(\mathbb{T})$ is a non-empty compact subset of $\ell^p$.
\end{proof}

The following example shows that the $\sigma(\mathbb T)$ of a row-bounded sequence $\mathbb T$ of commuting operators may be empty. 
Before constructing the example, we need some preliminary results.

For \(z \in \mathbb D\), define
\[
f_1(z)=\frac{\sqrt{1-z}}{\sqrt{2-z}} 
\qquad \text{and} \qquad 	
f_n(z)
=
\frac{\sqrt{2^{1-n}}\sqrt{1-z}}
{\sqrt{1-z+2^{1-n}}\sqrt{1-z+2^{2-n}}},
\qquad n\geq2,
\]
where all square roots are taken with respect to the principal branch.	
Clearly, \(f_n\in A(\mathbb D)\) for every \(n\geq 1\). 	
Let \(H=L^2(0,1)\), and let \(V\in B(H)\) be the Volterra
operator defined by
\[
(Vh)(x)=\int_0^x h(t)\,dt,
\qquad h\in H.
\]
Since \(\sigma(V)=\{0\}\) (see \cite[Section~1]{RansfordTsedenbayar2024}), we have \(I+V\) is invertible.
Thus we can define $$C=(I-V)(I+V)^{-1}.$$
	By \cite[Theorem~2.1]{Khadkhuu2015}, we have \(\|C\|=1\).
Notice that $f_i\in A(\overline{\mathbb D})$, there is a sequence of polynomials $p_n$ such that 
$$
\lim\limits_{n\to\infty}\sup\limits_{z\in\mathbb D}|f_i(z)-p_n(z)|=0.
$$
By von Neumann inequality, $p_n(C)$ converges in norm topology, and define 
$$
f_i(C)=\lim\limits_{n\to\infty} p_n(C).
$$
Then we have the following Lemma.

\begin{lem}\label{DNLN}
	For each \(N\geq 1\), 
	\begin{align}\label{sqrt2}
		\left\|\sum_{n=1}^N f_n(C)\otimes E_n^*\right\| \leq \sqrt2.
	\end{align}
	and
	\begin{equation}\label{SOT}
		\operatorname{SOT}\!-\!
		\sum_{n=1}^{\infty} \bigl(f_n(C)\bigr) ^2
		=
		I.
	\end{equation} 
\end{lem}
\begin{proof}
	
	For each \(z \in \overline{\mathbb D}\) and any \(a >0\), we have
	\begin{align*}
		|1-z+a|^2 \geq |1-z|^2+a^2 \geq \frac{(|1-z|+a)^2}{2}.
	\end{align*}
	Then
	\begin{equation}\label{sumfn}
		\begin{aligned}
			\sum_{n=1}^N |f_n(z)|^2 
			&=
			\frac{|1-z|}{|1-z+1|}+\sum_{n=2}^N \frac{2^{1-n}|1-z|}{|1-z+2^{1-n}||1-z+2^{2-n}|}\\
			&\leq
			\frac{2|1-z|}{|1-z|+1}+ 2 \sum_{n=2}^N \frac{ 2^{1-n}|1-z|}{(|1-z|+2^{1-n})(|1-z|+2^{2-n})}\\
			&=
			\frac{2|1-z|}{|1-z|+1}+2\sum_{n=2}^N \left(  \frac{|1-z|}{|1-z|+2^{1-n}} - \frac{|1-z|}{|1-z|+2^{2-n}}\right)\\
			&=
			\frac{2|1-z|}{|1-z|+2^{1-N}}\leq 2.
		\end{aligned}
	\end{equation}
	Hence, by the operator-valued von Neumann inequality
	\cite[Theorem~1.1, $(1)\Rightarrow(1'')$]{BallEtAl2008} and \eqref{sumfn}, we obtain
	\begin{align*}
		\left\|\sum_{n=1}^N f_n(C)\otimes E_n^*\right\|^2
		&\leq 
		\sup_{\zeta \in \mathbb T}\left\|\sum_{n=1}^{N}f_n(\zeta)E_n^*\right\|^2\\
		&=
		\sup_{\zeta \in \mathbb T}\left\|\left(\sum_{n=1}^{N}f_n(\zeta)E_n^*\right)^*\left(\sum_{n=1}^{N}f_n(\zeta)E_n^*\right)\right\|\\
		&\leq
			\sup_{\zeta \in \mathbb T}\left\|\left(\sum_{n=1}^{N}f_n(\zeta)E_n^*\right)^*\left(\sum_{n=1}^{N}f_n(\zeta)E_n^*\right)+\left(\sum_{n=1}^{N}f_n(\zeta)E_n^*\right)\left(\sum_{n=1}^{N}f_n(\zeta)E_n^*\right)^*\right\| \\
		&=
		\sup_{\zeta \in \mathbb T}\left\| \sum_{n,m=1}^N f_n(\zeta)\overline{f_m(\zeta)} (E_nE_m^*+E_m^*E_n)\right\|\\
		&\overset{(*)}{=}
		\sup_{\zeta \in \mathbb T} \left(\sum_{n=1}^N |f_n(\zeta)|^2\right)\leq2,
	\end{align*}
	where \((*)\) is given by \(E_nE_m^* +E_m^*E_n =\delta_{nm}I_{\Lambda}\).
	This proves \eqref{sqrt2}.
	
	For each \(N\geq 1\), since
	\begin{equation*}
		\sum_{n=1}^Nf_n^2(z)
		=
		\frac{1-z}{2-z} + \sum_{n=2}^N \left( \frac{1-z}{1-z+2^{1-n}} -\frac{1-z}{1-z+2^{2-n}}  \right)
		=
		\frac{1-z}{1-z+2^{1-N}},
	\end{equation*}  
	we have
	\begin{equation}\label{sumfnC}
		\sum_{n=1}^N\bigl(f_n(C)\bigr)^2
		=
		(I-C){\left(I-C+2^{1-N}I  \right)^{-1}}
		=
		I-(1-r_N)(I-r_NC)^{-1},
	\end{equation}
	where \(r_N=(1+2^{1-N})^{-1}\).
	Since \(C\) is a contraction, for \(0<r<1\), we have
	\[
	\left\|(1-r)(I-rC)^{-1}\right\|=\left\|(1-r)\sum_{k=0}^\infty r^k C^k\right\| \leq (1-r)\sum_{k=0}^\infty r^k=1.
	\]
	Moreover,  a direct computation shows that 
		\[
		(V^*h)(x)=\int_x^1 h(t)dt,\qquad h\in H,
		\] 
		and hence \(V^*\) is injective.
	Then
	\[
	\overline{(I-C)H}=\overline{2V(I+V)^{-1} H}=\overline{VH}=(\ker V^*)^{\perp} =H.
	\]
	Thus, by \cite[Proposition~5.4]{CohenCunyEisnerLin2020},
	\begin{align}\label{CCEL}
		\operatorname{SOT}\!-\!\lim_{N\to \infty} (1-r_N)(I-r_NC)^{-1}
		= \operatorname{SOT}\!-\!\lim_{r\to 1^-} (1-r)(I-rC)^{-1}=0.
	\end{align}
	Therefore, by \eqref{sumfnC} and \eqref{CCEL}, we obtain \eqref{SOT}.
\end{proof}

\begin{exam}\label{empty spec}
	Let \(\mathbb T =(f_1(C), f_2(C),\ldots)\).
	Then \(\sigma( \mathbb T)= \emptyset\).
\end{exam}
\begin{proof}
	For any \(\lambda =(\lambda_1,\lambda_2,\ldots) \neq 0\),  we have
	\begin{align*}
		&\left\|\sum_{n=1}^N (T_n-\lambda_nI_H) \otimes E_n^*\right\| \\
		\leq&
		\left\|\sum_{n=1}^N T_n \otimes E_n^*\right\| + \left\|\sum_{n=1}^N \lambda_nI_H \otimes E_n^*\right\| \\
		\overset{(*)}{\leq}&
		\sqrt2 + \left\|\left( \sum_{n=1}^N \lambda_nI_H \otimes E_n^*  \right)\left( \sum_{n=1}^N \lambda_nI_H \otimes E_n^*  \right)^*+ \left( \sum_{n=1}^N \lambda_nI_H \otimes E_n^*  \right)^*\left( \sum_{n=1}^N \lambda_nI_H \otimes E_n^*  \right) \right\|^{\frac{1}{2}}\\
		=& \sqrt2 + \left(\sum_{n=1}^N|\lambda_n|^2\right)^{\frac12}\\
		\leq&
		\sqrt2 + \|\lambda\|,
	\end{align*}
	where \((*)\) is given by \eqref{sqrt2}.
		Then by Lemma~\ref{bound and SOT}, \(\mathbb T-\lambda\) is row bounded.
		Choose \(k \geq 1\) such that \(\lambda_k \neq 0\).
	Since \(\sigma(V)=\{0\}\), we have \(\sigma(C)=\{1\}\). 
	Notice that \(f_k^2\) extends holomorphically to a neighborhood of \(1\). 
	Therefore,
	\[
	\left\{\lambda^2:\lambda\in\sigma\bigl(f_k(C)\bigr)\right\}
	=\sigma\bigl(f_k^2(C)\bigr)
	=(f_k^2)\bigl(\sigma(C)\bigr)
	=\{0\}.
	\]
	Hence \(\sigma(T_k)=\{0\}\). 
	Consequently, \(T_k -\lambda_k\) is invertible.
	By Corollary~\ref{TS=1} with \(\mathbb S = (0,\ldots,0,(T_k-\lambda_k)^{-1}, 0,\ldots)\), we have \(\mathbb T -\lambda\) is invertible.
	Hence \(\sigma(\mathbb T) \subset \{0\}\). 
	It remains to show that \(\mathbb T\) is invertible.
	
	By Lemma~\ref{bound and SOT} and \eqref{sqrt2}, \(\mathbb T\) is row-bounded and \(\sum\limits_{n=1}^\infty T_n \otimes E_n^*\) converges in the strong operator topology. 
	Since 
		\[
		\partial_{\mathbb T} = \operatorname{SOT}\!-\!\lim_{N\to\infty} \sum\limits_{n=1}^{N}T_n\otimes E_n^*,
		\]
		by \eqref{sqrt2} again, we have \(\|\partial_{\mathbb T}\| \leq \sqrt2\).
	Moreover, it follows from \eqref{SOT} that
	\[
	\operatorname{SOT}\!-\!\lim_{N\to\infty}
	\sum_{n=1}^{N}T_n^2
	=
	\operatorname{SOT}\!-\!\lim_{N\to\infty}
	\sum_{n=1}^{N}\left(f_n(C)\right)^2
	=I.
	\]
	Therefore, Corollary~\ref{TS=1} implies that \(\mathbb T\) is invertible.
\end{proof}

\section{Holomorphic functional calculus vs. the compact-set von Neumann inequality}
In this section, we prove Theorem~\ref{thm:weighted-polynomial-calculus}, Proposition~\ref{Prop:vNinequtocompactness} and establish the corresponding spectral mapping theorem.
We begin with some preliminary results.
Let \(c_0 = \{(z_i)_{i \geq 1} : z_i \in \mathbb C, z_i \to 0\}\) and  \((\mathbb R_{+})^{\mathbb N} = \{(r_i)_{i \geq 1} : r_i \in \mathbb R_{+}, i \geq 1\}\).
For $\mathbf d=(d_i)_{i\geq1}\in c_0\cap(\mathbb R_{+})^{\mathbb N}$, define
\[
\mathcal F_{\mathbf d}=
\left\{
z\in\ell^2:
\sum_{i=1}^{\infty}\frac{|z_i|^2}{d_i^2}\leq1
\right\}.
\]
\begin{lem}\label{lem:compact-ellipsoid}
	Let $Q$ be a nonempty compact set of $\ell^2$.
	Then there exists $\mathbf d \in c_0\cap(\mathbb R_{+})^{\mathbb N}$ such that $Q\subset\mathcal F_{\mathbf d}$.
	Moreover, $\mathcal F_{\mathbf d}$ is a compact set of $\ell^2$.
\end{lem}

\begin{proof}
	Since $Q$ is compact, we have
	$\sup\limits_{z\in Q}\|(I-P_n)z\|\to0$ as $n\to\infty$.
	Hence there exist integers
	$0=N_0<N_1<N_2<\cdots$ such that
	\begin{align}\label{eq:4.28}
		\sup_{z\in Q}\sum_{i>N_j}|z_i|^2
		=\sup_{z\in Q}\|(I-P_{N_j})z\|^2
		\leq \left( \sup\limits_{z\in Q}\|z\|\right)^2\,4^{-j},
		\ j\geq0.
	\end{align}
	Write $R=\max\Bigl\{1,\sup\limits_{z\in Q}\|z\|\Bigr\}$. 
	For every fixed \(j\geq0\), set \(	d_i=\sqrt{2}\,R\,2^{-\frac j2},\ N_j<i\leq N_{j+1}\).
	Then it is easy to see that $d=(d_i)_{i \geq 1} \in c_0 \cap (\mathbb R_{+})^{\mathbb N}$.
	For every $z\in Q$, by \eqref{eq:4.28},
	\[
	\sum_{i=1}^{\infty}\frac{|z_i|^2}{d_i^2}
	= \frac1{2R^2}\sum_{j=0}^{\infty}2^j
	\sum_{i=N_j+1}^{N_{j+1}}|z_i|^2
	\leq \frac12\sum_{j=0}^{\infty}2^{-j}
	=1.
	\]
	Thus $Q\subset\mathcal F_{\mathbf d}$.	Moreover, it is easy to see that \(\mathcal F_{\mathbf d}\) is closed and bounded, and 
	\begin{equation}\label{eq:Ed-tail}
		\begin{aligned}
			\sup_{z\in \mathcal F_{\mathbf d}}
			\left(\sum_{i>N}|z_i|^2\right)^{1/2}
			\leq
			\left(\sup_{i>N}d_i\right)
			\sup_{z\in \mathcal F_{\mathbf d}}
			\left(\sum_{i>N}\frac{|z_i|^2}{d_i^2}\right)^{1/2}
			\leq\sup_{i>N}d_i\longrightarrow0.
		\end{aligned}
	\end{equation}
	Therefore, $\mathcal F_{\mathbf d}$ is compact.
\end{proof}
\begin{prop}
If \(\mathbb{T} = (T_1, T_2, \dots)\) is a row-bounded sequence of commuting operators on \(H\) satisfying \(\sum\limits_{i=1}^\infty \|T_i\| < \infty\), then \(\mathbb{T}\) satisfies the compact-set von Neumann inequality.
\end{prop}
\begin{proof}
Let \(a_i = \|T_i\|\) and \(A = \sum\limits_{i=1}^\infty a_i\). If \(A=0\) the conclusion is trivial. Assume \(A>0\), and set
\[
d_i = 2\sqrt{Aa_i + 4^{-i}}.
\]
Let
\[
K_{\mathbf{d}}=\{z\in\ell^2:|z_i|\le d_i,\ i\ge1\}.
\]
Since
\[
\sum_{i=1}^{\infty}d_i^2
=4A^2+4\sum_{i=1}^{\infty}4^{-i}<\infty,
\]
$K_{\mathbf{d}}$ is closed and bounded in $\ell^2$. Moreover,
\[
\sup_{z\in K_{\mathbf{d}}}\sum_{i>N}|z_i|^2
\le\sum_{i>N}d_i^2\longrightarrow0.
\]
Hence $K_{\mathbf{d}}$ is compact. Notice that
\[
\sum_{i=1}^{\infty}\left\|\frac{T_i}{d_i}\right\|^2
=\sum_{i=1}^{\infty}
\frac{a_i^2}{4(Aa_i+4^{-i})}
\le\frac1{4A}\sum_{i=1}^{\infty}a_i
=\frac14.
\]
Hence the sequence \((T_1/d_1, T_2/d_2, \dots)\) satisfies Brehmer's positivity condition. By Brehmer's dilation theorem \cite{Brehmer1961}, every finite tuple \((T_1/d_1, \dots, T_n/d_n)\) has a commuting unitary dilation, and hence for any polynomial \(p\),
\[
\|p(T_1/d_1, \dots, T_n/d_n)\| \le \sup_{|z_i| \le 1, 1 \le i \le n} |p(z_1, \dots, z_n)|.
\]
Thus, for every $n\ge1$ and
$p\in\mathbb C[z_1,\ldots,z_n]$,
\[
\|p(T_1,\ldots,T_n)\|
\le
\sup_{\substack{|z_i|\le d_i\\1\le i\le n}}
|p(z_1,\ldots,z_n)|
=\sup_{z\in K_{\mathbf{d}}}|p(z)|.
\]
Consequently, $\mathbb T$ satisfies the compact-set
von Neumann inequality with respect to $K_{\mathbf{d}}$.
\end{proof}
Let $\mathbb T=(T_i)_{i\ge1}$ be a row-bounded sequence
of commuting operators,
and let $K$ be a given compact subset of $\ell^2$.
Throughout Lemmas~\ref{lem:product-neighborhood}--\ref{lem:bijinGammaV}, we assume that \textbf{$\pmb{\mathbb T}$ satisfies condition~(1) of Theorem~\ref{thm:weighted-polynomial-calculus}}.
Then there exists a nonempty compact set $\mathcal K$ of \(\ell^2\) such that
\[
\|p(\mathbb T)\|
\leq \sup_{z\in \mathcal K}|p(z)|
\]
for every $p\in \mathcal P$.
By Lemma~\ref{lem:compact-ellipsoid},  there exists
$\mathbf{d}=(d_i)_{i \geq 1} \in c_0 \cap (\mathbb R_{+})^{\mathbb N}$ such that 
\begin{align}\label{K subset Ed}
	\mathcal K\cup K\subset \mathcal F_{\mathbf d}.
\end{align}
Then for every $p\in \mathcal P$, we have
\begin{equation}\label{eq:ellipsoid-bound}
	\|p(\mathbb T)\|
	\leq\sup_{z\in \mathcal F_{\mathbf d}}|p(z)|.
\end{equation}

For each $n\geq1$, define $\iota_n:\mathbb C^n\to\ell^2$ by $$\iota_n(z_1,\ldots,z_n)=(z_1,\ldots,z_n,0,\ldots).$$
Set $P_n=\iota_n\pi_n$.
For $r>0$, define $F_{n}=(I-P_n)\mathcal F_{\mathbf d}$ and $\Omega_{r,n}=\pi_nK+rB_{\mathbb C^n}$.

\begin{lem}\label{lem:product-neighborhood}
	For every open neighborhood $U\subset\ell^2$ of $K$,
	there exist $r>0$ and $N\geq1$ such that
	\[
	K\subset \iota_N(\Omega_{r,N}) \oplus F_{N}
	\subset \iota_N(\Omega_{r,N}) \oplus r(I-P_N)B
	\subset U.
	\]
\end{lem}

\begin{proof}
	Since $K$ is compact, there exists $r>0$ such that
	\begin{equation}\label{eq:K-neighborhood}
		K+3rB\subset U.
	\end{equation}
	Choose $N$ so large that $\sup\limits_{i>N}d_i<\frac r2$.
	By \eqref{K subset Ed}, we have
	\[
	K\subset P_NK+(I-P_N)K \subset \iota_N(\Omega_{r,N}) \oplus (I-P_N) \mathcal F_{\mathbf d} = \iota_N(\Omega_{r,N}) \oplus F_{N}.
	\]
	Thus the first inclusion holds.
	Moreover, by \eqref{eq:Ed-tail},
	\begin{equation}\label{eq:tail-r}
		\sup_{y\in F_{N}}\|y\|
		=
		\sup_{z\in \mathcal F_{\mathbf d}}\|(I-P_N)z\|
		\leq\sup_{i>N}d_i<\frac r2.
	\end{equation}
	Thus 
	\begin{align}\label{FN subset B}
		F_N\subset \frac r2(I-P_N)B,
	\end{align}
	which implies the second inclusion.
	For $x\in\Omega_{r,N}$,
	there exists $\lambda\in K$ such that
	$\|x-\pi_N\lambda\|_2<r$.
	By \eqref{K subset Ed} and \eqref{eq:tail-r}, for \(y \in r(I-P_N)B\),
	\begin{equation}\label{eq:product-distance}
		\begin{aligned}
			\|\iota_N(x)+y-\lambda\|
			&=\|\iota_N(x)+y-P_N\lambda-(I-P_N)\lambda\|\\
			&\leq
			\|x-\pi_N\lambda\|_2+\|y\|+\|(I-P_N)\lambda\|\\
			&<r+r+\frac r2<3r.
		\end{aligned}
	\end{equation}
	Therefore, by \eqref{eq:product-distance} and \eqref{eq:K-neighborhood},
	\[
	\iota_N(x)+y \in\lambda+3rB\subset K+3rB\subset U.
	\]
	This proves the last inclusion.
\end{proof}
Let \(U \subset \ell^2\) be an open neighborhood of \(K\).
Choose $r$ and $N$ satisfying the conditions of Lemma~\ref{lem:product-neighborhood}.
For simplicity, \textbf{write \(\mathbf{\Omega=\Omega_{r,N}}\), \(\mathbf{F=F_N}\) and $$\bm{V=\iota_N(\Omega)\oplus r(I-P_N)B}.$$}
Let
\[
\mathcal A
=
\overline{\left\{q|_F:q\in\bigcup\limits_{m>N}\mathbb C[z_{N+1},\ldots,z_m]\right\}}^{\,C(F)}.
\]
Then $\mathcal A$ is a unital commutative Banach algebra,
with norm $\|a\|_{\mathcal A}=\sup\limits_{y\in F}|a(y)|$.
We have the following lemma.
\begin{lem}\label{lem:J}
	Let \(f \in \operatorname{Hol}(V)\), then
	\begin{enumerate}
		\item [(1)] For each  $x\in\Omega$, $f(x,\cdot)|_F\in\mathcal A$;
		\item [(2)] The \(\mathcal A\)-valued function $J_f:\Omega\to\mathcal A$ defined by
		$J_f(x)=f(x,\cdot)|_F$ is holomorphic.
	\end{enumerate}
\end{lem}

\begin{proof}
	(1) Fix $x\in\Omega$.
	For each $m>N$, by Lemma~\ref{lem:product-neighborhood}, $f(x,\cdot)$ is holomorphic on $rP_m(I-P_N)B$. 
	By \eqref{FN subset B}, we have $P_mF\subset \frac r2 P_m(I-P_N)B$. 
	Thus we have the Taylor expansion
	\[
	f(x,P_my)
	=
	\sum_{k=0}^{\infty}
	\frac{1}{k!}D_y^kf(x,0)[P_my,\ldots,P_my],
	\ y\in F,
	\]
	which converges uniformly on $F$.
	Therefore, by the definition of \(\mathcal A\), $f(x,P_m\cdot)|_F\in\mathcal A$.
	Since \(\mathcal F_{\mathbf d}\) is compact,
	\begin{align}\label{eq:to0}
		\sup_{y\in F}\|P_my-y\|
		=\sup_{y\in \mathcal F_{\mathbf d}}\|(I-P_m)(I-P_N)y\|
		=\sup_{y\in \mathcal F_{\mathbf d}}\|(I-P_m)y\|\longrightarrow0.
	\end{align}
	Consequently,
	\[
	\lim_{m\to\infty}
	\|f(x,P_m\cdot)|_F-f(x,\cdot)|_F\|_{C(F)}=0.
	\]
	Since $\mathcal A$ is closed in $C(F)$, it follows that
	$f(x,\cdot)|_F\in\mathcal A$.
	
	(2)
	For any $x_0\in\Omega$, there exists $\varepsilon>0$
	such that $\overline{B(x_0,\varepsilon)}\subset\Omega$.
	Then
	\[
	\sup_{x\in B(x_0,\varepsilon)}\|J_f(x)\|_{\mathcal A}
	\leq
	\sup_{(x,y)\in\overline{B(x_0,\varepsilon)}\times F}
	|f(x,y)|<\infty.
	\]
	Thus $J_f$ is locally bounded.
	Moreover, it is easy to verify that \(\operatorname{span}\{\delta_y, y\in F\}\) of \(\mathcal A^*\) separates the points of \(\mathcal A\) and 
	$\delta_y\circ J_f=f(\cdot,y)\in\operatorname{Hol}(\Omega)$. Therefore, by \cite[Theorem 6.1,\ \((f)\Rightarrow(a)\)]{Kruse2020}, we have $J_f\in\operatorname{Hol}(\Omega;\mathcal A)$.
\end{proof}

\begin{lem}\label{lem:rho-exist}
	There exists a bounded unital algebra homomorphism
	$\rho:\mathcal A\to B(H)$ with the following properties.
	\begin{enumerate}
		\item[(1)] For each $m>N$ and
		$p\in\mathbb C[z_{N+1},\ldots,z_m]$, \(\rho(p|_F)=p(T_{N+1},\ldots,T_m)\);
		\item[(2)] For every $a\in\mathcal A$ and	$[g]_N\in\operatorname{Hol}(\pi_NK)$,
		$\rho(a)$ commutes with $\tau_N([g]_N)$.
	\end{enumerate} 
\end{lem}
\begin{proof}
	For each $m>N$ and $p\in\mathbb C[z_{N+1},\ldots,z_m]$,
	by \eqref{eq:ellipsoid-bound},
	\begin{equation}\label{eq:rho-bound}
		\|p(T_{N+1},\ldots,T_m)\|
		\leq
		\sup_{z\in \mathcal F_{\mathbf d}}|p(z_{N+1},\ldots,z_m)|
		\leq
		\sup_{y\in F}|p(y)|
		=
		\|p|_F\|_{\mathcal A}.
	\end{equation}
	Thus there exists a bounded unital algebra homomorphism satisfying (1).
	Moreover, for $j>N$ and $[g]_N\in\operatorname{Hol}(\pi_NK)$, by the coordinate-preserving property and compatibility
	of the family $\{\tau_n\}_{n\ge1}$, we have
	\[
	\begin{aligned}
		\tau_N([g]_N)T_j
		=\tau_j([g\circ\pi_{j,N}]_j)\tau_j([z_j]_j)
		=\tau_j\bigl([(g\circ\pi_{j,N})z_j]_j\bigr)
		=\tau_j([z_j]_j)\tau_j([g\circ\pi_{j,N}]_j)
		=T_j\tau_N([g]_N).
	\end{aligned}
	\]
	Hence $\tau_N([g]_N)$ commutes with $\rho(a)$
	for every $a\in\mathcal A$.
\end{proof}

Define the bilinear map
$\Theta_0:\operatorname{Hol}(\Omega)\times\mathcal A\to B(H)$
by $(g,a)\mapsto\tau_N([g]_N)\rho(a)$.
By the continuity of $\tau_N$, there exist a compact set
$\mathbf K \subset\Omega$ and a constant $C>0$ such that
\begin{equation}\label{eq:tau-bound}
	\|\tau_N([g]_N)\|\leq C\sup_{x\in \mathbf K}|g(x)|.
\end{equation}
Then, by \eqref{eq:rho-bound} and \eqref{eq:tau-bound},
\begin{align}\label{eq:Theta-bound}
	\|\Theta_0(g,a)\|
	\leq C\sup_{x\in \mathbf K}|g(x)|\,\|a\|_{\mathcal A}.
\end{align}
Therefore, by \eqref{eq:Theta-bound} and Lemma~\ref{lem:rho-exist}, $\Theta_0$ induces a continuous unital
algebra homomorphism
\[
\Theta:
\operatorname{Hol}(\Omega)\widehat{\otimes}_{\pi}\mathcal A
\longrightarrow B(H),
\qquad
\Theta(g\otimes a)=\tau_N([g]_N)\rho(a),
\]
where $\operatorname{Hol}(\Omega)\widehat{\otimes}_{\pi}\mathcal A$ is the completion of the algebraic tensor product $\operatorname{Hol}(\Omega)\otimes\mathcal A$ with respect to the projective tensor product topology.
By $\operatorname{Hol}(\Omega)\widehat{\otimes}_{\pi}\mathcal A \cong\operatorname{Hol}(\Omega;\mathcal A)$ (see \cite[Chapter~II, \S4.2]{Helemskii1989}) and Lemma~\ref{lem:J}, define 
\[
\Gamma_V:\operatorname{Hol}(V)\longrightarrow B(H),
\qquad
\Gamma_V(f)=\Theta J_f.
\]

By the definition of $V$, it is easy to see that $P_mV\subset V$ for $m\geq N$. 
Then we have the following lemma.
\begin{lem}\label{lem:fPm to f}
	For any  \(f \in \operatorname{Hol}(V)\), $f\circ P_m\to f$ in the compact-open topology of $\operatorname{Hol}(V)$.
\end{lem}
\begin{proof}
	Suppose that there exists a compact set $\mathbf K\subset V$ such that $\sup\limits_{z\in \mathbf K}|f(P_mz)-f(z)|\not\to0$ as \(m\to \infty\).
	Then there exist $\varepsilon>0$, $m_j\to\infty$,
	and $z_j\in \mathbf K$ such that
	\begin{equation}\label{eq:Pm-contradiction}
		|f(P_{m_j}z_j)-f(z_j)|\geq\varepsilon,
		\ j\geq1.
	\end{equation}
	Since $\mathbf K$ is compact, without loss of generality, assume that $z_j\to z\in \mathbf K$. Therefore,
	\[
	\lim_{j \to \infty}\|P_{m_j}z_j-z\|
	\leq \lim_{j \to \infty}\left(\|P_{m_j}z_j-z_j\|+\|z_j-z\|\right)
	\leq 
	\lim_{j \to \infty}\left(\sup_{u\in \mathbf K}\|P_{m_j}u-u\|+\|z_j-z\|\right)
	\overset{(*)}{=}0,
	\]
	where \((*)\) follows from the compactness of $\mathbf K$.
	Hence, from the continuity of $f$ at $z$, we obtain a contradiction to \eqref{eq:Pm-contradiction}.
\end{proof}

\begin{lem}\label{lem:bijinGammaV}
	The map $\Gamma_V$ is a unital algebra homomorphism, continuous from the compact-open topology to the operator norm topology.
	Moreover, for every $f\in\operatorname{Hol}(V)$,
	\begin{align}\label{eq:bijinbiaoda}
		\Gamma_V(f)=\lim_{m\to\infty}\tau_m([f\circ\iota_m]_m).
	\end{align}
\end{lem}

\begin{proof}
	It is easy to verify that $\Gamma_V$ is a unital algebra homomorphism.
	For every compact set $Q \subset\Omega$, we have
	\[
	\sup_{x\in Q}\|J_f(x)\|_{\mathcal A}
	=
	\sup_{(x,y)\in Q\times F}|f(x,y)|.
	\]
	Therefore, $\Gamma_V$ is continuous from the compact-open topology to the operator norm topology.

	Fix $m\geq N$ and define
	$g_m=f\circ\iota_m\in\operatorname{Hol}(\pi_m(V))$.
	We have the expansion
	\[
	g_m(x,w)
	=
	\sum_{k=0}^{\infty}\sum_{|\beta|=k}g_{m,\beta}(x)w^\beta,
	\]
	where $g_{m,\beta}\in\operatorname{Hol}(\Omega)$ and
	the partial sums converge uniformly on compact subsets
	of $\pi_m(V)$.
	Therefore, by Lemma~\ref{lem:fPm to f}, the continuity of  \(\Gamma_V\) and the compatibility
	of the family $\{\tau_n\}_{n\ge1}$, we have
	\begin{align*}
		\Gamma_V(f) &= \lim_{m\to\infty}\Gamma_V(f\circ P_m)\\
		&= \lim_{m\to\infty}\Gamma_V(g_m\circ \pi_m)\\
		&=\lim_{m\to\infty}\sum_{k=0}^{\infty}\sum_{|\beta|=k} \Gamma_V\bigl((g_{m,\beta}(x)w^\beta) \circ \pi_m\bigr)\\
		&=\lim_{m\to\infty}\sum_{k=0}^{\infty}\sum_{|\beta|=k}
		\tau_N([g_{m,\beta}]_N)\prod_{j=1}^{m-N}T_{N+j}^{\beta_j}\\
		&=\lim_{m\to\infty} \sum_{k=0}^{\infty}\sum_{|\beta|=k}
		\tau_m([g_{m,\beta}\circ\pi_{m,N}]_m)
		\prod_{j=1}^{m-N}T_{N+j}^{\beta_j}\\
		&\overset{(*)}{=}\lim_{m\to\infty}\tau_m([g_m]_m)\\
		&=\lim_{m\to\infty}\tau_m([f\circ\iota_m]_m),
	\end{align*}
		where $(*)$ follows from $g_{m,\beta}\circ\pi_{m,N}(z_1,\ldots,z_m)=g_{m,\beta}(z_1,\ldots,z_N)$ for $m\geq N$.
\end{proof}
We rewrite Theorem {\ref{thm:weighted-polynomial-calculus}} as follows for convenience.
\begin{thm}\label{thm4.8}
	Let $\mathbb T$ be a row-bounded sequence of commuting operators,  $K$ be a compact subset in $\ell^2$. 
	Then the following statements are equivalent.
	\begin{enumerate}
		\item[(1)]
		$\mathbb T$ satisfies the compact-set von Neumann inequality,  and there is a compatible family of holomorphic functional calculi  $\tau_n: {\rm Hol}(\pi_nK)\to B(H)$ of $\mathbb T_n$ with respect to $\pi_nK$;
		\item[(2)]
		$\mathbb T$ admits a continuous holomorphic functional
		calculus on $K$.
	\end{enumerate}
\end{thm}

\begin{proof}
	$(2)\Rightarrow(1)$.
	By (2) in Definition~\ref{def:Gamma}, there exist a nonempty compact set
	$\mathbf K\subset\ell^2$ and a constant $C\geq1$ such that,
	for every $f\in\operatorname{Hol}(\ell^2)$,
	\begin{equation}\label{eq:Gamma-compact-bound}
		\|\Gamma([f])\|\leq C\sup_{z\in \mathbf K}|f(z)|.
	\end{equation}
	Write $a=2C+1$.
	By Lemma~\ref{lem:compact-ellipsoid}, choose a sequence
	$d=(d_1,d_2,\ldots) \in c_0 \cap (\mathbb R_{>0})^{\mathbb N}$ such that $a\mathbf K\subset \mathcal F_{\mathbf d}$.
	Thus, for every $n\geq1$, 
	\begin{align}\label{apinL}
		a\pi_n(\mathbf K)\subset \pi_n\mathcal F_{d}.
	\end{align}
	Fix $n\geq1$ and $p\in\mathbb C[z_1,\ldots,z_n]$.
	Without loss of generality, we assume that
	$p(0)=|p(0)|$ and $\sup\limits_{z\in \pi_n \mathcal F_{d}}|p(z)|=1$.
	Then for each $z\in \mathbf K$ and \(\theta \in [0,2\pi]\), by \eqref{apinL}, we have
	$|p(ae^{i\theta}\pi_nz)|\leq1$, and hence
	$1-\operatorname{Re}p(ae^{i\theta}\pi_nz)\geq0$.
	Write $p=p(0)+\sum\limits_{k=1}^{\deg p}p_k$, where \(p_k\) is the \(k\)-homogeneous component of \(p\).
	Then, for every \(1 \leq k \leq \deg p\), we have
	\begin{equation}\label{eq:homogeneous-bound}
		\begin{aligned}
			|a^kp_k(\pi_nz)|
			&= \left|\frac1\pi\int_0^{2\pi}
			\bigl(1-\operatorname{Re}p(ae^{i\theta}\pi_nz)\bigr)
			e^{-ik\theta}\,d\theta\right|\\
			&\leq
			\frac1\pi\int_0^{2\pi}
			\bigl(1-\operatorname{Re}p(ae^{i\theta}\pi_nz)\bigr)
			\,d\theta\\
			&\overset{(*)}{=}2(1-p(0)),
		\end{aligned}
	\end{equation}
	where \((*)\) is given by the mean value property.
	By (1) in Definition~\ref{def:Gamma},
	$\Gamma([p_k\circ\pi_n])=p_k(T_1,\ldots,T_n)$.
	Therefore, by \eqref{eq:Gamma-compact-bound}
	and \eqref{eq:homogeneous-bound},
	\[
	\begin{aligned}
		\|p(T_1,\ldots,T_n)\|
		&\leq
		p(0)+\sum_{k=1}^{\deg p}\|\Gamma(p_k\circ\pi_n)\|\\
		&\leq
		p(0)+C\sum_{k=1}^{\deg p}
		\sup_{z\in \mathbf K}|p_k(\pi_nz)|\\
		&\leq
		p(0)+2C(1-p(0))\sum_{k=1}^{\infty}a^{-k}\\
		&=1
		=\sup_{z\in \pi_n\mathcal F_{d}}|p(z)|.
	\end{aligned}
	\]
	Let $\mathcal K=\mathcal F_{\mathbf d}$.
	Hence $\mathbb T$ satisfies the compact-set von Neumann inequality.
	
	For each \(n\geq 1\), define \(\tau_n:\operatorname{Hol}(\pi_n K)\longrightarrow B(H)\) by \(\tau_n([g]_n)=\Gamma([g\circ \pi_n])\).
	It is easy to verify that $\tau_n$ is a unital algebra
	homomorphism satisfying the coordinate-preserving
	and continuity properties.
	Moreover, for every \(m\geq n\), we have
	\[
	\tau_m([g\circ \pi_{m,n}]_m)= \Gamma([g\circ \pi_{m,n} \circ \pi_m])=\Gamma([g\circ \pi_{n}])=\tau_n([g]_n).
	\]
	Thus $\{\tau_n\}_{n\ge1}$ is a compatible family of
	holomorphic functional calculi on $\{\pi_nK\}_{n\ge1}$.
	
	\medskip
	$(1)\Rightarrow(2)$.
	For each $[f]\in\operatorname{Hol}(K)$, there exists an open neighborhood \(U_f\) of \(K\) such that $f\in\operatorname{Hol}(U_f)$.
	Then choose $V_f=\iota_N(\Omega) \oplus r(I-P_N)B$ satisfying the conditions of Lemma~\ref{lem:product-neighborhood}.
	Define
	\[
	\Gamma:\operatorname{Hol}(K)\longrightarrow B(H),
	\qquad
	\Gamma([f])=\Gamma_{V_f}(f|_{V_f}).
	\]
	First, we show that \(\Gamma\) is well-defined.
	Let $[f]=[g]\in\operatorname{Hol}(K)$. 
	Then $f$ and $g$ are holomorphic on open neighborhoods $U_f$ and $U_g$ of $K$, respectively, and agree on	an open neighborhood $W\subseteq U_f\cap U_g$ of $K$. 
	Therefore, for any $m \geq 1$,
	\begin{align}\label{f=g}
		f\circ\iota_m |_{\iota_m^{-1}(W)}=g\circ\iota_m|_ {\iota_m^{-1}(W)}.
	\end{align}
	Since $\sup\limits_{z\in K}\|(I-P_m)z\|\to0$ and \(K\subset W\),
	we have $P_mK\subset W$ for all sufficiently large $m$.
	For such $m$, since $P_m=\iota_m\circ\pi_m$, it follows that 
	\begin{align}\label{neighborhood}
		\pi_mK \subset \iota_m^{-1}(W).
	\end{align}
	Therefore, by \eqref{f=g} and \eqref{neighborhood}, we have \([f\circ\iota_m]_m = [g\circ\iota_m]_m\), which implies that \(\tau_m([f\circ\iota_m]_m)=\tau_m([g\circ\iota_m]_m)\).
	Consequently, by Lemma~\ref{lem:bijinGammaV}, we have
	$\Gamma([f])=\Gamma([g])$.
	
	For each \(j \geq 1\), let \([f]=[z_j] \in \operatorname{Hol}(K)\).
	For any \(m\geq j\), by the coordinate-preserving property of $\tau_m$,
	\[
	\tau_m ([f \circ \iota_m]_m) = \tau_m([z_j]_m) = T_j.
	\]
	Therefore, by Lemma~\ref{lem:bijinGammaV}, (1) in Definition~\ref{def:Gamma} holds.
	For each fixed \(\ell^2\)-open neighborhood $U$ of \(K\), choose $V=\iota_N(\Omega) \oplus r(I-P_N)B$ satisfying the conditions of Lemma~\ref{lem:product-neighborhood}. 
	Define $\varphi:\operatorname{Hol}(U)\to\operatorname{Hol}(V)$
	by $f\mapsto f|_V$.
	Then 
	\[
	(\Gamma_V \circ \varphi) (f) = \Gamma([f]).
	\]
	Since both $\varphi$ and $\Gamma_V$ are continuous, (2) in Definition~\ref{def:Gamma} follows.
\end{proof}
\begin{rem}\label{unique}
		Suppose that (1) in Theorem~\ref{thm:weighted-polynomial-calculus} holds. 
		For this compatible family $\{\tau_n\}_{n\ge1}$,
		there exists a unique continuous holomorphic functional
		calculus $\Gamma$ on $K$ such that, for every $n\ge1$
		and every $[g]_n\in\operatorname{Hol}(\pi_nK)$,
		\begin{equation}\label{com}
			\Gamma([g\circ\pi_n])=\tau_n([g]_n).
		\end{equation}
		In fact, by Theorem~\ref{thm:weighted-polynomial-calculus}, there exists a continuous
		holomorphic functional calculus $\Gamma$ on $K$ constructed from $\{\tau_n\}_{n\ge1}$.
		For every \(n \geq 1\) and every \([g]_n \in \operatorname{Hol}(\pi_n(K))\), by Lemma~\ref{lem:bijinGammaV} and  the compatibility of $\{\tau_n\}_{n\ge1}$, we have
	\[
	\Gamma([g \circ \pi_n]) = \lim_{m \to \infty}\tau_{m} ([g \circ \pi_n \circ \iota_m]_m) = \lim_{m \to \infty}\tau_{m} ([g \circ \pi_{m,n}]_m)=\tau_n([g]_n).
	\]
	Thus the existence follows immediately.
	To prove uniqueness, let $\widetilde{\Gamma}$ and $\Gamma$
	be two continuous holomorphic functional calculi on $K$
	satisfying \eqref{com} for the same family
	$\{\tau_n\}_{n\ge1}$.
	For $[f]\in\operatorname{Hol}(K)$, there exists open neighborhood \(U_f\) of \(K\) such that \(f \in \operatorname{Hol}(U_f)\).
	Then choose $V_f=\iota_N(\Omega) \oplus r(I-P_N)B$ satisfying the conditions of Lemma~\ref{lem:product-neighborhood}.
	For every $m\geq N$, by \eqref{com}, we have
	\begin{equation}\label{eq:Gamma-uniqueness}
		\widetilde\Gamma([f\circ P_m])
		=
		\tau_m([f\circ\iota_m]_m)
		=
		\Gamma([f\circ P_m]).
	\end{equation}
	Therefore, by (2) in Definition~\ref{def:Gamma} with respect to $\Gamma$ and $\widetilde{\Gamma}$, Lemma~\ref{lem:fPm to f} and \eqref{eq:Gamma-uniqueness},
	\[
	\widetilde\Gamma([f])
	=
	\lim_{m\to\infty}\widetilde\Gamma([f\circ P_m])
	=
	\lim_{m\to\infty}\Gamma([f\circ P_m])
	=
	\Gamma([f]).
	\]
	This proves uniqueness.
\end{rem}
\begin{cor}
  	Let \(\mathcal M\) be a finitely generated submodule of \(H^2\otimes E\), and \(\mathbb{S}=\left(S_{z_1}^{{\mathcal M}^\perp}, S_{z_2}^{{\mathcal M}^\perp}, \cdots\right)\). Suppose that there is a \(\mu\in \mathbb{B}\), such that
$$
		\operatorname{codim}\{f(\mu)|f\in \mathcal M\}\geq 1.
$$
Then for any nonempty compact subset $K$ of $\ell^2$, $\mathbb S$ does not admit a continuous holomorphic functional calculus on $K$.
\end{cor}
\begin{proof}
Suppose $\mathbb S$ admit a continuous holomorphic functional calculus on a nonempty compact subset $K$. Then by Theorem \ref{thm:weighted-polynomial-calculus}, there is a compact subset $L$ of $\ell^2$ such that
\begin{equation}\label{compact von eq}
\|p(\mathbb S)\|\leq \sup\limits_{z\in L}|p(z)|,~\forall p\in \cal P.
\end{equation}
  As the proof in Proposition \ref{prop1.2}, there is $c>0$ such that
  \begin{equation}\label{norm szi}
    \liminf\limits_{i\to \infty}\left\|S_{z_i}^{{\mathcal M}^\perp}\right\|\geq c.
  \end{equation}
Then by \eqref{compact von eq} and \eqref{norm szi}, we have
$$
c\leq \liminf\limits_{i\to \infty}\left\|S_{z_i}^{{\mathcal M}^\perp}\right\|\leq \lim\limits_{i\to \infty}\sup\limits_{z\in L}|z_i(z)|=0,
$$
which is a contradiction.
\end{proof}

\begin{prop}\label{lem:compactness-of-both-spectra}
	Let $\mathbb T$ be a row-bounded sequence
	of commuting operators satisfying the compact-set
	von Neumann inequality. Then both $\sigma(\mathbb T)$
	and $\sigma_{SZ}(\mathbb T)$ are compact subsets
	of $\ell^2$.
\end{prop}

\begin{proof}
	By Lemma~\ref{lem:compact-ellipsoid} and the compact-set von Neumann inequality, there exists $\mathbf d=(d_1,d_2,\ldots)\in c_0\cap(\mathbb R_{+})^{\mathbb N}$ such that, for every $p\in\mathcal P$,
	\begin{equation}\label{eq:both-spectra-vN}
		\|p(\mathbb T)\| \leq \sup_{z\in \mathcal F_{\mathbf d}}|p(z)|.
	\end{equation}
	For every $\lambda=(\lambda_1,\lambda_2,\ldots)\in\sigma(\mathbb T)\cup\sigma_{SZ}(\mathbb T)$,
	$n\geq1$, and $\{a_i\}_{i=1}^n\subset \mathbb C$,
	we claim that
	\begin{equation}\label{eq:both-spectra-linear}
		\sum_{i=1}^{n}a_i\lambda_i
		\in \sigma\left(\sum_{i=1}^{n}a_iT_i\right).
	\end{equation}
	If $\lambda\in\sigma_{SZ}(\mathbb T)$, then
	$\pi_n\lambda\in\sigma(T_1,\ldots,T_n)$.
	Therefore,
	\[
	\sum_{i=1}^n a_i\lambda_i
	\in
	\left\{
	\sum_{i=1}^n a_i z_i:
	z\in\sigma(T_1,\ldots,T_n)
	\right\}
	=
	\sigma\left(\sum_{i=1}^n a_iT_i\right).
	\]
	If $\lambda\in\sigma(\mathbb T)$, suppose that
	$B=\sum\limits_{i=1}^{n}a_i(T_i-\lambda_iI)$ is invertible.
	Let \(\mathbb S=(a_1B^{-1},\ldots,a_nB^{-1},0,\ldots)\).
	Then
	\[
	\begin{aligned}
		\sum_{i=1}^\infty (T_i-\lambda_iI)S_i
		=\sum_{i=1}^{n}(T_i-\lambda_iI)a_iB^{-1}
		=\left(\sum_{i=1}^{n}a_i(T_i-\lambda_iI)\right)B^{-1}
		=BB^{-1}=I.
	\end{aligned}
	\]
	Therefore, by Corollary~\ref{TS=1}, we have $\mathbb T-\lambda$ is invertible, a contradiction.
	Consequently, \eqref{eq:both-spectra-linear} holds.
	The claim is proved.
	
	By \eqref{eq:both-spectra-linear} and \eqref{eq:both-spectra-vN}, we have
	\[
	\begin{aligned}
		\left|\sum_{i=1}^{n}a_i\lambda_i\right|
		\leq
		\left\|\sum_{i=1}^{n}a_iT_i\right\|
		\leq
		\sup_{z\in \mathcal F_{\mathbf d}}
		\left|\sum_{i=1}^{n}a_iz_i\right|
		\leq\left(\sum_{i=1}^{n}|a_i|^2d_i^2\right)^{\frac 12}.
	\end{aligned}
	\]
	Taking $a_i=\frac{\overline{\lambda_i}}{d_i^2}$, we obtain \(\sum\limits_{i=1}^{n}\frac{|\lambda_i|^2}{d_i^2}\leq1\).
	Then let $n\to\infty$. We obtain $\lambda\in \mathcal F_{\mathbf d}$.
	Therefore, 
	\begin{align}\label{sigma subset}
		\sigma(\mathbb T)\cup\sigma_{SZ}(\mathbb T)
		\subseteq \mathcal F_{\mathbf d}
	\end{align}
	By Theorem~\ref{bounded closed}, $\sigma(\mathbb T)$ is closed in $\ell^2$.
	Moreover, by \cite[Proposition~1.9]{SZ1974},
	$\sigma_{SZ}(\mathbb T)$ is closed in $\ell^2$.
	Therefore, by \eqref{sigma subset}, both $\sigma(\mathbb T)$ and $\sigma_{SZ}(\mathbb T)$ are compact sets of \(\ell^2\).
\end{proof}
\begin{rem}\label{rem:spectrum-calculus}
		We justify the assertion in the introduction concerning $K=\sigma(\mathbb T)$.
		If $\mathbb T$ satisfies the compact-set von Neumann inequality,
		then $\sigma(\mathbb T)$ is compact by
		Proposition~\ref{lem:compactness-of-both-spectra}.
		Together with the assumed compatibility,
		Theorem~\ref{thm:weighted-polynomial-calculus}
		yields the desired functional calculus.
		Conversely, the proof of $(2)\Rightarrow(1)$ in
		Theorem~\ref{thm:weighted-polynomial-calculus}
		gives the compact-set von Neumann inequality without
		assuming compactness of $K$.
		Thus $\sigma(\mathbb T)$ is compact by
		Proposition~\ref{lem:compactness-of-both-spectra},
		and $\tau_n([g]_n)=\Gamma([g\circ\pi_n])$
		defines the required compatible family.
		
		For Theorem~\ref{1.12}, the implication
		$(2)\Rightarrow(1)$ follows from the proof of
		Theorem~\ref{thm:weighted-polynomial-calculus}.
		Suppose that $\textup{(1)}$ holds.
		By Proposition~\ref{lem:compactness-of-both-spectra},
		$\sigma_{SZ}(\mathbb T)$ is a compact subset of $\ell^2$.
		For every $n\ge1$, since
		$\pi_n\sigma_{SZ}(\mathbb T)=\sigma(T_1,\ldots,T_n)$,
		we define
		\[
		\tau_n:
		\operatorname{Hol}(\pi_n\sigma_{SZ}(\mathbb T))
		\longrightarrow B(H),
		\qquad
		\tau_n([g]_n)=g(T_1,\ldots,T_n).
		\]
		By the properties of the finite-dimensional Taylor functional
		calculus, the family $\{\tau_n\}_{n\ge1}$ is compatible.
		Then Theorem~\ref{thm:weighted-polynomial-calculus}
		and Remark~\ref{unique} yield $(2)$.
\end{rem}

\begin{thm}
Let $\mathbb T=(T_i)_{i\geq1}$ be a row-bounded sequence of commuting operators.
Suppose that $\mathbb{T}$ admits a continuous holomorphic functional calculus $\Gamma$ on $\sigma_{SZ}(\mathbb T)$ as in Theorem \ref{1.12}(2). Let $F(z)=(f_i(z))_{i\geq1}\subset\operatorname{Hol}(\sigma_{SZ}(\mathbb T))$ such that 
$$
\Gamma(F):=(\Gamma(f_1),\Gamma(f_2),\cdots)
$$
is a row-bounded sequence of commuting operators on $H$. Then
\begin{equation}\label{eq:SM-product}
\sigma_{SZ}(\Gamma(F))=F(\sigma_{SZ}(\mathbb T)).
\end{equation}
\end{thm}

\begin{proof}

We first prove that, for every $m\geq1$,
\begin{equation}\label{eq:SM-finite}
\sigma(\Gamma(f_1),\ldots,\Gamma(f_m))
=
\left\{
(f_1(z),\ldots,f_m(z)):z\in \sigma_{SZ}(\mathbb T)
\right\}.
\end{equation}
Fix $m\geq1$.
Since $\Gamma$ is an algebra homomorphism and
$\Gamma(z_j)=T_j$, the operators $\Gamma(f_i)$ and $T_j$
commute with one another. For $n\geq1$, let
\[
S_n=\sigma(\Gamma(f_1),\ldots,\Gamma(f_m),T_1,\ldots,T_n),
\]
and 
\[
S=\sigma_{SZ}(\Gamma(f_1),\ldots,\Gamma(f_m),T_1,T_2,\ldots).
\]
We claim that 
\begin{equation}\label{mu=fz}
\mu_i=f_i(z),\quad\forall 1\leq i\leq m,~(\mu_1,\cdots,\mu_m,z)\in S.
\end{equation}
In fact, 
take $(\mu_1,\cdots,\mu_m,z)\in S$, then by Lemma \ref{project sz}, 
\begin{equation}\label{z in sz T}
z\in \sigma_{SZ}(\mathbb{T}).
\end{equation}
For \(1\le i\le m\) and all sufficiently large $n$, since \(f_i\circ\iota_n\in \operatorname{Hol}(\sigma(T_1,\cdots,T_n))\), by the spectral projection theorem, \(f_i\circ\iota_n\) can be regarded as an element of \(\operatorname{Hol}(S_n)\).
Then by (2) in Definition~\ref{def:Gamma} and Lemma~\ref{lem:fPm to f},
\begin{equation}\label{eq:SM-coordinate-approximation}
\begin{aligned}
&\lim\limits_{n\to\infty}\|\Gamma(f_i)-f_i\circ\iota_n(\Gamma(f_1),\ldots,\Gamma(f_m),T_1,\ldots,T_n)\|\\
=&\lim\limits_{n\to\infty}\|\Gamma(f_i)-f_i\circ\iota_n(T_1,\ldots,T_n)\|\\
=&\lim\limits_{n\to\infty}\|\Gamma(f_i)-\Gamma(f_i\circ P_n)\|
=0.
\end{aligned}
\end{equation}
By Lemma \ref{project sz}, 
\begin{equation}\label{mu pinz}
(\mu_1,\cdots,\mu_m,\pi_nz)\in \pi_{m+n} S= S_n.
\end{equation}
Since $z_i-f_i\circ\iota_n\in \operatorname{Hol}(S_n)$, the spectral mapping theorem gives
\[
\begin{aligned}
\mu_i-f_i(P_nz)=&\mu_i-f_i\circ\iota_n(\mu_1,\cdots,\mu_m,\pi_nz)\\
\in&
(z_i-f_i\circ\iota_n)(S_n)\\
=&
\sigma\!\left(
\Gamma(f_i)-f_i\circ\iota_n(\Gamma(f_1),\ldots,\Gamma(f_m),T_1,\ldots,T_n)
\right).
\end{aligned}
\]
Consequently by \eqref{eq:SM-coordinate-approximation},
$
\lim\limits_{n\to\infty} f_i(P_nz)=\mu_i.
$
Since $\lim\limits_{n\to\infty} P_nz=z$ and $f_i$ is continuous, we obtain 
$$
\mu_i=f_i(z),\quad\forall 1\leq i\leq m.
$$
The claim is proved. Hence by \eqref{mu=fz} and \eqref{z in sz T}, 
$$
S\subseteq\{(f_1(z),\ldots,f_m(z),z):z\in \sigma_{SZ}(\mathbb T)\}.
$$
Conversely, for fixed $z\in \sigma_{SZ}(\mathbb T)$, by Lemma \ref{project sz}, we can take 
$$
(\mu_1,\cdots,\mu_m)\in \sigma(\Gamma(f_1),\ldots,\Gamma(f_m))
$$ 
with $(\mu_1,\cdots,\mu_m,z)\in S$. Then by \eqref{mu=fz}, $\mu=(f_1(z),\ldots,f_m(z))$, which implies that 
$$\{(f_1(z),\ldots,f_m(z),z):z\in \sigma_{SZ}(\mathbb T)\}\subseteq S.$$ 
Hence
\begin{equation}\label{S=fmz z}
 S=\{(f_1(z),\ldots,f_m(z),z):z\in \sigma_{SZ}(\mathbb T)\}.
\end{equation}
Then
by \eqref{S=fmz z} and Lemma \ref{project sz},
$$
\left\{
(f_1(z),\ldots,f_m(z)):z\in \sigma_{SZ}(\mathbb T)
\right\}=\pi_mS=\sigma(\Gamma(f_1),\ldots,\Gamma(f_m)).
$$

We now prove \eqref{eq:SM-product}.
For every $z\in \sigma_{SZ}(\mathbb T)$ and $m\geq 1$, \eqref{eq:SM-finite} gives
\[
(f_1(z),\ldots,f_m(z))
\in\sigma(\Gamma(f_1),\ldots,\Gamma(f_m)).
\]
Thus $F(z)\in\sigma_{SZ}(\Gamma(F))\subset\ell^2$, where the last inclusion follows from the row boundedness of $\Gamma(F)$.
Therefore, 
$$
\begin{aligned}
F(\sigma_{SZ}(\mathbb T))=&\{(f_1(z),\cdots):z\in\sigma_{SZ}(\mathbb T)\}\\
\subset&\bigcap\limits_{m=1}^\infty\left(\{(f_1(z),\cdots,f_m(z)):z\in\sigma_{SZ}(\mathbb T)\}\times \ell^2\right)\\
=&\bigcap\limits_{m=1}^\infty\left(\sigma(\Gamma(f_1),\ldots,\Gamma(f_m))\times \ell^2\right)\\
=&\sigma_{SZ}(\Gamma(F)).
\end{aligned}
$$
Conversely, let
$\lambda\in\sigma_{SZ}(\Gamma(F))$.
For each $m\geq1$, define
\[
L_m=
\{z\in \sigma_{SZ}(\mathbb T):f_i(z)=\lambda_i,\ 1\leq i\leq m\}.
\]
By \eqref{eq:SM-finite}, $L_m$ is a non-empty compact subset of $\sigma_{SZ}(\mathbb T)$ and satisfies
$L_{m+1}\subseteq L_m$. Hence
$
\bigcap\limits_{m=1}^\infty L_m\neq\varnothing.
$
Take $z\in \bigcap\limits_{m=1}^\infty L_m$, then $F(z)=\lambda$.
Thus
$
\sigma_{SZ}(\Gamma(F))\subseteq F(\sigma_{SZ}(\mathbb T)).
$
\end{proof}

\section{Fredholm property and essential spectrum}
We now turn to the Fredholm property and the essential spectrum of row‑bounded sequences.  First we recall the definition of an unbounded Fredholm operator.

\begin{defn}\cite{MS}\label{def unbounded fredholm}
  Let $X$ and $Y$ be Banach spaces, and let $T$ be a densely defined, closed linear operator from $\mathrm{Dom}(T)\subseteq X$ to $Y$. We say that $T$ is a Fredholm operator if it satisfies:
  \begin{itemize}
    \item [(1)] $\dim \ker T<\infty$;
    \item [(2)] $\dim \ker T^*<\infty$;
    \item [(3)] $\operatorname{Im} T$ is closed.
  \end{itemize}
\end{defn}
For a row-bounded sequence $\mathbb{T}$, we have the following theorem.
\begin{thm}\label{3equal}
The following statements are equivalent.
\begin{itemize}
  \item[(1)] $\mathbb{T}$ is Fredholm;
  \item[(2)] $\dim \operatorname{ker} \partial_{\mathbb{T}} / \operatorname{Im} \partial_{\mathbb{T}}<\infty$;
  \item[(3)]$\widehat{\mathbb{T}}$ is Fredholm,
\end{itemize}
and in this case, $\operatorname{ind}(\widehat{\mathbb{T}})=\operatorname{ind}(\mathbb{T})$.
\end{thm}
\begin{proof}
(1)$\Rightarrow$(2): Since $\mathbb{T}$ is Fredholm, there is an integer $N,c> 0$ such that
\begin{equation}\label{f control 2}
\operatorname{dim} \operatorname{ker} \partial_{\mathbb{T}, n} / \operatorname{Im} \partial_{\mathbb{T}, n+1}<\infty, \forall n \leq N ~\text{and}~ \operatorname{dim} \operatorname{ker} \partial_{\mathbb{T}, n} / \operatorname{Im} \partial_{\mathbb{T}, n+1}=0, \forall n>N.
\end{equation}
Moreover, for every $n>N$ and $y_{n}\in \ker \partial_{\mathbb{T}, n}$, there exists an $x_{n+1} \in \Lambda^{n+1}(H)$ such that $\partial_{\mathbb{T}, n+1}x_{n+1}=y_{n}$ and $\|x_{n+1}\|\leq c\|y_{n}\|$. By \eqref{f control 2}, $\operatorname{Im} \partial_{\mathbb{T}, n}$ is closed for $1\leq n\leq N$. Thus there is $c_1>0$ such that for every $n\leq N$ and $y_{n}\in \operatorname{Im} \partial_{\mathbb{T}, n+1}$, there exists an $x_{n+1} \in \Lambda^{n+1}(H)$ such that $\partial_{\mathbb{T}, n+1}x_{n+1}=y_{n}$ and $\|x_{n+1}\|\leq c_1\|y_{n}\|$. Then by Lemma $\ref{Imclosed}$, $\operatorname{Im}\partial_{\mathbb{T}}$ is closed. So
$$
\operatorname{Im}\partial_{\mathbb{T}}=\bigoplus\limits_{k=1}^{\infty}\operatorname{Im}\partial_{\mathbb{T},k}.
$$
Hence
\begin{equation}\label{dim ker frac}
\begin{aligned}
\operatorname{dim}\operatorname{ker}\partial_{\mathbb{T}}/\operatorname{Im}\partial_{\mathbb{T}}
&=\operatorname{dim}\left(\bigoplus\limits_{k=0}^{\infty}\operatorname{ker}\partial_{\mathbb{T},k}\right)\ominus
\left(\bigoplus\limits_{k=1}^{\infty}\operatorname{Im}\partial_{\mathbb{T},k}\right)\\
&=\operatorname{dim}\bigoplus\limits_{k=0}^{\infty}\left(\operatorname{ker}\partial_{\mathbb{T},k}\ominus\operatorname{Im}\partial_{\mathbb{T},k+1}\right)\\
&=\sum_{k=0}^{\infty}\operatorname{dim}H_\mathbb{T}^{k}<\infty.
\end{aligned}
\end{equation}
(2)$\Rightarrow$(1):At first we claim that $\operatorname{Im}\partial_{\mathbb{T}}$ is closed. Since $\dim \ker \partial_{\mathbb{T}}/ \operatorname{Im} \partial_{\mathbb{T}}<\infty$, there exists a finite-dimensional subspace $Y_{0}$ of $\operatorname{ker} \partial_{\mathbb{T}}$ such that $\operatorname{Im} \partial_{\mathbb{T}} \oplus Y_{0}=\operatorname{ker} \partial_{\mathbb{T}}$. Considering $\Lambda(H) \times Y_{0}$ as a Banach space with norm
$$
\|(x, y)\|=\|x\|+\|y\|.
$$
Then $\operatorname{Dom}\left(\partial_{\mathbb{T}}\right) \times Y_{0}$ is dense in $\Lambda(H) \times Y_{0}$. Let $S$ be the linear map
$$
\begin{aligned}
S: D\left(\partial_{\mathbb{T}}\right) \times Y_{0} &\rightarrow \operatorname{ker} \partial_{\mathbb{T}} \\
(x, y) &\mapsto \partial_{\mathbb{T}} x+y.
\end{aligned}
$$
By Lemma $\ref{partial closed}$, $\partial_{\mathbb{T}}$ is closed. Hence it is easy to see that $S$ is closed. Notice that 
\begin{equation}\label{directsum}
  \operatorname{Im} S=\operatorname{Im}\partial_{\mathbb{T}}\oplus Y_{0}= \operatorname{ker} \partial_{\mathbb{T}}
\end{equation}
is closed. Then by Closed Range Theorem, there is a constant $c>0$ such that for every $z \in \operatorname{Im} \partial_{\mathbb{T}}\subseteq \operatorname{Im}S$, there exist $x\in \operatorname{Dom}(\partial_{\mathbb{T}})$ and $y\in Y_{0}$ satisfying
$$
\partial_{\mathbb{T}} x+y=S(x, y)=z \text { and }\|x\|+\|y\| \leq c\|z\|.
$$
Then $y=z-\partial_{\mathbb{T}} x\in \operatorname{Im} \partial_{\mathbb{T}}$. It follows from
\(\operatorname{Im}\partial_{\mathbb T}\oplus Y_0\) is direct that \(y=0\). Hence
\begin{equation}\label{Tx=y}
\partial_{\mathbb{T}}x=z \text { and }\|x\|\leq c\|z\|.
\end{equation}
By Closed Range Theorem, we have $\operatorname{Im} \partial_{\mathbb{T}}$ is closed. The claim is proved. 
Hence
$$
\operatorname{Im} \partial_{\mathbb{T}}=\bigoplus\limits_{k=1}^{\infty} \operatorname{Im} \partial_{\mathbb{T}, k}.
$$
Since $\operatorname{dim} \operatorname{ker} \partial_{\mathbb{T}} / \operatorname{Im} \partial_{\mathbb{T}}<\infty$, we have $\sum\limits_{k=0}^{\infty}\operatorname{dim}H_\mathbb{T}^{k}<\infty$. So there is an integer $N > 0$ such that
\begin{equation}\label{f control}
\operatorname{dim} \operatorname{ker} \partial_{\mathbb{T}, n} / \operatorname{Im} \partial_{\mathbb{T}, n+1}<\infty, \forall n \leq N ~\text{and}~ \operatorname{dim} \operatorname{ker} \partial_{\mathbb{T}, n} / \operatorname{Im} \partial_{\mathbb{T}, n+1}=0, \forall n>N.
\end{equation}
By $(1)\Leftrightarrow(2)$ of Lemma \ref{lem3.6}, for every $n$ and $y_{n}\in \operatorname{Im} \partial_{\mathbb{T}, n+1}$, there exists an $x_{n+1} \in \Lambda^{n+1}(H)$ such that $\partial_{\mathbb{T}, n+1}x_{n+1}=y_{n}$ and $\|x_{n+1}\|\leq c\|y_{n}\|$. Thus by $(\ref{f control})$, $\mathbb{T}$ is Fredholm.\\
(2)\(\Rightarrow\)(3): 
Assume (2) holds. As shown in the proof of $(2)\Rightarrow(1)$,
$\operatorname{Im}\partial_{\mathbb T}$ is closed. Thus by Lemma \ref{lem3.6}, $\operatorname{Im} \widehat{\mathbb{T}}$ is closed. Moreover,
\begin{equation}\label{kerdT frac+}
\begin{aligned}
\ker \partial_{\mathbb{T}}\cap \Lambda^+(H)
=& (\operatorname{Im} \partial_{\mathbb{T}}\oplus( \ker\partial_{\mathbb{T}} \cap (\operatorname{Im} \partial_{\mathbb{T}})^\perp) )\cap \Lambda^+(H)\\
=& (\operatorname{Im} \partial_{\mathbb{T}}\cap \Lambda^+(H))\oplus ( \ker\partial_{\mathbb{T}} \cap \ker\partial_{\mathbb{T}}^*\cap \Lambda^+(H))\\
\overset{(*)}{=}&(\operatorname{Im} \partial_{\mathbb{T}}\cap \Lambda^+(H))\oplus  \ker \widehat{\mathbb{T}},
\end{aligned}
\end{equation}
where $(*)$ follows from \eqref{hat T x eq dx d*x}. Hence by \eqref{kerdT frac+},
$$
\begin{aligned}
\dim \ker \widehat{\mathbb{T}}
=&\dim (\ker \partial_{\mathbb{T}}\cap \Lambda^+(H))\ominus (\operatorname{Im} \partial_{\mathbb{T}}\cap \Lambda^+(H))\\
\overset{(\star)}{=}&\dim \bigoplus\limits_{k=0}^{\infty}\left(\operatorname{ker}\partial_{\mathbb{T},2k}\ominus\operatorname{Im}\partial_{\mathbb{T},2k+1}\right)\\
=&\dim \bigoplus\limits_{k \text{ even}} H_\mathbb{T}^{k}<\infty,
\end{aligned}
$$
where $(\star)$ follows from \eqref{dim ker frac}.
By the same reasoning, $\dim \ker \widehat{\mathbb{T}}^*=\dim \bigoplus\limits_{k \text{ odd}} H_\mathbb{T}^{k}<\infty$. Therefore, \(\widehat{\mathbb{T}}\) is Fredholm, and 
\[
\operatorname{ind}(\widehat{\mathbb{T}}) = \dim \ker \widehat{\mathbb{T}} - \dim \ker \widehat{\mathbb{T}}^* = \dim \bigoplus\limits_{k \text{ even}} H_\mathbb{T}^{k} - \dim \bigoplus\limits_{k \text{ odd}} H_\mathbb{T}^{k} = \sum\limits_{k=0}^\infty (-1)^k \dim H_\mathbb{T}^{k} = \operatorname{ind}(\mathbb{T}).
\]
(3)\(\Rightarrow\)(2): 
Assume (3) holds. Then \(\ker \widehat{\mathbb{T}}\) and \(\ker \widehat{\mathbb{T}}^*\) are finite-dimensional, and \(\operatorname{Im} \widehat{\mathbb{T}}\) is closed. Thus by Lemma \ref{lem3.6}, $\operatorname{Im} \partial_{\mathbb{T}}$ and $\operatorname{Im} \partial_{\mathbb{T}}^*$ are closed. Hence \eqref{kerdT frac+} holds in this case as well. Moreover,
$$
\begin{aligned}
\ker \partial_{\mathbb{T}}\cap \Lambda^-(H)
=& (\operatorname{Im} \partial_{\mathbb{T}}\oplus( \ker\partial_{\mathbb{T}} \cap (\operatorname{Im} \partial_{\mathbb{T}})^\perp) )\cap \Lambda^-(H)\\
=& (\operatorname{Im} \partial_{\mathbb{T}}\cap \Lambda^-(H))\oplus ( \ker\partial_{\mathbb{T}} \cap \ker\partial_{\mathbb{T}}^*\cap \Lambda^-(H))\\
\overset{(*)}{=}&(\operatorname{Im} \partial_{\mathbb{T}}\cap \Lambda^-(H))\oplus  \ker \widehat{\mathbb{T}}^*,
\end{aligned}
$$
where $(*)$ follows from \eqref{hat T x eq dx d*x}.
Consequently,
$$
\begin{aligned}
\dim\ker \partial_{\mathbb{T}} / \operatorname{Im} \partial_{\mathbb{T}}
=&\dim\frac{(\ker \partial_{\mathbb{T}}\cap \Lambda^+(H))\oplus (\ker \partial_{\mathbb{T}}\cap \Lambda^-(H))}
{(\operatorname{Im} \partial_{\mathbb{T}}\cap \Lambda^+(H))\oplus (\operatorname{Im} \partial_{\mathbb{T}}\cap \Lambda^-(H))}\\
=&\dim\frac{(\operatorname{Im} \partial_{\mathbb{T}}\cap \Lambda^+(H))\oplus  \ker \widehat{\mathbb{T}}\oplus (\operatorname{Im} \partial_{\mathbb{T}}\cap \Lambda^-(H))\oplus  \ker \widehat{\mathbb{T}}^*}
{(\operatorname{Im} \partial_{\mathbb{T}}\cap \Lambda^+(H))\oplus (\operatorname{Im} \partial_{\mathbb{T}}\cap \Lambda^-(H))}\\
=&\dim \ker \widehat{\mathbb{T}}\oplus \ker \widehat{\mathbb{T}}^*<\infty.
\end{aligned}
$$
\end{proof} 
Similar to Corollary \ref{invert continious}, by \cite[Theorem 7.11]{MS} we have the following corollary.
\begin{cor}\label{index continious}
  Let $\mathbb{T}=(T_1,T_2,\cdots)$ be a row-bounded sequence on $H$. If $\mathbb{T}$ is Fredholm, then there exists an $\varepsilon>0$ such that for every row-bounded sequence $\mathfrak{T}=(\mathfrak{T}_{1}, \mathfrak{T}_{2}, \cdots)$ satisfying 
  $$
  \left\|\left(\widehat{\mathbb{T}}-\widehat{\mathfrak{T}}\right)x\right\|<\varepsilon\|x\|,~\forall x\in \Lambda_{\mathrm{fin}}^+(H),
  $$
  $\mathfrak{T}$ is Fredholm, and $\operatorname{ind}(\mathbb{T})=\operatorname{ind}(\mathfrak{T})$.
\end{cor}
By Theorem \ref{3equal} and Corollary \ref{index continious}, we have the following theorem.
\begin{thm}
  If $\mathbb{T}$ is a row-bounded sequence of commuting operators, then $\sigma_{e}(\mathbb{T})$ is a bounded closed set in $\ell^2$.
\end{thm}
\begin{rem}\label{rem:essential-spectrum-realization}
  Conversely, for every nonempty bounded closed subset \(E\) of \(\ell^2\), let \(H = \ell^2(E) \otimes \ell^2\) and \(T_i = M_{z_i} \otimes I\). Since \(E\) is bounded, \(\mathbb{T} = (T_1, T_2, \cdots)\) is row-bounded. If \(\lambda \notin E\), \(\mathbb{T} - \lambda\) is clearly invertible. If $\lambda\in E$, then
\[
\operatorname{Im}\partial_{\mathbb T-\lambda,1}
\subseteq\{f\in H:f(\lambda)=0\}.
\]
Hence evaluation at $\lambda$ induces a surjective
linear map from $H^0(\mathbb T-\lambda)$ onto $\ell^2$.
Thus $H^0(\mathbb T-\lambda)$ is infinite-dimensional.
Therefore, $\mathbb T-\lambda$ is not Fredholm,
and $\sigma_e(\mathbb T)=E$.
\end{rem}
Now we are ready to prove Theorem \ref{thm3.1.2}.
\vskip2mm
\noindent\emph{Proof of Theorem \ref{thm3.1.2}.}  By Theorem \ref{thm3.1}, it suffices to prove $\sigma_{SZ}(\mathbb{T})\subseteq \sigma_e(\mathbb{T})$. Take $\lambda\notin \sigma_e(\mathbb{T})$. By Corollary \ref{index continious} and Lemma \ref{partial norm hatT}, there exists $\varepsilon>0$ such that a row-bounded sequence $\mathfrak{T}$ is Fredholm whenever $\|\partial_{\mathfrak{T}}-\partial_{\mathbb{T}-\lambda}\|<\varepsilon$. Since the sequence $\left\{\sum\limits_{i=1}^{k}T_i\otimes E_i^*\right\}_{k\geq1}$ converges in norm, by Lemma \ref{lem:bounded-partial-normal}, 
  \begin{equation}\label{T-lambda conver}
 \lim\limits_{n,m\to\infty} \left\|\sum\limits_{i=m+1}^{n}(T_i-\lambda_iI_H)\otimes E_i^*\right\|=0.
  \end{equation}
  For $n>0$, ${\bf\lambda}^{(n)}=(\lambda_1,\lambda_2,\cdots,\lambda_n,0,0,\cdots)$, write
  $$
  \mathbb{T}^{(n)}-\lambda^{(n)}=(T_1-\lambda_1,\cdots,T_n-\lambda_n,0,0,\cdots).
  $$
  Then by \eqref{T-lambda conver}, there exists $N>0$ such that $\|\partial_{\mathbb{T}^{(n)}-\lambda^{(n)}}-\partial_{\mathbb{T}-\lambda}\|<\varepsilon$, which implies that $\mathbb{T}^{(n)}-\lambda^{(n)}$ is Fredholm. Hence by Lemma \ref{Tn and T'}, we have $(T_1-\lambda_1,\cdots,T_n-\lambda_n)$ is invertible, i.e. $\lambda\notin \sigma_{SZ}(\mathbb{T})$.

\hfill$\Box$

\section{Column boundedness vs. boundedness of Koszul boundary operators }
In this section, we prove the equivalence of column boundedness and boundedness of Koszul boundary operators. 
To do this, we need some preparations. 

Let \(\mathcal{H}\) be a reproducing kernel Hilbert space on \(\mathbf{X}\) with kernel \(K\), and let \(\mathcal{K}\) be another reproducing kernel Hilbert space on \(\mathbf{X}\). Let \(\mathcal{E}\) and \(\mathcal{F}\) be Hilbert spaces. \(\mathrm{Mult}(\mathcal{H}\otimes \mathcal{E},\mathcal{K}\otimes \mathcal{F})\) is defined  to be the space of all \(B(\mathcal{E},\mathcal{F})\)-valued functions on \(\mathbf{X}\) that pointwise multiply \(\mathcal{H}\otimes \mathcal{E}\) into \(\mathcal{K}\otimes \mathcal{F}\). The closed graph theorem easily implies that every element \(\Phi\in\mathrm{Mult}(\mathcal{H}\otimes \mathcal{E},\mathcal{K}\otimes \mathcal{F})\) induces a bounded multiplication operator \(M_\Phi:\mathcal{H}\otimes \mathcal{E}\to \mathcal{K}\otimes \mathcal{F}\) with
\begin{equation}\label{mult op}
	(M_\Phi f)(z)=\Phi(z)f(z).
\end{equation}
\(\mathcal H\) is said to satisfy the \emph{column-row property} with constant \(c\ge1\) if whenever \(\{\phi_n\}_{n\ge1}\) is a sequence in \(\mathrm{Mult}(\mathcal H)\) with 
\[
\|[M_{\phi_1},M_{\phi_2},\cdots]^T\|\leq 1,
\]
then
\[
\|[M_{\phi_1},M_{\phi_2},\cdots]\|\leq c.
\]
The reproducing kernel Hilbert space \(\mathcal H\) is said to be irreducible if the underlying set \(\mathbf{X}\) cannot be partitioned into two non-empty disjoint sets \(\mathbf{X}_1\) and \(\mathbf{X}_2\) such that
$$
K(x_1,x_2)=0,~\forall x_1\in \mathbf{X}_1,~x_2\in \mathbf{X}_2.
$$
By a \emph{normalized complete Pick space}, we mean a reproducing kernel Hilbert space whose kernel is of the form
\[
K(z,w)=\frac{1}{1-\langle b(z),b(w)\rangle},
\]
where \(b:\mathbf{X}\to\mathbb B_d\) for some \(d\in\mathbb N\cup\{\infty\}\), and there exist $x_0\in \mathbf{X}$ such that $b(x_0)=0$.
For such spaces, Hartz proved the following theorem.
\begin{thm}\cite[Theorem 1.2]{Hartz2023}\label{thm:hartz}
	Every normalized complete Pick space satisfies the column-row property with constant \(1\).
\end{thm}
Hartz constructed an example showing that the row norm can be strictly smaller than the column norm \cite[Example 3.15]{Hartz2023}. This naturally raises the question of how to characterize column boundedness. Let $\{\varphi_i\}_{i=1}^\infty\subset \mathrm{Mult}(\mathcal{H})$, and $\mathbb{T}=(M_{\varphi_1},M_{\varphi_2},\ldots)$. Then 
$
\partial_{\mathbb{T}}
$
is well-defined on the dense subspace 
\[
\Lambda_{\mathrm{alg}}(\mathcal{H})=\operatorname{span}\Bigl\{f\otimes e_{i_1}\wedge\cdots\wedge e_{i_p}:\ f\in \mathcal{H},\ i_1<\cdots<i_p,\ p\ge0\Bigr\},
\]
and hence $\partial_{\mathbb{T}}$ is  densely defined  on \(\Lambda(\mathcal{H})\).  We have the following theorem.
\begin{thm}\label{thm:main}
	Let $\mathcal{H}$ be a normalized complete Pick space, $\{\varphi_i\}_{i=1}^\infty\subset \mathrm{Mult}(\mathcal{H})$, and 
	\[
	\mathbb{T}=(M_{\varphi_1},M_{\varphi_2},\ldots).
	\]
	Then the following statements are equivalent.
	\begin{itemize}
		\item[(1)] $[M_{\varphi_1},M_{\varphi_2},\ldots]^{\mathrm T}$ is bounded.
		\item[(2)] $\partial_{\mathbb{T}}$ extends to a bounded linear operator on \(\Lambda(\mathcal H)\).
	\end{itemize}
	Moreover, in this case, 
	\[
	\|[M_{\varphi_1},M_{\varphi_2},\ldots]^{\mathrm T}\|
	=
	\|\partial_{\mathbb{T}}\|.
	\]
\end{thm}

Let $\mathcal{H}$ be a reproducing kernel Hilbert space on $\mathbf{X}$ with kernel $K$. For $\mathbf{Y}\subset \mathbf{X}$, let $\mathcal{H}|_\mathbf{Y}$ be the reproducing kernel Hilbert space on $\mathbf{Y}$ with kernel $K|_{\mathbf{Y}\times \mathbf{Y}}$. The following lemmas are from \cite{Hartz2023}; we omit the proof.
\begin{lem}\label{mult crit}\cite[Lemma 2.1]{Hartz2023}
	Let \(\mathcal H\) be a reproducing kernel Hilbert space on \(\mathbf{X}\) with kernel \(K\), and let \(\mathcal E_1,\mathcal E_2\) be Hilbert spaces. A function \(\Phi:\mathbf{X}\to B(\mathcal E_1,\mathcal E_2)\) is a contractive multiplier, i.e. \(\|M_\Phi\|\le1\) on \(\mathcal H\otimes\mathcal E_1\to\mathcal H\otimes\mathcal E_2\), if and only if for every finite set \(\{z_1,\cdots,z_n\}\subset \mathbf{X}\), the operator matrix
	\[
	\bigl[K(z_i,z_j)(I_{\mathcal E_2}-\Phi(z_i)\Phi(z_j)^*)\bigr]_{i,j=1}^n
	\]
	is positive.
\end{lem}

\begin{lem}\label{3.12a}\cite[Lemma 3.12(a)]{Hartz2023}
	Let \(F=E\cup\{0\}\) be finite, $\mathcal{E}$ be a Hilbert space and $\Phi$ be a multiplier in $\mathrm{Mult}(H_d^2|_F\otimes \mathcal{E})$ with $\Phi(0)=0$. Then $M_\Phi$ is a contraction on  $H_d^2|_F\otimes \mathcal{E}$ if and only if 
	$$
	\bigl[K(z^{(i)},z^{(j)})\bigl(I_{\mathcal{E}}-\Phi(z^{(i)})\Phi(z^{(j)})^*\bigr)-I_{\mathcal{E}}\bigr]_{i,j=1}^n\geq 0
	$$
	for any finite subset $\{z^{(k)}\}_{k=1}^n\subset E$, where $K(z,w)=\frac{1}{1-\langle z,w\rangle}$ is the reproducing kernel of $H_d^2$.
\end{lem}
Recall that for \(d \in \mathbb{N}\), \(\Lambda_d = \bigoplus\limits_{k=0}^d \Lambda_d^k\) is the antisymmetric Fock space over \(\mathbb{C}^d\).
For \(a\in\mathbb C^d\), define
$$
\partial^{(a)}=\sum\limits_{i=1}^d a_iE_i^*\in B(\Lambda_d).
$$
The following lemma is an easy consequence of the CAR relations \eqref{anticommutationrelations}; we omit the proof.
\begin{lem}\label{lem:partial-relations}
	For \(u,v\in \mathbb B_d\), set \(X=\partial^{(u)}\) and \(Y=\partial^{(v)}\). Then the following identities hold.
	\begin{enumerate}
		\item[(1)] \(XY+YX=0\). In particular, $X^2=Y^2=0$.
		\item[(2)] \(X^*Y+YX^*=\langle v,u\rangle I_{\Lambda_d}\).
		\item[(3)] $YX^*Y=\langle v,u\rangle Y$.
		\item[(4)] $XX^*Y+YX^*X=\langle v,u\rangle X+\|u\|^2Y$.
		\item[(5)] $XX^*YX^*X=\langle v,u\rangle\|u\|^2X$.
		\item[(6)] $\|X\|=\|u\|$.
	\end{enumerate}
\end{lem}
The following lemma is a straightforward consequence of the spectral theorem; its proof is omitted.
\begin{lem}\label{lem:sqrt-one-minus-Q}
	Let \(Q\) be a self-adjoint operator with \(\sigma(Q)\subseteq\{0,\lambda\}\) for some \(0<\lambda\leq 1\). Then
	\begin{equation}\label{eq:sqrt-one-minus-Q}
		\sqrt{I-Q}=I-\bigl(1-\sqrt{1-\lambda}\bigr)\frac{Q}{\lambda}.
	\end{equation}
\end{lem}

Let \(H\) be a Hilbert space and let \(A\in B(H)\) be a strict contraction. For \(X,Y\in B_1(H)\), define
\[
\Theta_A(X)=A-\Delta_A (I-XA^*)^{-1}X\Delta_{A^*},
\]
where \(\Delta_A=(I-AA^*)^{1/2}\) and $B_1(H)$ is the closed unit ball of $B(H)$. Then we have the following lemma.
\begin{lem}\label{operator aut}
	Let \(H\) be a Hilbert space and let \(A\in B(H)\) be a strict contraction. Then for \(X,Y\in B_1(H)\),
	\[
	I-\Theta_A(X)\Theta_A(Y)^*
	=
	\Delta_A (I-XA^*)^{-1}(I-XY^*)(I-AY^*)^{-1}\Delta_A.
	\]
\end{lem}
\begin{proof}
	By the functional calculus,
	\begin{equation}\label{ADA*=DAA}
		A\Delta_{A^*}=\Delta_A A,\text{ and }\Delta_{A^*}A^*=A^*\Delta_A.
	\end{equation}
	Then by \eqref{ADA*=DAA},
	\[
	\begin{aligned}
		&\quad I-\Theta_A(X)\Theta_A(Y)^*\\
		&= I-AA^* + A\Delta_{A^*}Y^*(I-AY^*)^{-1}\Delta_A \\
		&\quad + \Delta_A  (I -  XA^*)^{-1}X \Delta_{A^*}A^* \\
		&\quad - \Delta_A  (I -  XA^*)^{-1}X \Delta_{A^*}^2 Y^*(I-AY^*)^{-1}\Delta_A \\
		&= \Delta_A \Big[ I + A Y^*(I-AY^*)^{-1} + (I -  XA^*)^{-1}XA^*-(I -  XA^*)^{-1}XY^*(I-AY^*)^{-1} \\
		&\quad+ (I -  XA^*)^{-1}XA^*AY^*(I-AY^*)^{-1} \Big] \Delta_A\\
		&\overset{(*)}{=} \Delta_A \Big[ I + ((I-AY^*)^{-1}-I) + ((I-XA^*)^{-1}-I)-(I -  XA^*)^{-1}XY^*(I-AY^*)^{-1} \\
		&\quad+ ((I-XA^*)^{-1}-I)((I-AY^*)^{-1}-I) \Big] \Delta_A\\
		&=\Delta_A (I-XA^*)^{-1}(I-XY^*)(I-AY^*)^{-1}\Delta_A,
	\end{aligned}
	\]
	where $(*)$ follows from 
	$$
	A Y^*(I-AY^*)^{-1}=(I-AY^*)^{-1}-I, \text{ and } (I -  XA^*)^{-1}XA^*=(I-XA^*)^{-1}-I.
	$$
\end{proof}
	For \(a \in \mathbb{B}_d\setminus\{0\}\), let $\theta_a$ be the standard M\"obius automorphism of $\mathbb{B}_d$, given by
	$$
	\theta_a(z)=\frac{-\sqrt{1-|a|^{2}}|a|^{2} z+\left(\sqrt{1-|a|^{2}}-1\right)\langle z, a \rangle a+|a|^{2} a}{|a|^{2}(1-\langle z, a \rangle)},
	$$
and let $\theta_0(z)=-z$. Then we have the following lemma.

\begin{lem}\label{lem:delta-DXD}
	For \(u,v\in\mathbb B_d\), set $X=\partial^{(u)}$ and $Y=\partial^{(v)}$, then
	$$
	\Delta_{X}Y\Delta_{X^*}
	=\sqrt{1-\|u\|^2}Y
	-\frac{\sqrt{1-\|u\|^2}\langle v,u\rangle}{1+\sqrt{1-\|u\|^2}}X.
	$$
\end{lem}

\begin{proof}
	By Lemma~\ref{lem:partial-relations}(2), 
	$$
	(X^*X)^2=X^*XX^*X=X^*(\|u\|^2I-X^*X)X=\|u\|^2X^*X,
	$$
	so \(\sigma(X^*X)\subseteq\{0,\|u\|^2\}\). By Lemma~\ref{lem:sqrt-one-minus-Q}, 
	\begin{equation}\label{deltaa}
		\Delta_{X^*}=(I-X^*X)^{1/2}=I-\frac{X^*X}{1+\sqrt{1-\|u\|^2}}.
	\end{equation}
	By the same reasoning as above, we obtain
	\begin{equation}\label{deltaastar}
		\Delta_{X}=(I-XX^*)^{1/2}=I-\frac{XX^*}{1+\sqrt{1-\|u\|^2}}.
	\end{equation}
	Then by \eqref{deltaa} and \eqref{deltaastar},
	$$
	\begin{aligned}
		\Delta_{X}Y\Delta_{X^*}
		&=Y-\frac{XX^*Y+YX^*X}{1+\sqrt{1-\|u\|^2}}+\frac{XX^*YX^*X}{(1+\sqrt{1-\|u\|^2})^2}\\
		&\overset{(*)}{=}Y-\frac{\langle v,u\rangle X+\|u\|^2Y}{1+\sqrt{1-\|u\|^2}}
		+\frac{\langle v,u\rangle\|u\|^2X}{(1+\sqrt{1-\|u\|^2})^2}\\
		&=\sqrt{1-\|u\|^2}Y-\frac{\sqrt{1-\|u\|^2}\langle v,u\rangle}{1+\sqrt{1-\|u\|^2}}X,
	\end{aligned}
	$$
	where $(*)$ follows from Lemma \ref{lem:partial-relations}(4)(5).
\end{proof}

\begin{lem}
	\label{lem:compatibility}
	Let \(u,v\in \mathbb{B}_d\). Then
	$
	\Theta_{\partial^{(u)}}\bigl(\partial^{(v)}\bigr) = \partial^{(\theta_u(v))}.
	$
\end{lem}

\begin{proof}
	By Lemma~\ref{lem:partial-relations}(3),
	\begin{equation}\label{dvinvet}
		\partial^{(v)}(I-(\partial^{(u)})^*\partial^{(v)})^{-1}
		=\frac{\partial^{(v)}}{1-\langle v,u\rangle}.
	\end{equation}
	Hence by \eqref{dvinvet},
	\[
	\begin{aligned}
		\Theta_{\partial^{(u)}}\bigl(\partial^{(v)}\bigr)
		&=\partial^{(u)}
		-\Delta_{\partial^{(u)}}\partial^{(v)}
		(I-(\partial^{(u)})^*\partial^{(v)})^{-1}
		\Delta_{(\partial^{(u)})^*}\\
		&=\partial^{(u)}
		-\frac{1}{1-\langle v,u\rangle}
		\Delta_{\partial^{(u)}}\partial^{(v)}
		\Delta_{(\partial^{(u)})^*}\\
		&\overset{(*)}{=}\frac{\Bigl(1-\frac{\langle v,u\rangle}{1+\sqrt{1-\|u\|^2}}\Bigr)\partial^{(u)}
			-\sqrt{1-\|u\|^2}\partial^{(v)}}{1-\langle v,u\rangle}\\
		&=\sum_{i=1}^d
		\frac{\Bigl(1-\frac{\langle v,u\rangle}{1+\sqrt{1-\|u\|^2}}\Bigr)u_i
			-\sqrt{1-\|u\|^2}v_i}{1-\langle v,u\rangle}
		E_i^*\\
		&=\partial^{(\theta_u(v))},
	\end{aligned}
	\]
	where $(*)$ follows from Lemma~\ref{lem:delta-DXD}.
\end{proof}
Let \(\Phi=(\varphi_1,\cdots,\varphi_N)^{\mathrm T}\in\mathrm{Mult}(\mathcal{H},\mathcal{H}\otimes\mathbb C^N)\). Then \(\sum\limits_{i=1}^N M_{\varphi_i}\otimes E_i^*\) can be regarded as a multiplication operator with the associate multiplier 
$$
\left(\sum\limits_{i=1}^N M_{\varphi_i}\otimes E_i^*\right)(\cdot)\in \mathrm{Mult}(\mathcal{H}\otimes \Lambda_N,\mathcal{H}\otimes\Lambda_N),
$$
and
\begin{equation}\label{pointvalue}
	\left(\sum_{i=1}^N M_{\varphi_i}\otimes E_i^*\right)(z)=\partial^{(\Phi(z))}.
\end{equation}
We have the following lemma.
\begin{lem}
	\label{lem:range-rotation}
	Let \(F\subset\mathbb B_d\) be finite, \(\Phi=(\varphi_1,\cdots,\varphi_N)^{\mathrm T}\in\mathrm{Mult}(H_d^2|_F,H_d^2|_F\otimes\mathbb C^N)\), and \(a\in\mathbb B_N\). Suppose that
	$
	\left\|\sum\limits_{i=1}^N M_{\varphi_i}\otimes E_i^*\right\|\leq 1.
	$
	Then
	$
	\left\|\sum\limits_{i=1}^N M_{(\theta_a\circ\Phi)_i}\otimes E_i^*\right\|\le1.
	$
\end{lem}

\begin{proof}
By replacing $\Phi$ with $\rho\Phi$ for $0<\rho<1$ and letting $\rho\to 1$, we may assume without loss of generality that $\left\|\sum\limits_{i=1}^N M_{\varphi_i}\otimes E_i^*\right\|<1$.
	By Lemma \ref{mult crit} and \eqref{pointvalue}, it suffices to prove that 
	\begin{equation}\label{Ktheta>0}
		\left[K(z^{(i)},z^{(j)})\left(I-\partial^{(\theta_a\circ\Phi(z^{(i)}))}(\partial^{(\theta_a\circ\Phi(z^{(j)}))})^*\right)\right]_{i,j=1}^n\geq 0
	\end{equation}
	for every finite subset \(\{z^{(1)},\cdots,z^{(n)}\}\subset F\). By Lemma \ref{operator aut} and Lemma \ref{lem:compatibility}, for $z,w\in F$,
	\begin{equation}\label{1-theta=1-psi}
		\begin{aligned}
			I-\partial^{(\theta_a\circ\Phi(z))}(\partial^{(\theta_a\circ\Phi(w))})^*
			=&I-\Theta_{\partial^{(a)}}\bigl(\partial^{(\Phi(z))}\bigr)\Theta_{\partial^{(a)}}\bigl(\partial^{(\Phi(w))}\bigr)^*\\
			=&A_z\left(I-\partial^{(\Phi(z))}(\partial^{(\Phi(w))})^*\right)A_w^*,
		\end{aligned}
	\end{equation}
	where 
	$
	A_z=\Delta_{\partial^{(a)}}\left(I-\partial^{(\Phi(z))}(\partial^{(a)})^*\right)^{-1}.
	$
	Hence by \eqref{1-theta=1-psi}, for every finite subset \(\{z^{(1)},\cdots,z^{(n)}\}\subset F\),
	\begin{equation}\label{thetaaphi=Azphi}
		\begin{aligned}
			&\left[K(z^{(i)},z^{(j)})\left(I-\partial^{(\theta_a\circ\Phi(z^{(i)}))}(\partial^{(\theta_a\circ\Phi(z^{(j)}))})^*\right)\right]_{i,j=1}^n\\
			=&\operatorname{diag}(A_{z^{(1)}},\cdots,A_{z^{(n)}})\left[K(z^{(i)},z^{(j)})\left(I-\partial^{(\Phi(z^{(i)}))}(\partial^{(\Phi(z^{(j)}))})^*\right)\right]_{i,j=1}^n\operatorname{diag}(A_{z^{(1)}}^*,\cdots,A_{z^{(n)}}^*)\\
		\end{aligned}
	\end{equation}
	Since
	$
	\left\|\sum\limits_{i=1}^N M_{\varphi_i}\otimes E_i^*\right\|\le1,
	$
	by Lemma \ref{mult crit} and \eqref{pointvalue}, for every finite subset \(\{z^{(1)},\cdots,z^{(n)}\}\subset F\)
	$$
	\left[K(z^{(i)},z^{(j)})\left(I-\partial^{(\Phi(z^{(i)}))}(\partial^{(\Phi(z^{(j)}))})^*\right)\right]_{i,j=1}^n\geq0.
	$$
	Hence by \eqref{thetaaphi=Azphi}, \eqref{Ktheta>0} holds.
\end{proof}
Write $\mathrm{Mult}_1(\mathcal{H}\otimes \mathcal{E},\mathcal{K}\otimes \mathcal{F})$ for the closed unit ball of $\mathrm{Mult}(\mathcal{H}\otimes \mathcal{E},\mathcal{K}\otimes \mathcal{F})$. Then we have the following lemma.
\begin{lem}\label{Tpsicontrac}
	Let $\mathcal{H}$ be an irreducible complete Pick space on $\bf X$. Assume that for any 
	$$
	\psi=(\psi_1,\cdots,\psi_N)^\mathrm{T}\in\mathrm{Mult}_1(\mathcal{H},\mathcal{H}\otimes\mathbb C^N),
	$$
	we have
	$
	\left\|\sum\limits_{i=1}^{N}M_{\psi_i}\otimes E_i^*\right\|\leq 1.
	$
	Then for any $d\in \mathbb{N}$ and $\{\psi^{(k)}\}_{k=1}^d\subset \mathrm{Mult}_1(\mathcal{H},\mathcal{H}\otimes\mathbb C^N)$,
	$$
	\left\|\left[\sum_{i=1}^{N}M_{\psi_i^{(1)}}\otimes E_i^*,\cdots,\sum_{i=1}^{N}M_{\psi_i^{(d)}}\otimes E_i^*\right]^\mathrm{T}\right\|\leq \left\|M_\Psi\right\|,
	$$
where $\Psi=(\psi_1^{(1)},\cdots,\psi_N^{(1)},\cdots, \psi_1^{(d)},\cdots,\psi_N^{(d)})^\mathrm{T}\in \mathrm{Mult}(\mathcal{H},\mathcal{H}\otimes\mathbb C^{dN})$.
\end{lem}
\begin{proof}
	Fix $d\in \mathbb{N}$, and write 
	$$
	T_\Psi=\left[\sum\limits_{i=1}^{N}M_{\psi_i^{(1)}}\otimes E_i^*,\cdots,\sum\limits_{i=1}^{N}M_{\psi_i^{(d)}}\otimes E_i^*\right]^\mathrm{T}.
	$$
	For 
	$
	\Upsilon\in \mathrm{Mult}(\mathcal{H},\mathcal{H}\otimes\Lambda_N),
	$
	define 
	$
	B_\Upsilon(z):\mathbb C^N\to\Lambda_N
	$
	by
	$
	B_\Upsilon(z)a=\partial^{(a)}(\Upsilon(z)).
	$
	Then for any \(\alpha=(\alpha_1,\cdots,\alpha_N)^\mathrm{T}\in\mathrm{Mult}(\mathcal{H},\mathcal{H}\otimes\mathbb C^N)\), by \eqref{pointvalue} we have
	\begin{equation}\label{Balpha=partialU}
		(B_\Upsilon(z))(\alpha(z))=\partial^{(\alpha(z))}(\Upsilon(z))=\left(\sum\limits_{i=1}^{N}M_{\alpha_i}\otimes E_i^*\right)(z)\Upsilon(z).
	\end{equation}
	Hence by \eqref{mult op} and \eqref{Balpha=partialU}, $M_{B_\Upsilon}M_{\alpha}=\left(\sum\limits_{i=1}^{N}M_{\alpha_i}\otimes E_i^*\right)M_{\Upsilon}$. Then it is easy to see that 
	$$
	T_\Psi M_{\Upsilon}=(I_{\mathbb{C}^d}\otimes M_{B_\Upsilon})M_{\Psi}.
	$$
	Therefore by \cite[Lemma 3.9]{Hartz2023} and the assumption,
	$$
	\begin{aligned}
		\|T_\Psi\|
		&=\sup\limits_{\|M_\Upsilon\|\leq 1}\|T_\Psi M_\Upsilon\|\\
		&=\sup\limits_{\|M_\Upsilon\|\leq 1}\|(I_{\mathbb{C}^d}\otimes M_{B_\Upsilon})M_\Psi\|\\
		&\leq\|M_\Psi\|\sup\limits_{\|M_\Upsilon\|\leq 1}\|M_{B_\Upsilon}\|\\
		&=\|M_\Psi\|\sup\limits_{\|M_\Upsilon\|\leq 1}\left(\sup\limits_{\|M_\alpha\|\leq 1}\|M_{B_\Upsilon}M_\alpha\|\right)\\
		&=\|M_\Psi\|\sup\limits_{\|M_\Upsilon\|\leq 1}\left(\sup\limits_{\|M_\alpha\|\leq 1}\left\|\left(\sum\limits_{i=1}^{N}M_{\alpha_i}\otimes E_i^*\right)M_\Upsilon\right\|\right)\\
		&\leq \|M_\Psi\|.
	\end{aligned}
	$$
\end{proof}

\begin{lem}\label{lem:zPhi-contractive}
		Let \(E\subset\mathbb B_d\) be finite, and set \(F=E\cup\{0\}\).
		Let \(\mathcal{E}\) be a Hilbert space, and
		\[
		\Phi=(\Phi_1,\cdots,\Phi_d)^\mathrm{T}\in\mathrm{Mult}\bigl(H_d^2|_F\otimes\mathcal{E},\,
		H_d^2|_F\otimes(\mathcal{E}\otimes\mathbb C^d)\bigr),
		\]
	where $\Phi_i\in\mathrm{Mult}(H_d^2|_F\otimes\mathcal{E})$ for $1\leq i\leq d$.
	Suppose that
	$
	\|M_{\Phi|_E}\|\le1.
	$
	Then 
	$$
	\left\|\left[M_{z_1}\otimes I_{\mathcal{E}},\cdots,M_{z_d}\otimes I_{\mathcal{E}}\right]M_\Phi\right\|_{H_d^2|_F\otimes\mathcal{E}}\leq 1.
	$$
\end{lem}
\begin{proof}
	For $z\in \mathbb{B}_d$, set $\mathbf{z}(z)=[z_1I_\mathcal{E},\cdots,z_dI_\mathcal{E}]$. Then $z\Phi\in \mathrm{Mult}(H_d^2|_F\otimes\mathcal{E})$, and it is easy to see that
	$$
	M_{\mathbf{z}\Phi}=\left[M_{z_1}\otimes I_{\mathcal{E}},\cdots,M_{z_d}\otimes I_{\mathcal{E}}\right]M_\Phi.
	$$
	Hence by Lemma \ref{3.12a}, it suffices to prove that for any finite subset \(\{z^{(1)},\cdots,z^{(n)}\}\subset E\),
	\begin{equation}\label{eq:zPhi}
		\bigl[K(z^{(i)},z^{(j)})\bigl(I_{\mathcal{E}}-(\mathbf{z}\Phi)(z^{(i)})(\mathbf{z}\Phi)(z^{(j)})^*\bigr)-I_{\mathcal{E}}\bigr]_{i,j=1}^n\geq 0
	\end{equation}
	For $z,w\in E$, by \(K(z,w)\langle z,w\rangle=K(z,w)-1\), we have
	\begin{equation}\label{Ks-I=zK-1zs}
		\begin{aligned}
			&K(z,w)\bigl(I_{\mathcal{E}}-(\mathbf{z}\Phi)(z)(\mathbf{z}\Phi)(w)^*\bigr)-I_{\mathcal{E}}\\
			=&K(z,w)\bigl(\langle z,w\rangle I_{\mathcal{E}}-(\mathbf{z}\Phi)(z)(\mathbf{z}\Phi)(w)^*\bigr)\\
			=&K(z,w)\bigl(\mathbf{z}(z)\mathbf{z}(w)^*-\mathbf{z}(z)\Phi(z)\Phi(w)^*\mathbf{z}(w)^*\bigr)\\
			=&\mathbf{z}(z)K(z,w)\bigl(I_{\mathcal{E}}\otimes I_{\mathbb C^d}-\Phi(z)\Phi(w)^*\bigr)\mathbf{z}(w)^*
		\end{aligned}
	\end{equation}
	Hence by \eqref{Ks-I=zK-1zs}, for any finite subset \(\{z^{(1)},\cdots,z^{(n)}\}\subset E\),
	\begin{equation}\label{IE-zphi}
		\begin{aligned}
			&\bigl[K(z^{(i)},z^{(j)})\bigl(I_{\mathcal E}-(\mathbf{z}\Phi)(z^{(i)})(\mathbf{z}\Phi)(z^{(j)})^*\bigr)-I_{\mathcal E}\bigr]_{i,j=1}^n\\
			=&\operatorname{diag}(\mathbf{z}(z^{(1)}),\cdots,\mathbf{z}(z^{(n)}))
			\bigl[K(z^{(i)},z^{(j)})\bigl(I_{\mathcal E}\otimes I_{\mathbb C^d}\\
			&-\Phi(z^{(i)})\Phi(z^{(j)})^*\bigr)\bigr]_{i,j=1}^n
			\operatorname{diag}(\mathbf{z}(z^{(1)})^*,\cdots,\mathbf{z}(z^{(n)})^*)\\
		\end{aligned}
	\end{equation}
	By Lemma \ref{mult crit} and $\|\Phi|_{E}\|\leq 1$, for every finite subset \(\{z^{(1)},\cdots,z^{(n)}\}\subset E\),
	\begin{equation}\label{IF-phi}
		\bigl[K(z^{(i)},z^{(j)})\bigl(I_{\mathcal{E}}\otimes I_{\mathbb C^d}-\Phi(z^{(i)})\Phi(z^{(j)})^*\bigr)\bigr]_{i,j=1}^n\ge0.
	\end{equation}
	Therefore by \eqref{IE-zphi} and \eqref{IF-phi}, \eqref{eq:zPhi} holds.
\end{proof}
\begin{prop}
	\label{prop:reduction}
	Fix \(d,N\in\mathbb N\). Then for any finite set \(E\subset\mathbb B_d\), the following statement holds: if
	$$
	\Phi=(\varphi_1,\cdots,\varphi_N)^\mathrm{T}\in\mathrm{Mult}_1(H_d^2|_E,H_d^2|_E\otimes\mathbb C^N),
	$$
	then 
	$$
	\left\|\sum\limits_{i=1}^{N}M_{\varphi_i}\otimes E_i^*\right\|_{H_d^2|_E\otimes\Lambda_N\to H_d^2|_E\otimes\Lambda_N}\le1.
	$$
\end{prop}
\begin{proof}
	By replacing $\Phi$ with $\rho\Phi$ for $0<\rho<1$ and letting $\rho\to 1$, we may assume without loss of generality that $\|\Phi\|_{\mathrm{Mult}(H_d^2|_E,H_d^2|_E\otimes\mathbb C^N)}<1$.
	We prove the statement by induction on \(|E|\). For \(|E|=1\), write \(E=\{w\}\), then \(\dim H_d^2|_E=1\). By Lemma \ref{mult crit} and the assumption, $K(w,w)(I-\Phi(w)\Phi(w)^*)\geq 0$, i.e. $\|\Phi(w)\|\leq 1$. Hence by \eqref{pointvalue} and Lemma \ref{lem:partial-relations}(6), 
	$$
	\left\|\left(\sum\limits_{i=1}^NM_{\varphi_i}\otimes E_i^*\right)(w)\right\|=\|\partial^{(\Phi(w))}\|=\|\Phi(w)\|\leq 1.
	$$
	Hence 
	$$
	K(w,w)\left(I-\left(\sum\limits_{i=1}^NM_{\varphi_i}\otimes E_i^*\right)(w)\left(\sum\limits_{i=1}^NM_{\varphi_i}\otimes E_i^*\right)(w)^*\right)\geq 0.
	$$
	Then by Lemma \ref{mult crit}, 
	$$
	\left\|\sum\limits_{i=1}^{N}M_{\varphi_i}\otimes E_i^*\right\|_{H_d^2|_E\otimes\Lambda_N\to H_d^2|_E\otimes\Lambda_N}\le1.
	$$

	Assume the statement holds for every finite set \(E\subset\mathbb B_d\) with \(|E|\le n-1\). Let \(F\subset\mathbb B_d\) with \(|F|=n\) and let \(\Phi\in\mathrm{Mult}_1(H_d^2|_F,H_d^2|_F\otimes\mathbb C^N)\). 
	
	By \cite[Lemma 3.4]{Hartz2023}, we may apply a domain automorphism of \(\mathbb B_d\) and assume \(0\in F\). Let \(a=\Phi(0)\) and set \(\Phi'=\theta_a\circ\Phi\). By \cite[Lemma 3.5]{Hartz2023}, \(\Phi'\) is a contractive column multiplier with \(\Phi'(0)=0\). Since \(\Phi=\theta_a\circ\Phi'\), by Lemma \ref{lem:range-rotation}, it suffices to prove
	$
	\Bigl\|\sum\limits_{i=1}^N M_{\Phi'_i}\otimes E_i^*\Bigr\|\le1.
	$
	Thus we may assume \(\Phi(0)=0\).
	
	By \cite[Proposition 3.8]{Hartz2023}, there exist functions
	$
	\psi^{(i)}=(\psi_{1i},\cdots,\psi_{di})^{\mathrm T}
	\in\mathrm{Mult}(H_d^2|_F,H_d^2|_F\otimes\mathbb C^d)
	$
	for $1\leq i\leq N$
	such that
	\begin{equation}\label{phi=zpsi}
		\varphi_i=\sum_{j=1}^d z_j\psi_{ji},
	\end{equation}
	and
	\[
	\Psi=(\psi^{(1)\mathrm T},\cdots,\psi^{(N)\mathrm T})^{\mathrm T}
	\in\mathrm{Mult}_1(H_d^2|_F,H_d^2|_F\otimes\mathbb C^{dN}).
	\]
	Write 
	\[
	T_\Psi=\left(\sum_{i=1}^{N}M_{\psi_{1i}}\otimes E_i^*,\cdots,\sum_{i=1}^{N}M_{\psi_{di}}\otimes E_i^*\right)^{\mathrm T}.
	\]
	Then by \eqref{phi=zpsi},
	\begin{equation}\label{Mphi=zTpsi}
		\sum_{i=1}^{N}M_{\varphi_i}\otimes E_i^*=[M_{z_1}\otimes I_{\Lambda_N},\cdots, M_{z_d}\otimes I_{\Lambda_N}] T_\Psi.
	\end{equation}
	Set $F'=F\setminus\{0\}$. 
	By Lemma \ref{Tpsicontrac} and the induction hypothesis, we have
	$$
	\|T_{\Psi|_{F'}}\|\leq \|M_{\Psi|_{F'}}\| \le \|M_{\Psi}\|\leq 1.
	$$
	Hence by \eqref{Mphi=zTpsi} and Lemma \ref{lem:zPhi-contractive}, 
	$$
	\left\|\sum\limits_{i=1}^{N}M_{\varphi_i}\otimes E_i^*\right\|_{H_d^2|_F\otimes\Lambda_N\to H_d^2|_F\otimes\Lambda_N}\le1.
	$$
\end{proof}

\begin{proof}[Proof of Theorem \ref{thm:main}]
	\((2)\Rightarrow(1)\): Since
	$$
	\sum_{i=1}^{N}\|\varphi_if\|^2=\|\partial_{\mathbb{T}}(f\otimes e_1\wedge\cdots \wedge e_N)\|^2\leq \|\partial_{\mathbb{T}}\|^2\|f\|^2,
	$$
	(1) holds and $\|[M_{\varphi_1},M_{\varphi_2},\cdots]^T\| \leq  \|\partial_{\mathbb{T}}\|$. \\
\vskip 1mm	
 \((1)\Rightarrow(2)\):
	Similar to \cite[Lemma 3.2]{Hartz2023}(replace \(\Phi^\mathrm{T}\) by \(\partial_{\mathbb{T}}\) in the proof of \cite[Lemma 3.2]{Hartz2023}), by Proposition \ref{prop:reduction}, $\|\partial_{\mathbb{T}}\|\leq 1$ whenever $\|[M_{\varphi_1},M_{\varphi_2},\cdots]^T\|\leq 1$. Thus (2) holds and 
	$$
	\|\partial_{\mathbb{T}}\|\leq \|[M_{\varphi_1},M_{\varphi_2},\cdots]^T\|.
	$$
\end{proof}

\section{Applications to the Drury--Arveson space and its quotient modules}
In this section, we prove Theorem \ref{7.1}. In order to prove this theorem, we need some preparation.

For $\lambda\in\mathbb{B}\setminus\{0\}$, let $\varphi_{\lambda}(z)$ be the automorphism of $\mathbb{B}$ taking $\lambda$ to 0 of the form
\begin{align}\label{def:varphi_lambda}
\varphi_{\lambda}(z)=\frac{\lambda-P_{\lambda} z-s_{\lambda} Q_{\lambda} z}{1-\left\langle z, \lambda\right\rangle}=\frac{s_{\lambda} A_{\lambda}(z-\lambda)}{1-\left\langle z, \lambda\right\rangle},
\end{align}
where 
$$
s_{\lambda}=\left(1-|\lambda|^{2}\right)^{1 / 2},~A_{\lambda}=-\left(\frac{1}{s_{\lambda}} P_{\lambda}+Q_{\lambda}\right),~P_{\lambda} z=\frac{\left\langle z, \lambda\right\rangle}{|\lambda|^{2}} \lambda,~Q_{\lambda}=I-P_{\lambda}.
$$
For \(\lambda=0\), we set \(\varphi_{0}(z)=-z\).
Let $\mathbb{T}=\left(T_{1}, T_{2}, \cdots\right)$ be a row-contraction on $H$. For $\lambda \in \mathbb{B} $, we have $I-\left\langle \mathbb{T}, \lambda\right\rangle=1-\sum\limits_{i=1}^{\infty} \bar{\lambda}_{i} T_{i} \quad$ is invertible. Write $A_{\lambda}=\left(a_{i j}(\lambda)\right)_{i, j \geq 1}$ with respect to the standard orthonormal basis in $l^{2}$, then it is easy to see that
$$
a_{ij}(\lambda) = -\delta_{ij} + \left(1 - \frac{1}{s_\lambda}\right) \frac{\overline{\lambda_j} \lambda_i}{|\lambda|^2},
\quad s_\lambda = \sqrt{1-|\lambda|^2}.
$$ 
Then for $i\geq 1$, $\sum\limits_{j=1}^{\infty}\left| a_{ij}(\lambda)\right|^2<\infty$. Hence for $x\in H$, $\sum\limits_{j=1}^{\infty}a_{ij}(\lambda)x\otimes e_j\in \Lambda^1(H)$. Notice that for $n,p>0$ and $x\in H$, 
\begin{equation}\label{sum aij=sum Ti}
\begin{aligned}
\left\|\sum_{j=n}^{n+p}a_{ij}(\lambda)(T_j-\lambda_j)x\right\|&=\left\|\left(\sum_{j=n}^{n+p}(T_j-\lambda_j)\otimes E_j^*\right)\sum_{j=n}^{n+p}a_{ij}(\lambda)x\otimes e_j\right\|\\
&=
\left\|\left(\sum_{j=n}^{n+p}(T_j-\lambda_j)\otimes E_j^*\right)\sum_{j=1}^{\infty}a_{ij}(\lambda)x\otimes e_j\right\|.
\end{aligned}
\end{equation}
By Theorem \ref{prop 1.1}, $\left\{\sum\limits_{j=1}^{m}(T_j-\lambda_j)\otimes E_j^*\right\}_{m\geq 1}$ converges in SOT. Then by \eqref{sum aij=sum Ti}, $\left\{\sum\limits_{j=1}^{m} a_{i j}(\lambda)\left(T_{j}-\lambda_{j}\right)\right\}_{m\geq 1}$ converges in SOT. We can define $\varphi_{\lambda}(\mathbb{T})$ as
\begin{equation}\label{varphi_lambda}
\left(\varphi_{\lambda}(\mathbb{T})\right)_{i}=(I-\left\langle \mathbb{T}, \lambda\right\rangle)^{-1} s_{\lambda} \sum_{j=1}^{\infty} a_{i j}(\lambda)\left(T_{j}-\lambda_{j}\right), ~\forall i \geq 1.
\end{equation}
Let \(\Lambda_{\mathrm{alg}}=\bigoplus\limits_{p\ge0}\Lambda^p\) be the
algebraic direct sum. For \(T\in B(\ell^2)\), define
\(\Gamma(T)\) on \(\Lambda_{\mathrm{alg}}\) by
\[
\Gamma(T)|_{\Lambda^0}=I,
\qquad
\Gamma(T)(v_1\wedge\cdots\wedge v_p)
=
Tv_1\wedge\cdots\wedge Tv_p
\quad(p\ge1,\;v_i\in\ell^2),
\]
and linear extension. If \(\Gamma(T)\) is bounded on
\(\Lambda_{\mathrm{alg}}\), it extends uniquely to a bounded operator on \(\Lambda\), still denoted
by \(\Gamma(T)\).

On \(\Lambda_{\mathrm{alg}}\) one has
\[
\Gamma(AB)=\Gamma(A)\Gamma(B),
\qquad
\Gamma(I)=I,
\]
and these identities extend by continuity whenever the corresponding
operators are bounded. Moreover, if \(U\) is unitary, then \(\Gamma(U)\)
is unitary.

\begin{lem}\label{T fredholm iff phi T}
  $\mathbb{T}-\lambda$ is Fredholm if and only if $\varphi_{\lambda}(\mathbb{T})$ is Fredholm, and in this case
  $\operatorname{ind}(\mathbb{T}-\lambda)=\operatorname{ind}\varphi_{\lambda}(\mathbb{T})$.
\end{lem}

\begin{proof}
The case \(\lambda=0\) is immediate from the fact that \(\varphi_0(\mathbb T)=-\mathbb T\), so we may assume \(\lambda\neq0\). For $i\ge 1$ set
\[
\mathcal{S}_{i}=\sum_{j=1}^{\infty} a_{i j}(\lambda)(T_{j}-\lambda_{j}).
\]
By \eqref{varphi_lambda}, clearly $\varphi_{\lambda}(\mathbb{T})$ is Fredholm iff $\mathcal{S}=(\mathcal{S}_1,\mathcal{S}_2,\cdots)$ is Fredholm, with the same index.
Notice that
$$
\partial_{\mathcal{S}} = \sum_{i=1}^{\infty} \mathcal{S}_i \otimes E_i^*
= \sum_{j=1}^{\infty} (T_j-\lambda_j) \otimes D_j^*,
$$
where $D_j^* = \sum\limits_{i=1}^{\infty} a_{i j}(\lambda) E_i^*$. We claim that $\Gamma(A_{\lambda}^T)D_{i}^{*}=E_{i}^{*} \Gamma(A_{\lambda}^T), \forall i \geq 1$. Indeed, since $A_{\lambda}=-\left(\frac{1}{s_\lambda}P_\lambda+Q_\lambda\right)$, $A_{\lambda}^T$ is self-adjoint and can be unitarily diagonalized by
$$
A_{\lambda}^T=V B V^{-1},
$$
where $B=\operatorname{diag}(-1/s_\lambda,-1,-1,\cdots)$. Hence $\Gamma(A_\lambda^T)=\Gamma(V)\Gamma(B)\Gamma(V^{-1})$.  For $p\ge 1$,
\[
\|\Gamma(B)|_{\Lambda^p}\|
= \sup_{i_1<\cdots<i_p} \prod_{k=1}^{p} |b_{i_k}|
= \max\{1/s_\lambda,\,1\} = \frac1{s_\lambda},
\]
and $\|\Gamma(B^{-1})|_{\Lambda^p}\| = \max\{s_\lambda,1\} = 1 $. Thus $\|\Gamma(A_\lambda^T)\| = 1/s_\lambda$ and $\|\Gamma(A_\lambda^T)^{-1}\|=1$. Hence $\Gamma(A_\lambda^T)$ is bounded and invertible.
For \(v\in\ell^2\), define \(E(v)^*:\Lambda\to\Lambda\) by
$$
E(v)^*(x_1\wedge\cdots\wedge x_p)= \sum_{k=1}^{p}(-1)^{k-1}\langle x_k,v\rangle\,
x_1\wedge\cdots\wedge\widehat{x_k}\wedge\cdots\wedge x_p,
$$
where the hat indicates that \(x_k\) is omitted. Then \(E_i^* = E(e_i)^*\). 
Notice that \(A_\lambda=A_\lambda^*\), write \(A_\lambda=(a_{ij})\). Then
$$
A_\lambda^T e_j = \sum_{i=1}^{\infty} a_{ji} e_i = \sum_{i=1}^{\infty} \overline{a_{ij}} e_i.
$$
Since \(E(v)^*\) is conjugate-linear in \(v\), we have
$$
E(A_\lambda^T e_j)^* 
= E\left(\sum_{i=1}^{\infty} \overline{a_{ij}} e_i\right)^*
= \sum_{i=1}^{\infty} a_{ij} E(e_i)^*
= \sum_{i=1}^{\infty} a_{ij} E_i^*
= D_j^*.
$$
Moreover, \(A_\lambda^T\) is self-adjoint. For any \(\xi = x_1\wedge\cdots\wedge x_p\), we have
$$
\begin{aligned}
\Gamma(A_\lambda^T)E(v)^*\Gamma(A_\lambda^T)^{-1}\xi 
&= \Gamma(A_\lambda^T)E(v)^*((A_\lambda^T)^{-1}x_1\wedge\cdots\wedge (A_\lambda^T)^{-1}x_p)\\
&= \Gamma(A_\lambda^T) \left( \sum_{k=1}^p (-1)^{k-1} \langle (A_\lambda^T)^{-1}x_k, v \rangle (A_\lambda^T)^{-1}x_1\wedge\cdots\wedge \widehat{(A_\lambda^T)^{-1}x_k} \wedge\cdots\wedge (A_\lambda^T)^{-1}x_p \right)\\
&= \sum_{k=1}^p (-1)^{k-1} \langle (A_\lambda^T)^{-1}x_k, v \rangle x_1\wedge\cdots\wedge \widehat{x_k} \wedge\cdots\wedge x_p\\
&\stackrel{(*)}{=} \sum_{k=1}^p (-1)^{k-1} \langle x_k, (A_\lambda^T)^{-1}v \rangle x_1\wedge\cdots\wedge \widehat{x_k} \wedge\cdots\wedge x_p\\
&= E((A_\lambda^T)^{-1}v)^*\xi,
\end{aligned}
$$
where \((*)\) follows from
$$
\langle (A_\lambda^T)^{-1}x_k, v \rangle
= \langle x_k, ((A_\lambda^T)^{-1})^*v \rangle
= \langle x_k, (A_\lambda^T)^{-1}v \rangle,
$$
since \(A_\lambda^T\) is self-adjoint. By linearity and density, this equality extends to all of \(\Lambda\), yielding
$$
\Gamma(A_\lambda^T)E(v)^*\Gamma(A_\lambda^T)^{-1} = E((A_\lambda^T)^{-1}v)^*.
$$
Consequently,
$$
\Gamma(A_\lambda^T)D_j^*\Gamma(A_\lambda^T)^{-1}
=E((A_\lambda^T)^{-1}A_\lambda^T e_j)^*
=E_j^*.
$$
The claim is proved. Since \(\Gamma(A_\lambda^T)\) is bounded and invertible on \(\Lambda\), the operator \(I \otimes \Gamma(A_\lambda^T)\) is a bounded invertible operator on \(\Lambda(H)\). It is straightforward to check that \(I \otimes \Gamma(A_\lambda^T)\) maps \(\operatorname{Dom}(\partial_{\mathcal{S}})\) bijectively onto \(\operatorname{Dom}(\partial_{\mathbb{T}-\lambda})\). 
Consequently, the intertwining relation
\[
(I \otimes \Gamma(A_\lambda^T)) \partial_{\mathcal{S}} = \partial_{\mathbb{T}-\lambda} (I \otimes \Gamma(A_\lambda^T))
\]
holds on \(\operatorname{Dom}(\partial_{\mathcal{S}})\). This implies that \(\mathcal{S}\) is Fredholm if and only if \(\mathbb{T}-\lambda\) is Fredholm, and in this case,
\[
\operatorname{ind}(\mathbb{T}-\lambda)=\operatorname{ind}\mathcal{S}.
\]
\end{proof}

\begin{lem}\label{lem:composition}
For $\lambda\in\mathbb{B}$, the composition operator 
\[
C_{\varphi_\lambda}: H^2 \to H^2,\qquad (C_{\varphi_\lambda} f)(z)=f(\varphi_\lambda(z)),
\]
is bounded and invertible.
\end{lem}

\begin{proof}
Define \(U_\lambda: H^2\to H^2\) by
\[
(U_\lambda f)(z) = \frac{\sqrt{1-|\lambda|^2}}{1-\langle z,\lambda\rangle}\, f(\varphi_\lambda(z)).
\]
By \cite[Corollary 4.2]{hartz2025}, the operator \(U_\lambda\) is unitary.  Notice that
\[
C_{\varphi_\lambda} f = (U_\lambda f) \cdot \frac{1-\langle z,\lambda\rangle}{\sqrt{1-|\lambda|^2}}
= M_{\alpha} U_\lambda f,
\]
where \(\alpha(z)=\frac{1-\langle z,\lambda\rangle}{\sqrt{1-|\lambda|^2}}\) is an invertible multiplier. Since \(M_\alpha\) and \(U_\lambda\) are both bounded and invertible, their composition \(C_{\varphi_\lambda}\) is bounded and invertible.
\end{proof}

For \(t\geq 0\), let 
$$
H_t=\overline{\operatorname{span}\{z^\alpha:|\alpha|=t\}}
$$ 
be the space of homogeneous functions of degree \(t\) in \(H^2\), and \(P_t\) be the orthogonal projection from \(H^2\) onto \(H_t\).
\begin{prop}\label{spec H2}
  Let $\mathbb{T}=(M_{z_1},M_{z_2},\cdots)$ on $H^2$. Then
  $\sigma(\mathbb{T})=\sigma_{SZ}(\mathbb{T})=\overline{\mathbb B}$,
  $\sigma_e(\mathbb{T})=\partial\mathbb B$, and for every $\lambda\in\mathbb B$,
  $\operatorname{ind}(\mathbb{T}-\lambda)=1$.
\end{prop}

\begin{proof}
Notice that
  \[
  \operatorname{Im}\partial_{\mathbb{T},1} = \left\{ \sum_{i=1}^{\infty} z_i f_i : \sum_i\|f_i\|^2<\infty \right\}.
  \]
  Clearly $\operatorname{Im}\partial_{\mathbb{T},1} \subset \{f\in H^2 : f(0)=0\}$. Conversely, for any $f\in H^2$ with $f(0)=0$, since $I = P_{0} + \sum\limits_{i=1}^\infty M_{z_i}M_{z_i}^*$, we have
  \[
  f = \sum_{i=1}^{\infty} z_i (M_{z_i}^* f) \quad \text{and}
  \quad \sum_{i=1}^{\infty} \|M_{z_i}^* f\|^2 = \|f\|^2 .
  \]
  Hence $\operatorname{Im}\partial_{\mathbb{T},1} = \{f\in H^2 : f(0)=0\}$, which implies that 
  \begin{equation}\label{dim H0T}
  \dim H^0(\mathbb{T})=1.
  \end{equation}
  For any monomial $z^\alpha$ with $|\alpha|=t$, a direct computation gives
\[
(t+1)M_{z_i}^*M_{z_j}z^\alpha - t M_{z_j}M_{z_i}^*z^\alpha = \delta_{ij} z^\alpha .
\]
Since $ H_t=\overline{\operatorname{span}\{z^\alpha:~|\alpha|=t\}}$, we obtain
\begin{equation}\label{zizjcommute}
(t+1)M_{z_i}^*M_{z_j} - t M_{z_j}M_{z_i}^* = \delta_{ij} I \qquad\text{on } H_t. 
\end{equation}
Hence by \eqref{zizjcommute} and $E_i^*E_j+E_jE_i^*=\delta_{ij}I$, for $p\geq 1$ we have on $H_t\otimes\Lambda^p$,
\begin{equation}\label{t+1partial=pI-t}
\begin{aligned}
(t+1)\partial_{\mathbb{T},p}^*\partial_{\mathbb{T},p}
&= (t+1)\sum_{i,j} M_{z_i}^*M_{z_j}\otimes E_iE_j^* \\
&= \sum_{i,j}\left(t M_{z_j}M_{z_i}^* + \delta_{ij}I\right)\otimes E_iE_j^* \\
&= t\sum_{i,j} M_{z_i}M_{z_j}^*\otimes E_jE_i^* + \sum_i I\otimes E_iE_i^* \\
&= t\sum_{i,j} M_{z_i}M_{z_j}^*\otimes (\delta_{ij}I - E_i^*E_j) + pI \\
&= t\sum_i M_{z_i}M_{z_i}^*\otimes I - t\sum_{i,j} M_{z_i}M_{z_j}^*\otimes E_i^*E_j + pI \\
&= t\sum_i M_{z_i}M_{z_i}^*\otimes I - t\partial_{\mathbb{T},p+1}\partial_{\mathbb{T},p+1}^* + pI .
\end{aligned}
\end{equation}
Notice that
\[
\sum_{i=1}^{\infty} M_{z_i}M_{z_i}^*|_{H_t} =
\begin{cases}
I_{H_t}, & t\ge 1,\\[2pt]
0, & t=0 .
\end{cases}
\]
Then by \eqref{t+1partial=pI-t} for every $t\ge 0$ and $p\ge 1$,
\begin{equation}\label{tpar+t+1par=t+pI}
t\partial_{\mathbb{T},p+1}\partial_{\mathbb{T},p+1}^* + (t+1)\partial_{\mathbb{T},p}^*\partial_{\mathbb{T},p} = (t+p)I.
\end{equation}

Let \(x\in\ker\partial_{\mathbb{T},p}\) and write
\(x_t = (P_t\otimes I)x\). Since \(\partial_{\mathbb{T},p}\) maps \(H_t\otimes\Lambda^p\)
into \(H_{t+1}\otimes\Lambda^{p-1}\), we have \(\partial_{\mathbb{T},p} x_t = 0\) for every \(t\). Then by \eqref{tpar+t+1par=t+pI},
\begin{equation}\label{partialxt}
  t\partial_{\mathbb{T},p+1}\partial_{\mathbb{T},p+1}^*x_t=(t+p)x_t.
\end{equation}
Hence
\begin{equation}\label{tpartial star x t =t+p x t}
t\|\partial_{\mathbb{T},p+1}^* x_t\|^2=(t+p)\|x_t\|^2.
\end{equation}
In particular, $x_0=0$. Now set
$$
y = \sum_{t= 1}^\infty \frac{t}{t+p}\,\partial_{\mathbb{T},p+1}^* x_t \in \Lambda^{p+1}(H^2).
$$
Then by \eqref{partialxt} and \eqref{tpartial star x t =t+p x t}, for $p\geq 1$ and $x\in \ker\partial_{\mathbb{T},p}$, $y\in\Lambda^{p+1}(H^2)$ such that 
\begin{equation}\label{dp+1y=x}
\partial_{\mathbb{T},p+1} y = x,\quad \|y\|^2 = \sum\limits_{t= 1}^\infty \frac{t}{t+p}\|x_t\|^2 \le \|x\|^2
\end{equation}
Hence combine with \eqref{dim H0T},$\mathbb{T}$ is Fredholm with $\operatorname{ind}(\mathbb{T})=1$.

For $\lambda\in\mathbb B$, by Lemma \ref{lem:composition}, 
    \[
    (C_{\varphi_\lambda}^{-1} M_{z_i} C_{\varphi_\lambda} f)(z)
    = (\varphi_\lambda)_i(z) f(z)
    = (M_{(\varphi_\lambda)_i} f)(z),~\forall f\in H^2
    \]
    Hence
    $
    C_{\varphi_\lambda}^{-1}\, \mathbb T\, C_{\varphi_\lambda}
    = \varphi_\lambda(\mathbb T),
    $
    so $\mathbb T$ is similar to $\varphi_\lambda(\mathbb T)$, thus
    \[
    \operatorname{ind}(\varphi_\lambda(\mathbb T)) = \operatorname{ind}(\mathbb T) = 1.
    \]
    Then by Lemma \ref{T fredholm iff phi T}, $\mathbb T-\lambda$ is Fredholm with index $1$ for every $\lambda\in\mathbb B$. By Theorem \ref{bounded closed} we have $\sigma(\mathbb{T})\subset\overline{\mathbb{B}}$. Then by Corollary~\ref{index continious}, $\sigma_e(\mathbb{T})=\partial\mathbb{B}$ and $\sigma(\mathbb{T})=\overline{\mathbb B}$. Moreover, for any $\lambda\in \mathbb{B}$ and $n>0$, it is easy to see that $\sum\limits_{i=1}^{n}(M_{z_i}-\lambda_i)H^2\neq H^2$, hence $\mathbb{B}\subset\sigma_{SZ}(\mathbb{T})$. Then by Theorem \ref{bounded closed} we have $\sigma_{SZ}(\mathbb{T})=\overline{\mathbb{B}}$.
\end{proof}
Let $N\in\mathbb{N}$ and $\mathcal{M}$ be a submodule of $H^2\otimes \mathbb{C}^N$. By \cite[Theorem 0.7]{McCulloughTrent2000}, there exist a separable Hilbert space \(E\) and an inner multiplier \(\Phi \in B(H^2 \otimes E, H^2 \otimes \mathbb{C}^N)\) such that \(\operatorname{ran} \Phi = \mathcal{M}\) and \(M_\Phi M_\Phi^*\) is the orthogonal projection onto \(\mathcal{M}\). Let \(\phi: \mathbb{B} \to B(E, \mathbb{C}^N)\) be the associated operator valued analytic function. Then \(\mathcal{M}_\lambda = \operatorname{ran} \phi(\lambda)\) for every \(\lambda \in \mathbb{B}\), and there exists \(\lambda_0 \in \mathbb{B}\) such that \(\operatorname{rank} \phi(\lambda_0) = \operatorname{fd}(\mathcal{M})\). For simplicity, write \(m = \operatorname{fd}(\mathcal{M})\) and \(D = \mathcal{M}_{\lambda_0} = \operatorname{ran} \phi(\lambda_0)\).

Let $\{e_n\}_{1\leq n\leq m}$ be an orthonormal basis for $(\ker \phi(\lambda_0))^\perp$, and $\{d_k\}_{1\leq k\leq m}$ be an orthonormal basis for $D$.
We define the \(m\times m\) matrix
	\[
	\mathcal M(\lambda)
	=
	\bigl(
	\langle\phi(\lambda)e_j,d_k\rangle
	\bigr)_{1\leq k,j\leq m},
	\]	
and the analytic function $\varphi$ by
\begin{align}\label{defvarphi}
	\varphi(\lambda)=\frac{\det \mathcal{M}(\lambda)}{\det \mathcal{M}(\lambda_0)}.
\end{align}
Then $\varphi$ is a multiplier of $H^2$.
Similar to \cite[Lemma 3.3]{index of Invariant Subspace}. we have the following lemma.
\begin{lem}
	There exists a multiplication operator \(\Psi \in B(H^2\otimes D, H^2\otimes\mathbb C^N)\) with \(\operatorname{ran} \Psi \subset \mathcal M\) and such that
	\begin{align}\label{rightequ-Psi}
		(I_{H^2}\otimes P_D) \Psi = M_\varphi\otimes I_D,
	\end{align}
	and
	\begin{align}\label{leftequ-Psi}
		\Psi (I_{H^2}\otimes P_D) \big|_{\mathcal M}
		=
		(M_\varphi\otimes I_{\mathbb C^N})
		\big|_{\mathcal M},
	\end{align}
\end{lem}
Define \(\delta_{\mathcal M^\perp}:  \mathcal M^\perp \to D^\perp\) and \(\delta_{D^\perp}: H^2\otimes D^\perp \to D^\perp \) by \(\delta_{\mathcal M^\perp}q=P_{D^\perp}\bigl(q(\lambda_0)\bigr)\)
and \(\delta_{D^\perp}h=h(\lambda_0)\).

\begin{lem}\label{AJ}
	There exist bounded Hilbert module homomorphisms \(A:\mathcal  M^\perp\longrightarrow H^2\otimes D^\perp\)
	and \(J:H^2\otimes D^\perp\longrightarrow \mathcal M^\perp\)
	such that
	\begin{enumerate}[label=(\arabic*)]
	\item \(JA=S_\varphi\). \label{JA1}
	\item \(AJ=M_\varphi\otimes I_{D^\perp}\).\label{JA2}
	\item \(\delta_{D^\perp}A=\delta_{\mathcal M^\perp}\).\label{JA3}
	\item \(\delta_{\mathcal M^\perp}J=\delta_{D^\perp}\).\label{JA4}
	\end{enumerate}
\end{lem}

\begin{proof}
	Define
	\[
	A=\bigl(M_\varphi\otimes I_{\mathbb C^N}-\Psi (I_{H^2}\otimes P_D)\bigr)\big|_{\mathcal M^\perp},
	\qquad
	J=P_{\mathcal M^\perp}\big|_{H^2\otimes D^\perp}.
	\]
	By \eqref{rightequ-Psi},
	\begin{align*}
		(I_{H^2}\otimes P_D)\bigl(M_\varphi\otimes I_{\mathbb C^N}-\Psi (I_{H^2}\otimes P_D)\bigr) 
		=
		M_\varphi\otimes P_D -(M_\varphi \otimes I_D)(I_{H^2}\otimes P_D)
		= 0.
	\end{align*}
	Hence \(\operatorname{ran} A \subset H^2 \otimes D^\perp\).
	Moreover, it is easy to verify that \(A\) and \(J\) are bounded Hilbert module homomorphisms.\\
	\ref{JA1} Since $\operatorname{ran}\Psi\subset\mathcal M$, we have $P_{\mathcal M^\perp}\Psi=0$.
	Therefore, for $q\in\mathcal M^\perp$,
	\begin{align*}
		JAq
		=P_{\mathcal M^\perp}\bigl(M_\varphi\otimes I_{\mathbb C^N}-\Psi (I_{H^2}\otimes P_D)\bigr) q
		=S_\varphi q.
	\end{align*}
	\ref{JA2} For $h\in H^2\otimes D^\perp$, we have $(I_{H^2}\otimes P_D)h=0$.
	Hence, by \eqref{leftequ-Psi},
	\begin{align*}
		AJh
		=A P_{\mathcal M^\perp}h
		=\bigl(M_\varphi\otimes I_{\mathbb C^N}-\Psi (I_{H^2}\otimes P_D)\bigr) h
		=(M_\varphi\otimes I_{D^\perp})h.
	\end{align*}
	\ref{JA3} 
	For \(h \in H^2\otimes D\), since \((\Psi h)(\lambda_0) \in D\), we have
	\[
	(\Psi h)(\lambda_0)=P_D\bigl( (\Psi h)(\lambda_0)\bigr) = \bigl((I_{H^2 }\otimes P_D)\Psi h\bigr)(\lambda_0) \overset{(*)}{=} \bigl((M_{\varphi} \otimes I_D)h\bigr)(\lambda_0)\overset{(**)}{=}h(\lambda_0),
	\]
	where $(*)$ and $(**)$ follow from \eqref{rightequ-Psi} and \eqref{defvarphi}, respectively.
	Therefore, for $q\in\mathcal M^\perp$, 
	\[
	\delta_{D^\perp}Aq
	=\left(\bigl(M_\varphi\otimes I_{\mathbb C^N}-\Psi (I_{H^2}\otimes P_D)\bigr) q\right)(\lambda_0)
	=q(\lambda_0) - P_D\bigl(q(\lambda_0)\bigr)
	=\delta_{\mathcal M^\perp}q.
	\]
	\ref{JA4} For $h\in H^2\otimes D^\perp$, notice that \(h(\lambda_0)-(P_{\mathcal M^\perp }h)(\lambda_0)\in \mathcal M_{\lambda_0}=D\), we have
	\[
	\delta_{\mathcal M^\perp}Jh
	=P_{D^\perp}\bigl((P_{\mathcal M^\perp}h)(\lambda_0)\bigr)
	=P_{D^\perp}\bigl(h(\lambda_0)\bigr)
	=h(\lambda_0)=\delta_{D^\perp}h.
	\]
\end{proof}

Let \(\mathcal M\) be a closed submodule of
\(H^2\otimes\mathbb C^N\), 
Denote \(\mathbb S=(S_{z_1},S_{z_2},\ldots)\).
Then we have the following lemma.
\begin{lem}\label{g0mult}
	If \(g\in\mathrm{Mult}(H^2)\) satisfies
	\(g(0)=0\), then there exists a sequence of
	operators
	\[
	K_p:{\mathcal{M}^\perp}\otimes\Lambda^p
	\longrightarrow
	{\mathcal{M}^\perp}\otimes\Lambda^{p+1},
	~ p\geq0,
	\]
	such that $\sup\limits_{p\geq0}\|K_p\|<\infty$ and
	\[
	\partial_{\mathbb S,p+1}K_p
	+
	K_{p-1}\partial_{\mathbb S,p}
	=
	S_g\otimes I_{\Lambda^p},
	~ p\geq0,
	\]
	where \(K_{-1}=0\).
\end{lem}

\begin{proof}
	Define $M_Z:H^2 \otimes\ell^2\rightarrow H^2$ by
	\[
	M_Z\left( \sum_{j=1}^\infty f_j \otimes e_j \right)=\sum_{j=1}^\infty z_jf_j.
	\]
	Since \(g(0)=0\), we have \(M_g^*1=0\).
	Hence \(M_g^*(I-P_0) =M_g^*\).
	Consequently, 
	\[
	M_gM_g^*
	=M_g M_g^*(I-P_0)
	\leq
	\|M_g\|^2(I-P_0)
	=
	\|M_g\|^2M_ZM_Z^* .
	\]
	By the Leech--Douglas factorization theorem\cite[Theorem~1.10]{McCulloughTrent2012}, there exists a multiplier column \(\eta=(\eta_1,\eta_2,\cdots)^{\rm T}\) such that
	\(\|M_\eta\|\leq\|M_g\|\) and
	\begin{align}\label{g}
		g = \sum_{j=1}^\infty z_j\eta_j.
	\end{align}
	Let
		\[
		U:(H^2\otimes\mathbb C^N)\otimes\Lambda
		\longrightarrow (H^2\otimes\Lambda)\otimes\mathbb C^N,
		\qquad
		U\bigl((f\otimes v)\otimes\xi\bigr)=(f\otimes\xi)\otimes v.
		\]
	Then \(U\) is a unitary operator
	and for each $1\leq i\leq r$, 
	\[
	U\left(
	\sum_{i=1}^r(M_{\eta_i}\otimes I_{\mathbb C^N})\otimes E_i
	\right)U^*
	=
	\left(\sum_{i=1}^rM_{\eta_i}\otimes E_i\right)
	\otimes I_{\mathbb C^N}.
	\]
	Therefore,
	\begin{equation}\label{sup<infty}
		\begin{aligned}
			\left\|\sum_{i=1}^{r}S_{\eta_i}\otimes E_i\right\|
			&=\left\| (P_{\mathcal M^\perp} \otimes I_{\Lambda})\left(\sum_{i=1}^r (M_{\eta_i}\otimes I_{\mathbb C^N})\otimes E_i\right)(P_{\mathcal M^\perp} \otimes I_{\Lambda}) \right\|\\
			&\leq
			\left\| \sum_{i=1}^r (M_{\eta_i}\otimes I_{\mathbb C^N})\otimes E_i \right\|\\
			&=\left\|\left(\sum_{i=1}^r M_{\eta_i}\otimes E_i \right)\otimes I_{\mathbb C^N} \right\|\\
			&\overset{(*)}{=}
			\|M_{(\eta_1,\ldots,\eta_r)^{\rm T}}\|\\
			&\leq
			\|M_\eta\|
			\leq 
			\|M_g\|,
		\end{aligned}
	\end{equation}	
	where \((*)\) is given by Lemma~\ref{sup=} and Theorem \ref{thm:main}.
	Hence, by \eqref{sup<infty} and Lemma~\ref{bound and SOT}, we have \( (S_{\eta_1}, S_{\eta_2},\ldots)\) is row bounded and \(\sum\limits_{i=1}^\infty S_{\eta_i} \otimes E_i\) converges in the strong operator topology.
	Then define
	\[
	K_p
	=
	\left(\operatorname{SOT}\!-\!\lim_{r\to\infty}
	\sum_{i=1}^{r}
	S_{\eta_i}\otimes E_{i}\right)\Bigr|_{\mathcal M^\perp \otimes \Lambda^p} 
	= \operatorname{SOT}\!-\!\lim_{r\to\infty}
	\sum_{i=1}^{r}
	S_{\eta_i}\otimes E_{i,p+1}.
	\]
	For any \(\xi \in \mathcal M^\perp \otimes \Lambda^p\), by \eqref{sup<infty}, we have 
	\[
	\sup\limits_{p\geq0}\|K_p\|\leq \|M_g\|<\infty.
	\]
Moreover, for \(p \geq 0\), let
\[
K_p^{(r)} = \sum_{i=1}^r S_{\eta_i} \otimes E_{i,p+1} 
\quad \text{and} \quad 
\partial_{\mathbb S,p}^{(r)} = \sum_{i=1}^r S_{z_i} \otimes E_{i,p}^*.
\]
Then by the CAR relations, a direct computation yields
\begin{equation}\label{dK+kd}
\begin{aligned}
&\partial_{\mathbb S,p+1}^{(r)}K_p^{(r)}
+
K_{p-1}^{(r)}\partial_{\mathbb S,p}^{(r)}\\
=&
\sum_{i,j=1}^r
S_{z_i}S_{\eta_j}
\otimes
\left(E_{i,p+1}^*E_{j,p+1}+E_{j,p}E_{i,p}^*\right)
\\
=&
\sum_{i=1}^r S_{z_i}S_{\eta_i}\otimes I_{\Lambda^p},
\end{aligned}
\end{equation}
Since \(K_p^{(r)} \to K_p\) and \(\partial_{\mathbb S,p}^{(r)} \to \partial_{\mathbb S,p}\) in the strong operator topology as \(r\to\infty\), passing to the strong limit in \eqref{dK+kd} yields
\[
\partial_{\mathbb S,p+1}K_p
+
K_{p-1}\partial_{\mathbb S,p}
=
\left(\sum_{i=1}^\infty S_{z_i}S_{\eta_i}\right)\otimes I_{\Lambda^p}
\overset{(*)}{=}
S_g\otimes I_{\Lambda^p},
\]
where \((*)\) is given by \eqref{g}.
\end{proof}

\begin{lem}\label{lem:global}
	Let
	\[ 
	(X,\partial_X): \quad 0 \longleftarrow X_{-1} \xleftarrow{\partial_{X,0}} X_0 \xleftarrow{\partial_{X,1}} X_1 \xleftarrow{\partial_{X,2}} X_2 \xleftarrow{\partial_{X,3}} \cdots 
	\] 
	and 
	\[ 
	(Y,\partial_Y): \quad 0 \longleftarrow Y_{-1} \xleftarrow{\partial_{Y,0}} Y_0 \xleftarrow{\partial_{Y,1}} Y_1 \xleftarrow{\partial_{Y,2}} Y_2 \xleftarrow{\partial_{Y,3}} \cdots 
	\] 
	be Hilbert complexes,
	where \(\{X_p\}_{p \geq -1}\) and \(\{Y_p\}_{p\geq -1}\) are Hilbert spaces, \( \partial_{X,p}\in B(X_p,X_{p-1})\) and \(\partial_{Y,p}\in B(Y_p,Y_{p-1})\), \(p\geq0\).
	By convention, \(X_{-2}=Y_{-2}=\{0\}\).
	Assume that the following conditions hold:
	\begin{enumerate}
		\item There exists a uniformly bounded operator $R_{X,p}:X_p\to X_{p+1},\ p\geq -1$ such that
		\[
		\partial_{X,0} R_{X,-1} =I_{X_{-1}}, \quad \partial_{X,p+1} R_{X,p} + R_{X,p-1}\partial_{X,p} = I_{X_p}, \ p \geq 0.
		\]

		\item There exist uniformly bounded operators $A_p:Y_p\to X_p$, $J_p:X_p\to Y_p$ and $K_p:Y_p\to Y_{p+1}, \ p\geq -1$ such that
		\[
		A_{p-1}\partial_{Y,p} = \partial_{X,p} A_p,\quad J_{p-1}\partial_{X,p} = \partial_{Y,p} J_p, \ p \geq 0,\]
		and
		\[\partial_{Y,0}K_{-1} = J_{-1}A_{-1}-I_{Y_{-1}}, \quad \partial_{Y,p+1} K_p + K_{p-1}\partial_{Y,p} = J_pA_p - I_{Y_p}, \ p \geq 0.
		\]
	\end{enumerate}
	Then there exists a uniformly bounded operator $R_{Y,p}:Y_p\to Y_{p+1}, \ p \geq -1$ such that
		\[
		\partial_{Y,0}R_{Y,-1}=I_{Y_{-1}}, \quad \partial_{Y,p+1} R_{Y,p} + R_{Y,p-1}\partial_{Y,p} = I_{Y_p}, \ p \geq 0.
		\]
\end{lem}
\begin{proof}
	For each \(p\geq -1\), define
	\[
	R_{Y,p}
	=
	J_{p+1}R_{X,p}A_p-K_p
	:
	Y_p\longrightarrow Y_{p+1}.
	\]
	Then
	\[
	\sup_{p \geq -1} \|R_{Y,p}\| \leq \left(\sup_{p \geq 0} \|J_p\|\right) \left(\sup_{p \geq -1}  \|R_{X,p}\|\right)\left(\sup_{p \geq -1}  \|A_p\|\right) + \sup_{p \geq -1}  \|K_p\| < \infty.
	\]
	For \(p \geq 0\), by assumptions 1 and 2, we have
	\begin{align*}
		&\partial_{Y,p+1} R_{Y,p} + R_{Y,p-1}\partial_{Y,p}\\
		=& \partial_{Y,p+1}(J_{p+1}R_{X,p}A_p) - \partial_{Y,p+1} K_p + (J_{p}R_{X,p-1}A_{p-1})\partial_{Y,p} - K_{p-1}\partial_{Y,p} \\
		=&
		J_p\partial_{X,p+1}R_{X,p}A_p + J_{p}R_{X,p-1}\partial_{X,p} A_p-(\partial_{Y,p+1} K_p +K_{p-1}\partial_{Y,p} )\\
		=& J_p(\partial_{X,p+1}R_{X,p} + R_{X,p-1}\partial_{X,p})A_p - J_{p}A_{p}+ I_{Y_p}  \\
		=&J_pI_{X_p}A_p - J_{p}A_{p} + I_{Y_p} \\
		=&  I_{Y_p}.
	\end{align*}
	Moreover,
	\begin{align*}
	\partial_{Y,0}R_{Y,-1}&=\partial_{Y,0}(J_0R_{X,-1}A_{-1})-\partial_{Y,0}K_{-1}\\
	&=J_{-1}\partial_{X,0}R_{X,-1}A_{-1}-J_{-1}A_{-1}+I_{Y_{-1}}\\
	&=J_{-1}I_{X_{-1}}A_{-1}-J_{-1}A_{-1}+I_{Y_{-1}}\\
	&=I_{Y_{-1}}.
	\end{align*}
\end{proof}

\begin{lem}\label{dengshi}
	Let $\mathcal M$ be a submodule of $H^2\otimes\mathbb C^N$ and $\mathbb S=(S_{z_1},S_{z_2},\ldots)$.
	For $\lambda\in\mathbb B\setminus\{0\}$, let $\varphi_\lambda$ and $v_\lambda$
	be as in Lemma~\ref{vmuM}, and set \(\widetilde{\mathcal M}
	=(v_\lambda\otimes I_{\mathbb C^N})\mathcal M\).
	Then, for every $i\geq1$,
	\[
	(v_\lambda\otimes I_{\mathbb C^N})^*
	S_{z_i}^{\widetilde{\mathcal M}^{\perp}}
	(v_\lambda\otimes I_{\mathbb C^N})
	=(\varphi_\lambda(\mathbb S))_iP_{M^\perp}.
	\]
\end{lem}
\begin{proof}
	By Lemma~\ref{vmuM} \((iv)\), we have
	\begin{equation}\label{biaoda3}
		\begin{aligned}
			(v_\lambda\otimes I_{\mathbb C^N})^*
			S_{z_i}^{\widetilde{\mathcal M}^\perp}
			(v_\lambda\otimes I_{\mathbb C^N})
			=
			S_{z_i \circ \varphi_\lambda}^{\mathcal M^\perp}
			=
			P_{\mathcal M^\perp}
			(M_{z_i \circ \varphi_\lambda}\otimes I_{\mathbb C^N})P_{\mathcal M^\perp}.
		\end{aligned}
	\end{equation}
	By \eqref{def:varphi_lambda}, we have
	\begin{align}\label{bianxing}
	(1-\langle z,\lambda\rangle)\varphi_{\lambda,i}(z)
	=
	s_\lambda\sum_{j=1}^{\infty}
	a_{ij}(\lambda)(z_j-\lambda_j).
	\end{align}
	Therefore, for every $f\in\mathcal M^\perp$,
	\begin{equation}\label{biaoda2}
		\begin{aligned}
			(I-\langle \mathbb S,\lambda\rangle)
			P_{\mathcal M^\perp}
			(M_{z_i \circ \varphi_\lambda}\otimes I_{\mathbb C^N})f
			&=
			P_{\mathcal M^\perp}\left(
			I-\sum_{j=1}^{\infty}\overline{\lambda_j}
			(M_{z_j}\otimes I_{\mathbb C^N})
			\right)
			(M_{z_i \circ \varphi_\lambda}\otimes I_{\mathbb C^N})f\\
			&\overset{(*)}{=}
			P_{\mathcal M^ \perp} \left(s_\lambda\sum_{j=1}^{\infty}
			a_{ij}(\lambda)
			\bigl(M_{z_j}\otimes I_{\mathbb C^N}-\lambda_jI\bigr)\right)f\\
			&=
			s_\lambda\sum_{j=1}^{\infty}
			a_{ij}(\lambda)(S_{z_j}-\lambda_jI)f,
		\end{aligned}
	\end{equation}
	where \((*)\) is given by \eqref{bianxing}.
 	Since $\mathbb S$ is a row contraction and $\lambda\in\mathbb B$, we have \(\|\langle \mathbb S,\lambda\rangle\|\leq\|\lambda\|<1\), so $I-\langle\mathbb S,\lambda\rangle$ is invertible.
	Therefore, by \eqref{biaoda2}  and \eqref{varphi_lambda},
	\begin{align}\label{right=}
		P_{\mathcal M^\perp}
		(M_{z_i \circ \varphi_\lambda}\otimes I_{\mathbb C^N})f
		=(\varphi_\lambda(\mathbb S))_if.
	\end{align}
	Then by  \eqref{biaoda3} and \eqref{right=},
	\[
	(v_\lambda\otimes I_{\mathbb C^N})^*
	S_{z_i}^{\widetilde{\mathcal M}^\perp}
	(v_\lambda\otimes I_{\mathbb C^N})
	=(\varphi_\lambda(\mathbb S))_iP_{\mathcal M^\perp}.
	\]
\end{proof}

\begin{proof}[Proof of Theorem \ref{7.1}]
	Choose \(\lambda \in \mathbb B \setminus Z(\mathcal M)\). By \cite[Theorem 1.2]{curv}, it suffices to prove that $\lambda\in \mathbb B\setminus\sigma_e(\mathbb S)$ and $\operatorname{ind}(\mathbb S-\lambda)
	=
	N-\operatorname{fd}(\mathcal M).$

	Let $\varphi_\lambda$ and $v_\lambda$ be as in Lemma~\ref{vmuM}.
	Set \(\widetilde{\mathcal M}=(v_{\lambda}\otimes I_{\mathbb C^N})\mathcal M\) and \(\widetilde{\mathbb S}= (S_{z_1}^{\widetilde{\mathcal M}^\perp},S_{z_2}^{\widetilde{\mathcal M}^\perp},\cdots)\).
	By Lemma~\ref{vmuM}(ii), $\widetilde{\mathcal M}$ is a submodule
		of $H^2\otimes\mathbb C^N$.
	For any \(z \in \mathbb B\),
	\begin{align}\label{widetildeM=M}
	\widetilde{\mathcal M}_z =\left\{\widetilde{f}(z)\ | \ \widetilde{f} \in  \widetilde{\mathcal M}\right\}=\left\{ k_\lambda(z)f(\varphi_{\lambda}(z))\ |\ {f} \in  {\mathcal M}\right\} = \mathcal M_{\varphi_\lambda(z)}.
	\end{align}
	Thus 
	\begin{align}\label{fdwidetildeM=fdM}
	\operatorname{fd}\left(\widetilde{\mathcal M}\right) = \sup_{z \in \mathbb B}\dim 	\widetilde{\mathcal M}_z =  \sup_{z \in \mathbb B}\dim \mathcal M_{\varphi_\lambda(z)} =\sup_{w \in \mathbb B} \dim \mathcal M_{w} =\operatorname{fd} (\mathcal M).
	\end{align}
	Since \(\lambda \in \mathbb B \setminus Z(\mathcal M)\),  by \eqref{widetildeM=M} and \eqref{fdwidetildeM=fdM}, we have \(\dim\widetilde{\mathcal M}_0 = \dim {\mathcal M}_\lambda = \operatorname{fd}(\mathcal M) =\operatorname{fd}\left(\widetilde{\mathcal M}\right)\). 
	Thus $0 \in \mathbb B \setminus Z(\widetilde{\mathcal M})$.
	By Lemma~\ref{dengshi}, 
	\[
	(v_\lambda\otimes I_{\mathbb C^N})^*
	S_{z_i}^{\widetilde{\mathcal M}^\perp}
	(v_\lambda\otimes I_{\mathbb C^N})
	=(\varphi_\lambda(\mathbb S))_i \quad \text{on}\ \mathcal M^\perp.
	\]	
	Hence, by Lemma~\ref{T fredholm iff phi T},
	\(\widetilde{\mathbb S}\) is Fredholm
	if and only if \(\mathbb S-\lambda\) is Fredholm and 
	\(	\operatorname{ind}(\widetilde{\mathbb S})= \operatorname{ind}({\mathbb S}-\lambda)\).
	Thus, without loss of generality, we assume that \(\lambda=0\). We now take $\lambda_0=0$ in the constructions of \eqref{defvarphi} and Lemma 7.4-7.5. Thus $D=\mathcal{M}_0$ and $\varphi(0)=1$.
	
	Consider the augmented Koszul complexes
	\[
	K(\mathcal{M}^\perp,D^\perp):\quad
	0
	\longleftarrow
	D^\perp
	\xleftarrow{\delta_{\mathcal{M}^\perp}}
	\mathcal{M}^\perp
	\xleftarrow{\partial_{\mathbb S,1}}
	\Lambda^1(\mathcal{M}^\perp)
	\xleftarrow{\partial_{\mathbb S,2}}
	\Lambda^2(\mathcal{M}^\perp)
	\longleftarrow\cdots ,
	\]
	and
	\[
	K(H^2\otimes D^\perp,D^\perp):\quad
	0
	\longleftarrow
	D^\perp
	\xleftarrow{\delta_{D^\perp}}
	H^2\otimes D^\perp
	\xleftarrow{\partial_{\mathbb T,1}}
	\Lambda^1(H^2\otimes D^\perp)
	\xleftarrow{\partial_{\mathbb T,2}}
	\Lambda^2(H^2\otimes D^\perp)
	\longleftarrow\cdots.
	\]
Set \(\Lambda^{-1}(\mathcal{M}^\perp) = \Lambda^{-1}(H^2\otimes D^\perp) := D^\perp\).
	Choose \(A\) and \(J\) as in Lemma~\ref{AJ}.
	For each \(p\ge0\), let
	\(A_p=A\otimes I_{\Lambda^p} \in B(\Lambda^p(\mathcal M^\perp) , \Lambda^p(H^2\otimes D^\perp))\) 
	and
	\(J_p=J\otimes I_{\Lambda^p} \in B(\Lambda^p(H^2\otimes D^\perp),  \Lambda^p(\mathcal M^\perp))\). 
	Set \(A_{-1}=J_{-1}=I_{D^\perp}\in B(D^\perp, D^\perp)\). 
	Then for \(p\ge0\), by \ref{JA3} and \ref{JA4} of Lemma~\ref{AJ} and the fact that both \(A\) and \(J\) are Hilbert module homomorphisms, we have
	\begin{equation}\label{condition1}
		A_{p-1}\partial_{\mathbb S,p}=\partial_{\mathbb T,p}A_p,\qquad
	J_{p-1}\partial_{\mathbb T,p}=\partial_{\mathbb S,p}J_p.
	\end{equation}
	Applying Lemma~\ref{g0mult} to \(g=\varphi-1\), there exists a uniformly bounded operator \(K_p:\Lambda^p(\mathcal M^\perp) \to \Lambda^{p+1}(\mathcal M^\perp), \ p\geq -1\) such that
	\begin{equation}\label{condition2}
		\begin{gathered}
			\partial_{\mathbb S,0}K_{-1} =J_{-1}A_{-1}-I_{D^\perp},\\
			\partial_{\mathbb S,p+1}K_p+K_{p-1}\partial_{\mathbb S,p}
			= S_{\varphi}\otimes I_{\Lambda^p}-I_{\Lambda^p(\mathcal M^\perp)}\overset{(*)}{=}J_pA_p-I_{\Lambda^p(\mathcal M^\perp)},\ p\geq 0,
		\end{gathered}
	\end{equation}
	where \((*)\) is given by \ref{JA1} of Lemma~\ref{AJ}.
We claim that the augmented complex \(K(H^2\otimes D^\perp,D^\perp)\) satisfies the condition Proposition \ref{Iinvert iff uniform homotopy}(2) with $c=1$. In fact, for $\lambda \in D^\perp$, let $y_0=1\otimes \lambda$. Then
\begin{equation}\label{uni bound x0}
  \lambda=\delta_{D^\perp}y_0,\quad \|y_0\|\leq c\|\lambda\|.
\end{equation}
For $x_0\in \ker\delta_{D^\perp}$, since $x_0(0)=\delta_{D^\perp}x_0=0$ and $\sum\limits_{i=1}^\infty M_{z_i}M_{z_i}^*=I-P_0$, we have
\begin{equation}\label{construction y2}
x_0=\left(\left(\sum\limits_{i=1}^{\infty}M_{z_i}M_{z_i}^*\right)\otimes I_{D^\perp}\right)x_0=\sum_{i=1}^{\infty} \left(M_{z_i}\otimes I_{D^\perp}\right)\left(M_{z_i}^*\otimes I_{D^\perp}\right)x_0
\end{equation}
Hence by \eqref{construction y2},
\begin{equation}\label{norm y2}
\begin{aligned}
\left\|\sum\limits_{i=1}^{\infty}\left(\left(M_{z_i}^*\otimes I_{D^\perp}\right) x_0\right)\otimes e_i\right\|^2
&=\sum\limits_{i=1}^{\infty}\left\|\left(M_{z_i}^*\otimes I_{D^\perp}\right) x_0\right\|^2\\
&=\left\langle \sum_{i=1}^{\infty} \left(M_{z_i}\otimes I_{D^\perp}\right)\left(M_{z_i}^*\otimes I_{D^\perp}\right)x_0,x_0\right\rangle\\
&=\|x_0\|^2
\end{aligned}
\end{equation}
Let $y_1=\sum\limits_{i=1}^{\infty}\left(\left(M_{z_i}^*\otimes I_{D^\perp}\right) x_0\right)\otimes e_i\in \Lambda^1(H^2\otimes D^\perp)$. Then by \eqref{construction y2} and \eqref{norm y2},
\begin{equation}\label{y2 verify}
  x_0=\partial_{\mathbb{T},1}y_1,\quad \|y_1\|\leq c\|x_0\|.
\end{equation}
For $p\geq 1$ and $x_p\in \ker\partial_{\mathbb{T},p}$, by \eqref{dp+1y=x}, there exists $y_{p+1}\in \Lambda^{p+1}(H^2\otimes D^\perp)$ satisfying 
\begin{equation}\label{yge3 verify}
  x_p=\partial_{\mathbb{T},p+1}y_{p+1},\quad \|y_{p+1}\|\leq c\|x_p\|.
\end{equation}
By \eqref{uni bound x0}, \eqref{y2 verify} and \eqref{yge3 verify}, the claim is proved. Hence by Proposition \ref{Iinvert iff uniform homotopy}, \(K(H^2\otimes D^\perp,D^\perp)\) admits a uniformly bounded contraction \(R^{D^\perp}_p:\Lambda^p(H^2\otimes D^\perp) \to \Lambda^{p+1}(H^2\otimes D^\perp)\) satisfying
	\begin{equation}\label{condition3}
		\partial_{\mathbb T,p+1}R_p^{D^\perp}+R_{p-1}^{D^\perp}\partial_{\mathbb T,p}=I_{\Lambda^p(H^2\otimes D^\perp)}.
	\end{equation}
	By \eqref{condition1}, \eqref{condition2} and \eqref{condition3}, the hypotheses of Lemma~\ref{lem:global} are satisfied. Hence, there exists a uniformly bounded operator $R_p^{\mathcal M^\perp}: \Lambda^p(\mathcal M^\perp) \to \Lambda^{p+1}(\mathcal M^\perp), \ p\geq-1$ such that
	\[
	\partial_{\mathbb S,0}R_{-1}^{\mathcal M^\perp}=I_{D^\perp}, \quad
	\partial_{\mathbb S,p+1}R_p^{\mathcal M^\perp}+R_{p-1}^{\mathcal M^\perp}\partial_{\mathbb S,p}
	=I_{\Lambda^p(\mathcal M^\perp)},\ p\geq0.
	\]
	Therefore, by Proposition~\ref{Iinvert iff uniform homotopy}, the augmented complex \(K(\mathcal M^\perp, D^\perp)\) is exact,
	and \(3)\) in Definition~\ref{def:Fredholm} is satisfied.
	Moreover, since $H_0(\mathcal M^\perp)\cong D^\perp$, we have
	\[
	\dim H_0(\mathcal M^\perp)=\dim D^\perp=N-\operatorname{fd}(\mathcal M).
	\]
	Consequently, \(\mathbb S\) is Fredholm and \(\operatorname{ind}(\mathbb S)=N-\operatorname{fd}(\mathcal M)\).
\end{proof}
\begin{prop}\label{cor:spectrum}
Under the same assumptions as above, we have the following spectral characterizations:
\begin{enumerate}
    \item [(1)]\(\sigma_e(\mathbb{S}) \subset Z(\mathcal{M}) \cup \partial \mathbb{B}\).
    \item [(2)]If \(\operatorname{fd}(\mathcal{M}) = N\), then \(\sigma(\mathbb{S}) \cap \mathbb{B} = Z(\mathcal{M})\).
    \item [(3)]If \(\operatorname{fd}(\mathcal{M}) < N\), then \(\sigma(\mathbb{S}) = \overline{\mathbb{B}}\) and \(\partial \mathbb{B} \subset \sigma_e(\mathbb{S})\).
\end{enumerate}
\end{prop}
\begin{proof}
(1) Since \(\mathbb{S}\) is a row contraction, by Theorem \ref{bounded closed},
\begin{equation}\label{sie le B}
\sigma_e(\mathbb{S})\subset \sigma(\mathbb{S}) \subset \overline{\mathbb{B}}.
\end{equation}
By Theorem \ref{7.1}, 
\begin{equation}\label{sie cap B}
\sigma_e(\mathbb{S}) \cap \mathbb{B} \subset Z(\mathcal{M}).
\end{equation}
Hence by \eqref{sie le B} and \eqref{sie cap B}, \(\sigma_e(\mathbb{S}) \subset Z(\mathcal{M}) \cup \partial \mathbb{B}\).\\
(2) Assume \(\operatorname{fd}(\mathcal{M}) = N\). If \(\lambda \in \mathbb{B} \setminus Z(\mathcal{M})\), then \(\mathcal{M}_\lambda = \mathbb{C}^N\). By Theorem \ref{7.1}, \(\mathbb{S} - \lambda\) is Fredholm, and 
\begin{equation}\label{hp S-lambda eq 0}
H^p_{\mathbb{S} - \lambda} = 0, \quad \forall p \ge 0.
\end{equation}
By Theorem \ref{3equal}, we have \(\dim(\ker \partial_{\mathbb{S}-\lambda} / \operatorname{Im} \partial_{\mathbb{S}-\lambda}) < \infty\), which implies that \(\operatorname{Im} \partial_{\mathbb{S}-\lambda}\) is closed. Then by Lemma \ref{lem3.6} and \eqref{hp S-lambda eq 0}, there exists a constant \(c>0\) such that for every \(k \ge 0\) and \(x_k \in \ker \partial_{\mathbb{S}-\lambda, k}\), there is \(y_{k+1} \in \Lambda^{k+1}\) satisfying
\[
x_k = \partial_{\mathbb{S}-\lambda, k+1} y_{k+1} \quad \text{and} \quad \|y_{k+1}\| \le c\|x_k\|.
\]
Hence \(\lambda \notin \sigma(\mathbb{S})\). 

Conversely, if \(\lambda \in Z(\mathcal{M})\), take a nonzero \(w \in \mathbb{C}^N\) such that \(w \perp \mathcal{M}_\lambda\) and set \(x = K_\lambda \otimes w \in \mathcal{M}^\perp\). Since \(M_{z_i}^* K_\lambda = \bar{\lambda}_i K_\lambda\), we have \((S_{z_i} - \lambda_i)^* x = 0\). For any \(y = \sum\limits_{i=1}^\infty y_i \otimes e_i \in \Lambda^1(\mathcal{M}^\perp)\),
\[
\langle x, \partial_{\mathbb{S}-\lambda, 1} y \rangle = \left\langle x, \sum\limits_{i=1}^\infty (S_{z_i} - \lambda_i)y_i \right\rangle = \sum\limits_{i=1}^\infty \langle (S_{z_i} - \lambda_i)^* x, y_i \rangle = 0.
\]
Thus \(x \perp \operatorname{Im} \partial_{\mathbb{S}-\lambda, 1}\), which implies \(H_0(\mathbb{S}-\lambda) \neq 0\). Hence \(\mathbb{S}-\lambda\) is not invertible. Thus, \(\sigma(\mathbb{S}) \cap \mathbb{B} = Z(\mathcal{M})\).\\
(3) Assume \(\operatorname{fd}(\mathcal{M}) < N\). For every \(\lambda \in \mathbb{B} \setminus Z(\mathcal{M})\), the index 
\[
\operatorname{ind}(\mathbb{S} - \lambda) = N - \operatorname{fd}(\mathcal{M}) > 0,
\]
so \(\mathbb{S} - \lambda\) is not invertible. For \(\lambda \in Z(\mathcal{M})\), the same argument as in (2) shows that \(\mathbb{S} - \lambda\) is not invertible. Hence, \(\mathbb{B} \subset \sigma(\mathbb{S})\). Since \(\sigma(\mathbb{S}) \subset \overline{\mathbb{B}}\) and \(\sigma(\mathbb{S})\) is closed, we conclude 
\[
\sigma(\mathbb{S}) = \overline{\mathbb{B}}.
\]
Now suppose \(\lambda_0 \in \partial \mathbb{B}\) such that $\mathbb{S}-\lambda_0$ is Fredholm. By Corollary \ref{index continious}, there exists a neighborhood \(U\) of \(\lambda_0\) such that \(\operatorname{ind}(\mathbb{S} - \mu)\) is constant for all \(\mu \in U\). Choosing \(\mu \in U\setminus \overline{\mathbb{B}}\) gives \(\operatorname{ind}(\mathbb{S} - \mu) = 0\), while choosing \(\lambda \in U\cap (\mathbb{B} \setminus Z(\mathcal{M}))\) gives \(\operatorname{ind}(\mathbb{S} - \lambda) > 0\), a contradiction. Therefore, 
$$
\partial \mathbb{B} \subset \sigma_e(\mathbb{S}).
$$
\end{proof}

 \noindent{Penghui Wang, School of Mathematics, Shandong University, Jinan 250100, Shandong, P. R. China, Email:
phwang@sdu.edu.cn}

 \noindent{Fan Zhang, School of Mathematics, Shandong University, Jinan 250100, Shandong, P. R. China, Email:
202211823@sdu.edu.cn}

\noindent{Ruoyu Zhang, School of Mathematics, Shandong University, Jinan 250100, Shandong, P. R. China, Email:
ryzhangmath@163.com}

 \noindent{Zeyou Zhu, School of Mathematical Sciences, Fudan University, Shanghai 200433, P. R. China, Email:
zhuzeyou@fudan.edu.cn}

\end{document}